\documentclass[a4paper,12pt,twoside]{amsart}
\usepackage{vmargin}
\usepackage{amssymb}
\usepackage{amsmath}
\usepackage{mathrsfs}
\usepackage[dvips]{graphicx}
\usepackage{vmargin}
\usepackage{color}
\usepackage{amscd}
\DeclareGraphicsExtensions{.pdf, .png, .jpg, .mps}
\usepackage[colorlinks=false]{hyperref}
\definecolor{myc}{cmyk}{0.0009,0.8,0.8,0.00}

 \numberwithin{equation}{section}
\newcommand{\N}{{\mathbb N}}

\newcommand{\Z}{{\mathbb Z}}
\newcommand{\R}{{\mathbb R}}
\newcommand{\h}{{\mathcal H}}

\newcommand{\C}{{\mathbb C}}
\newcommand{\SL}{{\mathscr L}}
\def\poscal#1#2{\langle#1,#2\rangle} 
 
\def\normd#1{\Vert#1\Vert_{L^2(\R^2)}}

\def\normt#1{\Vert#1\Vert}
\def\p{\partial}
\def\Dk{\langle D_k\rangle^{-2}}

\newtheorem{Thm}{Theorem}[section]
\newtheorem{thm}[Thm]{Theorem}
\newtheorem{coro}[Thm]{Corollary}
\newtheorem{rem}[Thm]{Remark}
\newtheorem{lem}[Thm]{Lemma}
\newtheorem{prop}[Thm]{Proposition}
\newtheorem{defi}[Thm]{Definition}

\begin{document}

\title[Oseen vortices]{Pseudospectrum for Oseen vortices operators}
\author[W. Deng]{Wen Deng}
\begin{address}{Wen Deng, Institut de Math\'ematiques de Jussieu, Universit\'e Pierre-et-Marie-Curie (Paris 6), 4 place Jussieu, 75005 Paris, France.}
\email{wendeng@math.jussieu.fr}
\urladdr{http://www.math.jussieu.fr/~wendeng/}
\end{address}

\begin{abstract}
In this paper, we give resolvent estimates for the linearized operator of the Navier-Stokes equation in $\R^2$ around the Oseen vortices, in the fast rotating limit $\alpha\to+\infty$. 
\end{abstract}
\keywords{multiplier method, metric on the phase space}
\subjclass[2000]{ }
\date{\today}
\maketitle
{\footnotesize
\baselineskip=0.72
\normalbaselineskip
\tableofcontents}

%%%%
\section{Introduction}

\subsection{The origin of the problem}\label{sec.origin}

Consider the motion of a viscous incompressible fluid in the whole plane, which is described by the Navier-Stokes equation in $\R^2$.
In two dimensions where the vorticity is a scalar, it is more convenient to study the evolution of the vorticity which is given by
\begin{equation}\label{eq.NS}
\frac{\p\omega}{\p t}+v\cdot\nabla\omega=\nu\Delta\omega,\quad x\in\R^2,\quad t\geq0,
\end{equation}
where $\nu$ is the kinematic viscosity, $\omega(x,t)\in\R$ is the vorticity of the fluid, $v(x,t)\in\R^2$ is the divergence-free velocity field reconstructed from $\omega$ by the Biot-Savart law 
\begin{equation}\label{eq.BS}
v(x,t)=(K_{BS}\ast \omega) (x,t)=\frac{1}{2\pi}\int_{\R^2}\frac{(x-y)^\perp}{|x-y|^2}\omega(y,t)dy,
\end{equation}
where we denote $x^\perp=(-x_2,x_1)$ for $x=(x_1,x_2)\in\R^2$.
The equation \eqref{eq.NS} is globally well-posed in $L^1(\R^2)$ (\cite{BenArtzi}, \cite{Kato2}), i.e. for any initial data $\omega_0\in L^1(\R^2)$, \eqref{eq.NS} has a unique global solution $\omega\in C^0([0,+\infty);L^1(\R^2))$ such that $\omega(0)=\omega_0$. 
The {\it total circulation} of the velocity field
\begin{equation}
\int_{\R^2}\omega(x,t)dx%=\int_{\R^2}\omega_0(x)dx
=\lim_{R\to+\infty}\oint_{|x|=R}v(x,t)\cdot dl
\end{equation}
is a quantity conserved by the semi-flow defined by \eqref{eq.NS} in $L^1(\R^2)$.
It is well-known that the equation \eqref{eq.NS} has a family of explicit self-similar solutions, called {\it Oseen vortices}, which is given by
\begin{equation}
\omega(x,t)=\frac{\alpha}{\nu t}G\big(\frac{x}{\sqrt{\nu t}}\big),\quad v(x,t)=\frac{\alpha}{\sqrt{\nu t}}v^G\big(\frac{x}{\sqrt{\nu t}}\big), 
\end{equation}
where
\begin{equation}\label{eq.G}
G(x)=\frac{1}{4\pi}e^{-|x|^2/4},\qquad v^G(x)=\frac{1}{2\pi}\frac{x^\perp}{|x|^2}\big(1-e^{-|x|^2/4}\big),\quad x\in\R^2,
\end{equation}
and the parameter $\alpha\in\R$ is referred to as the {\it circulation Reynolds number}. In fact these solutions are trivial in the sense that $v(x,t)\cdot\nabla\omega(x,t)\equiv0$ so that \eqref{eq.NS} reduces to the linear heat equation, and the Oseen vortices are the only self-similar solutions to the Navier-Stokes equations in $\R^2$ whose vorticity is integrable.
Moreover, it is proved by T. Gallay and C.E. Wayne in \cite{GW1} that if the initial vorticity $\omega_0$ is in $L^1(\R^2)$, then the solution $\omega(x,t)$ of \eqref{eq.NS} satisfies
\begin{equation}
\lim_{t\to+\infty}\|\omega(\cdot,t)-\frac{\alpha}{\nu t}G\big(\frac{\cdot}{\sqrt{\nu t}}\big)\|_{L^1(\R^2)}=0,
\end{equation}
where $\alpha=\int_{\R^2}\omega_0(x)dx$.
In physical terms, this means that the Oseen vortices are globally stable for any value of the circulation Reynolds number $\alpha$. In contrast to many situations in hydrodynamics, such as the Poiseuille or the Taylor-Couette flows, increasing the Reynolds number does not produce any instability.

In order to investigate the stability of the Oseen vortices, we introduce the self-similar variables $\widetilde x=x/\sqrt{\nu t}$, $\widetilde t=\log(t/T)$ and we set
$$\omega(x,t)=\frac{1}{t}\widetilde\omega\big(\frac{x}{\sqrt{\nu t}},\log\frac{t}{T}\big),\quad
v(x,t)=\sqrt{\frac{\nu}{t}}\widetilde v\big(\frac{x}{\sqrt{\nu t}},\log\frac{t}{T}\big). $$
Then the rescaled system reads (replacing $\widetilde x$ by $x$, $\widetilde\omega$ by $\omega$ and so on)
\begin{equation}\label{eq.vorticity}
\frac{\p\omega}{\p t}+ v\cdot\nabla\omega=\Delta\omega+\frac{1}{2} x\cdot\nabla\omega+\omega,\quad  x\in\R^2,\  t\geq0,
\end{equation}
where $\omega(x,t)\in\R$ is the rescaled vorticity, $v(x,t)\in\R^2$ is the rescaled velocity field again given by the Biot-Savart law \eqref{eq.BS}.
Then for any $\alpha\in\R$, the Oseen vortex $\omega=\alpha G$ is a stationary solution of \eqref{eq.vorticity}.
%
%the equation \eqref{eq.vorticity} has a family of stationary solutions called {\it Oseen vortices}, of the form $\omega=\alpha G$, where
%\begin{equation}\label{eq.G}
%G(x)=\frac{1}{4\pi}e^{-|x|^2/4},\quad x\in\R^2,
%\end{equation}
%and $\alpha\in\R$ is a parameter called the {\it circulation Reynolds number}. The corresponding velocity field is given by
%\begin{equation} \label{eq.vG}
%v^G(x)=\frac{1}{2\pi}\frac{x^\perp}{|x|^2}\big(1-e^{-|x|^2/4}\big),\quad x\in\R^2.
%\end{equation}
Linearizing the equation \eqref{eq.vorticity} at $\alpha G$, we get a linear evolution equation
\begin{equation*}
\frac{\p \omega}{\p t}
=-({\mathcal L}+\alpha\Lambda)\omega,
\end{equation*}
where 
\begin{equation}
{\mathcal L}\omega=-\Delta\omega-\frac{1}{2}x\cdot \nabla \omega-\omega,\quad \Lambda\omega=v^G\cdot\omega+(K_{BS}\ast\omega)\cdot\nabla G.
\end{equation}
It turns out that the operator ${\mathcal L}$ is self-adjoint, non-negative on the weighted space $L^2(\R^2;G^{-1}dx)$ and $\Lambda$ is a relatively compact perturbation of ${\mathcal L}$, which is the sum of two skew-adjoint operators on $L^2(\R^2;G^{-1}dx)$. The spectrum of ${\mathcal L}+\alpha\Lambda$ is a sequence of eigenvalues by classical perturbation theory (\cite{Kato}).
Introducing the following subspaces of $Y=L^2(\R^2;G^{-1}dx)$:
\begin{align*}
Y_0&=\big\{\omega\in Y;\ \int_{\R^2}\omega(x) dx=0\big\}=\{G\}^\perp,\\
Y_1&=\big\{\omega\in Y_0;\ \int_{\R^2}x_j\omega(x) dx=0\text{ for }j=1,2\big\}=\{G;\p_1G;\p_2G\}^\perp,\\
Y_2&=\big\{\omega\in Y_1;\ \int_{\R^2}|x|^2\omega(x) dx=0\big\}=\{G;\p_1G;\p_2G;\Delta G\}^\perp,
\end{align*}
which are invariant spaces for ${\mathcal L}$ and $\Lambda$,
the following spectral bounds for ${\mathcal L}+\alpha\Lambda$ are proved in \cite{GW1}, 
\begin{align*}
&{\rm Spec}({\mathcal L}+\alpha\Lambda)\subset\big\{z\in\C;\ \text{Re}(z)\geq0\big\}\quad\text{in }Y,\\
&{\rm Spec}({\mathcal L}+\alpha\Lambda)\subset\big\{z\in\C;\ \text{Re}(z)\geq\frac{1}{2}\big\}\quad\text{in }Y_0,\\
&{\rm Spec}({\mathcal L}+\alpha\Lambda)\subset\big\{z\in\C;\ \text{Re}(z)\geq1\big\}\quad\text{in }Y_1,\\
&{\rm Spec}({\mathcal L}+\alpha\Lambda)\subset\big\{z\in\C;\ \text{Re}(z)>1\big\}\quad\text{in }Y_2,\text{ if }\alpha\neq0.
\end{align*}
These spectral bounds allow us to obtain estimates on the semigroup associated to ${\mathcal L}+\alpha\Lambda$, which can be used to show that Oseen vortex $\alpha G$ is a stable stationary solution of \eqref{eq.vorticity} for any $\alpha\in\R$. However, these bounds are not precise. 
The eigenvalues that do not move are those which correspond to eigenvectors in the kernel of $\Lambda$.
All eigenvalues of ${\mathcal L}+\alpha\Lambda$ which correspond to eigenvectors in the orthogonal complement of ${\rm ker}(\Lambda)$, have a real part that goes to $+\infty$ as $|\alpha|\to\infty$, observed numerically by A. Prochazka and D. Pullin \cite{ProPullin} and recently proved by Y. Maekawa \cite{Maekawa}.

In this paper, we are interested in pseudospectral properties of this linearized operator.
We conjugate the linear operators ${\mathcal L}$ and $\Lambda$ with $G^{1/2}$, then we obtain two operators on $L^2(\R^2;dx)$ 
\begin{align}
L\omega&=G^{-1/2}{\mathcal L}G^{1/2}\omega=-\Delta\omega+\frac{|x|^2}{16}\omega-\frac{1}{2}\omega,\\
\label{eq.M}M\omega&=G^{-1/2}\Lambda G^{1/2}\omega=v^G\cdot\nabla\omega-\frac{1}{2}G^{1/2}x\cdot\big(K_{BS}\ast(G^{1/2}\omega)\big).
\end{align}
Up to some numerical constants, $L$ is the two-dimensional harmonic oscillator, which is self-adjoint and non-negative on $L^2(\R^2;dx)$. On the other hand, both terms in $M$ are separately skew-adjoint on $L^2(\R^2;dx)$.
Letting
\begin{align}
\notag\h_\alpha\omega&=L\omega+\alpha M\omega,\qquad \omega\in L^2(\R^2;dx)\\
\label{eq.h}&=\big(-\Delta\omega+\frac{|x|^2}{16}\omega-\frac{1}{2}\omega\big)+\alpha\Bigl[v^G\cdot\nabla\omega-\frac{1}{2}G^{1/2}x\cdot\big(K_{BS}\ast(G^{1/2}\omega)\big)\Big],
\end{align}
our aim is to give estimates for the resolvent of the non-self-adjoint operator $\h_\alpha$ along the imaginary axis, in the fast rotating limit $\alpha\to+\infty$.

\subsection{About non-self-adjoint operators}

In many problems originated from mathematical physics, one encounters a linear evolution equation with a non-self-adjoint generator, of the form $H=A+iB$, where $A$ is self-adjoint, non-negative and $iB$ is skew-adjoint such that $A,B$ do not commute.  
$A$ is usually called the dissipative term and $iB$ the conservative term.
The conservative term can affect and sometimes enhance the dissipative effects or the regularizing properties of the whole system.
When a large skew-adjoint term $iB$ is present, the spectrum and the pseudospectrum of the whole operator $H$ may be strongly stabilized. In particular, the norm of the resolvent $\|(H-z)^{-1}\|$ may tend to 0 quickly.

In the paper \cite{GGN}, a one-dimensional analogue of $\h_\alpha$ is studied by I. Gallagher, T. Gallay and F. Nier
\begin{equation}
H_\epsilon=-\p_x^2+x^2+\frac{i}{\epsilon}f(x),\quad x\in\R,
\end{equation} 
where $\epsilon>0$ is a small parameter, $f\colon\R\to\R$ is a bounded smooth function.
Here the limit $\epsilon\to0$ corresponds to the fast rotating limit $\alpha\to+\infty$.
They studied the asymptotics of two quantities related to the spectral and pseudospectral properties in the limit $\epsilon\to0$. More precisely, they define $\Sigma(\epsilon)$ as the infimum of the real part of the spectrum of $H_\epsilon$ and 
$$\Psi(\epsilon)^{-1}=\sup_{\lambda\in\R}\|(H_\epsilon-i\lambda)^{-1}\|$$
as the supremum of the norm of the resolvent of $H_\epsilon$ along the imaginary axis. Under some appropriate conditions on $f$, both quantities $\Sigma(\epsilon)$, $\Psi(\epsilon)$ tend to infinity as $\epsilon\to0$ and lower bounds are given by using the so-called hypocoercive method. Furthermore, they focused on Morse functions of $C^3(\R;\R)$ which are bounded together with their derivatives up to the third order, and which behave like $|x|^{-k}$ as $|x|\to \infty$ (Hypothesis 1.6 in \cite{GGN}). For functions verifying these hypotheses, some precise and optimal estimates on $\Psi(\epsilon)$ are proved (Theorem 1.8 in \cite{GGN}): there exists $M\geq1$ such that for any $\epsilon\in(0,1]$,
$$\frac{1}{M\epsilon^{\nu}}\leq \Psi(\epsilon)\leq \frac{M}{\epsilon^{\nu}},\quad\text{with }\nu=\frac{2}{k+4}.$$
Their proof is based on the localization techniques and some semiclassical subelliptic estimates.

%$\Psi(\epsilon)$ is a quantity related to the pseudospectral properties of the operator.

%\medskip

In our recent work \cite{D}, a two-dimensional non-self-adjoint operator is considered
\begin{equation}\label{eq.La}
{\mathcal L}_\alpha=-\Delta+|x|^2+\alpha\sigma(|x|)\p_\theta,\quad x\in\R^2,
\end{equation}
where $\sigma(r)=r^{-2}(1-e^{-r^2})$, $\p_\theta=x_1\p_{2}-x_2\p_1$ and $\alpha$ is a positive parameter tending to infinity. 
Note that up to some numerical constants, the differential operator ${\mathcal L}_\alpha$ is equal to the operator $\h_\alpha$ given in \eqref{eq.h}, by neglecting the second member in the skew-adjoint part $\alpha M$, which is a non-local, lower-order term.
In that paper, we gave a complete study of the resolvent of ${\mathcal L}_\alpha$ along the imaginary axis in the limit $\alpha\to+\infty$ and proved an estimate of type (Theorem 2.2 in \cite{D})
\begin{equation}\label{eq.D}
\sup_{\lambda\in\R}\|({\mathcal L}_\alpha-i\lambda)^{-1}\|_{{\mathcal L}(\tilde L(\R^2))}\leq C\alpha^{-1/3},
\end{equation}
which is optimal.
The result is established by using a multiplier method, metrics on the phase space and localization techniques.

The present paper is devoted to proving resolvent estimates similar to \eqref{eq.D} for the whole linearized operator $\h_\alpha$ in \eqref{eq.h}.

\medskip

\noindent{\bf Acknowledgements.}
The author would like to thank Professors I. Gallagher and T. Gallay for kindly forwarding the question and generously providing helpful and detailed motivation arguments. In particular, the author is very grateful to the invitation of the summer school ``Spectral analysis of non-selfadjoint operators and applications", held in University Rennes I, June 2011, where the notes \cite{Gallay} were taken.

%%%%
\section{Statement of the result}

\subsection{The theorem}
Using the notations in Section \ref{sec.origin}, we consider the operator on $L^2(\R^2;dx)$
\begin{align}
\notag\h_\alpha\omega&=\underbrace{-\Delta\omega+\frac{|x|^2}{16}\omega-\frac{1}{2}\omega}_{\text{self-adjoint and non-negative on }L^2(\R^2)}\\
\label{eq.h1}&\qquad\qquad+\underbrace{\alpha v^G\cdot\nabla\omega-\frac{\alpha}{2}G^{1/2}x\cdot\big(K_{BS}\ast(G^{1/2}\omega)\big)}_{\text{skew-adjoint on }L^2(\R^2)},
\end{align}
where $G$, $v^G$ are given by \eqref{eq.G}, $K_{BS}$ is given in \eqref{eq.BS} and $\alpha\geq1$ is a large parameter. The real part of $\h_\alpha$ is the two-dimensional harmonic oscillator and the imaginary part of $\h_\alpha$ is the sum of a divergence-free vector field and a non-local integral operator, multiplied by the circulation Reynolds number $\alpha$.

The skew-adjoint part of $\h_\alpha$ vanishes on radial functions and in particular the function $e^{-|x|^2/8}$ is an eigenfunction of $\h_\alpha$ corresponding to the eigenvalue  0, for any $\alpha\in\R$, which implies that the ground state of the two-dimensional harmonic-oscillator does not move under the large skew-adjoint perturbation.
Moreover, one can also check that the skew-adjoint part of $\h_\alpha$ vanishes on the functions $x_1e^{-|x|^2/8}$, $x_2e^{-|x|^2/8}$. 
Thus we shall work in some subspaces of $L^2(\R^2;dx)$, defined below.

Using polar coordinates in $\R^2$, for $k_0\geq1$, we define the subspace of $L^2(\R^2;dx)$
\begin{equation}
X_{k_0}=\Big\{\omega\in L^2(\R^2;dx); \ \omega(r\cos\theta,r\sin\theta)=\sum_{|k|\geq k_0}\omega_k(r)e^{ik\theta}\Big\},
\end{equation}
which is a Hilbert space equipped with the norm $\normd{\cdot}$ and which is an invariant space for $\h_\alpha$.

\begin{defi}[Domain of $\h_\alpha$]\label{def.domain}
Let
$$D=\{\omega\in L^2(\R^2);\ \omega\in H^2(\R^2),\ |x|^2\omega\in L^2(\R^2)\}.$$
\end{defi}

Then $(\h_\alpha,D)$ is a closed operator on $L^2(\R^2)$.  Moreover, for any $k_0\geq1$, $\h_\alpha$ is a closed operator on $X_{k_0}$ with dense domain $D\cap X_{k_0}$ and its the numerical range defined by
$$\Theta(\h_\alpha;X_{k_0})=\{\poscal{\h_\alpha\omega}{\omega}_{L^2(\R^2)}\in\C;\ \omega\in D\cap X_{k_0},\ \normd{\omega}=1\}$$
is included in the set $\{z\in\C;\ {\rm Re}z\geq k_0/2 \}$, so that its spectrum is also contained in $\{z\in\C;\ {\rm Re}z\geq k_0/2 \}$.

Now let us state our main result.

\begin{thm}\label{thm.result}
There exist constants $C>0$, $k_0\geq3$, $\alpha_0\geq8\pi$ such that for all $\alpha\geq\alpha_0$, $\lambda\in\R$, for all $\omega\in C_0^\infty(\R^2)\cap X_{k_0}$, we have 
\begin{equation}\label{eq.result1}
\normd{(\h_\alpha-i\lambda)\omega}\geq C\alpha^{1/3}\normd{|D_\theta|^{1/3}\omega},
\end{equation}
where $|D_\theta|^{1/3}\omega=\sum_k|k|^{1/3}\omega_k(r)e^{ik\theta}$, for $\omega=\sum_k\omega_k(r)e^{ik\theta}$.
In particular, we have
\begin{equation}\label{eq.result2}
\|(\h_\alpha-i\lambda)^{-1}\|_{{\mathcal L}(X_{k_0})}\leq C^{-1}\alpha^{-1/3}k_0^{-1/3}.
\end{equation}
\end{thm}

The resolvent estimate \eqref{eq.result2} gives information about the pseudospectrum of the family of operators $\{\h_\alpha\}_{\alpha\geq1}$.

\begin{defi}
For the operators $\{\h_\alpha\}_{\alpha\geq1}$ on $X_{k_0}$, we define the pseudospectrum of $\{\h_\alpha\}_{\alpha\geq1}$ as the complement of the set of $z\in\C$ such that 
$$\exists N_0\in\N,\quad \sup_{\alpha\geq1}\|(\h_\alpha-z)^{-1}\|_{{\mathcal L}(X_{k_0})}\alpha^{-N_0}<+\infty.$$
\end{defi}

\begin{coro}
The pseudospectrum of $\{\h_\alpha\}_{\alpha\geq1}$ is included in the set 
$$\big\{z\in\C;\ {\rm Re}z\geq C\alpha^{1/3}k_0^{1/3}\big\}.$$
\end{coro}

Indeed, if ${\rm Re}z\leq0$, then for $\omega\in D\cap X_{k_0}$ with $\normd{\omega}=1$,
$$|\poscal{(\h_\alpha-z)\omega}{\omega}_{L^2(\R^2)}|\geq|{\rm Re}\poscal{\h_\alpha\omega}{\omega}_{L^2(\R^2)}-{\rm Re}z|\geq \frac{k_0}{2},$$
implying $\normt{(\h_\alpha-z)^{-1}}_{{\mathcal L}(X_{k_0})}\leq 2k_0^{-1}$.%, thus that $z$ is not in the pseudospectrum.

Let $\kappa\in(0,1)$. For $z=\mu+i\lambda$ with $0<\mu\leq\kappa C\alpha^{1/3}k_0^{1/3}$ and $\lambda\in\R$, we infer from the resolvent formula
$$(\h_\alpha-\mu-i\lambda)^{-1}-(\h_\alpha-i\lambda)^{-1}=\mu(\h_\alpha-i\lambda)^{-1}(\h_\alpha-\mu-i\lambda)^{-1}$$
and the resolvent estimate \eqref{eq.result2} that 
$$\normt{(\h_\alpha-\mu-i\lambda)^{-1}}_{{\mathcal L}(X_{k_0})}\leq \frac{\normt{(\h_\alpha-i\lambda)^{-1}}_{{\mathcal L}(X_{k_0})}}{1-\mu \normt{(\h_\alpha-i\lambda)^{-1}}_{{\mathcal L}(X_{k_0})}}\leq \frac{C^{-1}\alpha^{-1/3}k_0^{-1/3}}{1-\kappa}.$$
As a result, the set $\{z\in\C;\ {\rm Re}z< C\alpha^{1/3}k_0^{1/3}\}$ is included in the complement of the pseudospectrum of $\{\h_\alpha\}_{\alpha\geq1}$, so that the corollary is proved.

\subsection{Comments}
\subsubsection{The nonlocal term}
The term
$$\frac{\alpha}{2}G^{1/2}x\cdot\big(K_{BS}\ast(G^{1/2}\omega)\big)$$
is an integral operator which is non-local and skew-adjoint. This term should be carefully treated as it has a large coefficient $\alpha$.

\subsubsection{A weight}

We shall reduce the two-dimensional operator $\h_\alpha-i\lambda$ to a family of one-dimensional operators acting on the positive-half real line $\R_+$ by using polar coordinates and expanding the angular variable $\theta$ in Fourier series, indexed by the Fourier mode parameter $k\in\Z$. Then we transform the problem onto the whole real line $\R$ by making a change of variable $r=e^t$ and multiplying by a weight $e^{2t}$. After these transformations, the properties of self-adjointness and skew-adjointness are preserved (see Section \ref{sec.reduction}), and the non-local term turns out to be a skew-adjoint pseudodifferential operator with ${\mathcal L}(L^2(\R;dt))$-norm bounded above by $\alpha |k|^{-1}$.
The discussion is divided into different cases according to a change-of-sign situation. 

\subsubsection{Multiplier method}

The proof relies on a classical multiplier method.
For the non-trivial cases where the change-of-sign takes place (see Section \ref{sec.nontrivialcase}, \ref{sec.nontrivial.c}), we shall construct a multiplier bounded on $L^2(\R;dt)$, which is a pseudodifferential operator associated to a H\"ormander-type metric.
The non-local term will be treated as a perturbation and will be absorbed by the main term letting $|k|\geq k_0$, with $k_0$ a constant independent of the circulation parameter $\alpha$.

\subsubsection{The value of $k_0$}

Given $\epsilon_0,\epsilon_1\in(0,1)$, we shall discuss 4 different cases given in \eqref{4cases}. Then $k_0$ can be expressed as a function of $(\epsilon_0,\epsilon_1)$. For example, if we take $\epsilon_0\simeq0.462$, $\epsilon_1\simeq0.426$, then Theorem \ref{thm.result} holds with $k_0=84$, see \cite{D3}.
(In fact, we can obtain $k_0=51$ if we do some improvements.)

%%%%%%%%
%%%%%%%%
\section{The proof}

\subsection{First reductions}\label{sec.reduction}
The operator $\h_\alpha$ in \eqref{eq.h1} is invariant under rotations with respect to the origin in $\R^2$. We can reduce the problem to a family of one-dimensional operators by using polar coordinates and expanding the angular variable $\theta$ in Fourier series.

\subsubsection{Polar coordinates}
We can write for $\omega\in L^2(\R^2)$ and $v=K_{BS}\ast \omega$ given by \eqref{eq.BS} as
\begin{align*}
\omega(r\cos\theta,r\sin\theta)&=\sum_{k\in\Z}\omega_k(r)e^{ik\theta},\\
v(r\cos\theta,r\sin\theta)&=\sum_{k\in\Z}\Big(\frac{u_k(r)}{r}{\bf e_r}+\frac{w_k(r)}{r}{\bf e_\theta}\Big)e^{ik\theta},
\end{align*}
where ${\bf e_r}=(\cos\theta,\sin\theta)$ and ${\bf e_\theta}=(-\sin\theta,\cos\theta)$. The relations $\p_1v_1+\p_2v_2=0$, $\p_1v_2-\p_2v_1=\omega$ become
$$u_k'+\frac{ik}{r}w_k=0,\quad w_k'-\frac{ik}{r}u_k=r\omega_k,$$
so that $-\Delta_ku_k=ik\omega_k$, where
$$-\Delta_k=-\p_r^2-\frac{1}{r}\p_r+\frac{k^2}{r^2}.$$
If $k\neq0$, the Poisson equation $-\Delta_k\Omega=f$ has the explicit solution $\Omega={\mathcal K}_k[f]$, where
\begin{equation}\label{eq.K}
{\mathcal K}_k[f](r)=\frac{1}{2|k|}\int_0^{+\infty}\Big(\big(\frac{r}{s}\big)^{|k|}H(r)H(s-r)+\big(\frac{s}{r}\big)^{|k|}H(s)H(r-s)\Big)f(s)sds,
\end{equation}
where $H(r)$ is the Heaviside function.
We thus have
$$u_k=ik{\mathcal K}_k[\omega_k]\quad\text{and}\quad w_k=-r{\mathcal K}_k[\omega_k]'$$ 
if $k\neq0$. For $k=0$, we find $u_0=0$ and $w_0'=r\omega_0$, hence $w_0(r)=\int_0^rs\omega(s)ds$.

By using the following notations:
\begin{equation}\label{eq.sigma}
\sigma(r)=\frac{1-e^{-r^2/4}}{r^2/4},\quad g(r)=e^{-r^2/8},\quad\text{for } r>0,
\end{equation}
and observing that $v^G=\frac{1}{8\pi}r\sigma(r) {\bf e_\theta}$, we rewrite the skew-adjoint part of $\h_\alpha$ in polar coordinates as 
$$\alpha v^G\cdot \nabla\omega=\sum_{k\neq0}\frac{i\alpha k}{8\pi}\sigma(r)\omega_k(r)e^{ik\theta},$$
$$\frac{\alpha}{2}G^{1/2}x\cdot\big(K_{BS}\ast(G^{1/2}\omega)\big)=\sum_{k\neq0}\frac{i\alpha k}{8\pi}g(r){\mathcal K}_k[g\omega_k](r)e^{ik\theta}.$$
Thus we find that for $\h_\alpha$ given by \eqref{eq.h1}, $\lambda\in\R$ and for $\omega=\sum_{k\in\Z^\ast}\omega_k(r)e^{ik\theta}$,
\begin{equation}\label{eq.h.alpha}
\big((\h_\alpha-i\lambda)\omega\big)(r\cos\theta,r\sin\theta)=\sum_{k\in\Z^\ast}(\h_{\alpha,k,\lambda}\omega_k)(r)e^{ik\theta},
\end{equation}
where $\h_{\alpha,k,\lambda}$ acts on $L^2(\R_+;rdr)$ and is given by
\begin{equation}\label{eq.hk1}
\h_{\alpha,k,\lambda}v=-\partial_r^2v-\frac{1}{r}\partial_rv+\frac{k^2}{r^2}v+\frac{r^2}{16}v-\frac{1}{2}v+\frac{ik\alpha}{8\pi}\Big(\sigma(r)v-g{\mathcal K}_k[gv]\Big)-i\lambda v.
\end{equation}
Introducing two new notations
\begin{equation}\label{eq.beta}
\beta_k=\frac{\alpha k}{8\pi},\quad \lambda=\beta_k\nu_k,\quad \nu_k\in\R,
\end{equation}
we are led to study the resolvent of the one-dimensional operator $\h_{\alpha,k,\lambda}$ on $L^2(\R_+;rdr)$ for $|\beta_k|\to+\infty$, where (we omit the indices $\alpha,\lambda$ in $\h_{\alpha,k,\lambda}$)
\begin{align}
\notag\h_{k}v&=\underbrace{-\partial_r^2v-\frac{1}{r}\partial_rv+\frac{k^2}{r^2}v+\frac{r^2}{16}v-\frac{1}{2}v}_{\text{self-adjoint and non-negative on }L^2(\R_+;rdr)}\\
\label{eq.hk}&\qquad\qquad\underbrace{+i\beta_k\big(\sigma(r)-\nu_k\big)v-i\beta_kg{\mathcal K}_k[gv]}_{\text{skew-adjoint on }L^2(\R_+;rdr)}.
\end{align}
Note that the non-local term is transformed to $i\beta_kg{\mathcal K}_kg$ with ${\mathcal K}_k$ given by \eqref{eq.K}. Moreover, $C_0^\infty((0,+\infty))$ is a core for the closed operator $\h_k$ with domain
$$D(\h_k)=\big\{v\in L^2(\R_+;rdr);\ \p_r^2v,\ \frac{1}{r}\p_rv,\ \frac{1}{r^2}v,\ r^2v\in L^2(\R_+;rdr)\big\},\quad\text{if }|k|\geq2,$$
$$\text{and}\quad D(\h_k)=\big\{v\in L^2(\R_+;rdr);\ \p_r^2v,\ \p_r\big(\frac{v}{r}\big),\ r^2v\in L^2(\R_+;rdr)\big\},\quad \text{if }|k|=1.$$

\subsubsection{Change of variables}
We wish to transform the operator $\h_k$ in \eqref{eq.hk} acting on the positive half-line into an operator acting on the whole real line, by making the change of variables $r=e^t$.
A simple but key observation is
\begin{lem}\label{lem.change}
For $v\in L^2(\R_+;rdr)$, define $u(t)=v(e^t)$. Then
\begin{equation}\label{eq.change1}
\|e^tu\|_{L^2(\R;dt)}^2=\int_\R|e^tu(t)|^2dt=\int_0^{+\infty} |v(r)|^2rdr=\|v\|_{L^2(\R_+;rdr)}^2.
\end{equation}
Moreover, for $v\in C_0^\infty((0,+\infty))$, multiplying $(\h_kv)(e^t)$ by the weight $e^{2t}$, we have 
\begin{equation}\label{eq.change2}
e^{2t}(\h_kv)(e^t)=(\widetilde\SL_ku)(t),\quad \text{for }u(t)=v(e^t),
\end{equation}
where
\begin{align}
\notag \widetilde\SL_k=&-\partial_t^2+k^2+\frac{1}{16}e^{4t}-\frac{1}{2}e^{2t}\\
\label{eq.tilde.lk}&\qquad+i\beta_ke^{2t}\big(\sigma(e^t)-\nu_k\big)-i\beta_ke^{2t}g(e^t) (k^2+D_t^2)^{-1} e^{2t}g(e^t).
\end{align}
\end{lem}

\begin{proof}
Indeed, we have $$r^2(\p_r^2+r^{-1}\p_r)=(r\p_r)^2=\p_t^2\quad\text{for }r=e^t.$$
On the other hand, by the definition \eqref{eq.K} of ${\mathcal K}_k$, we have for $v\in C_0^\infty((0,+\infty))$,
\begin{align*}
&e^{2t}\big(g{\mathcal K}_k[gv]\big)(e^t)=e^{2t}g(e^t){\mathcal K}_k[gv](e^t)\\
&=\frac{1}{2|k|}\int_0^{+\infty} e^{2t}g(e^t)\Bigl[\Big(\frac{e^t}{s}\Big)^{|k|}H(e^t)H(s-e^t)+\Big(\frac{s}{e^t}\Big)^{|k|}H(s)H(e^t-s)\Bigr]g(s)v(s)sds\\
&=\frac{1}{2|k|}\int_\R e^{2t}g(e^t)\Bigl[\Big(\frac{e^t}{e^s}\Big)^{|k|}H(e^s-e^t)+\Big(\frac{e^s}{e^t}\Big)^{|k|}H(e^t-e^s)\Bigr]g(e^s)v(e^s)e^{2s}ds\\
&=\frac{1}{2|k|}\int_\R e^{2t}g(e^t)\Bigl[e^{-|k|(s-t)}H(s-t)+e^{-|k|(t-s)}H(t-s)\Bigr]e^{2s}g(e^s)v(e^s)ds\\
&=\frac{1}{2|k|}\int_\R e^{2t}g(e^t)e^{-|k||t-s|}e^{2s}g(e^s)u(s)ds.
\end{align*}
For $k\neq0$, we have
$$ \frac{1}{2|k|}\int_\R e^{-|k||t|}e^{i t\tau}dt=\frac{1}{k^2+\tau^2},$$
(see Lemma \ref{lem.Dk2}) so that the non-local term $g{\mathcal K}_kg$ becomes
$$e^{2t}\big(g{\mathcal K}_k[gv]\big)(e^t)=\big(e^{2t}g(e^t) (k^2+D_t^2)^{-1} e^{2t}g(e^t)u\big)(t),\quad\text{for } u(t)=v(e^t),$$
which is a self-adjoint, positive (non-local) pseudodifferential operator on $L^2(\R;dt)$. 
The proof of the lemma is complete.
\end{proof}

When $\widetilde\SL_k$ given by \eqref{eq.tilde.lk} is viewed as an operator on $L^2(\R;dt)$, we see that
\begin{align}
\notag \widetilde\SL_k=&\underbrace{-\partial_t^2+k^2+\frac{1}{16}e^{4t}-\frac{1}{2}e^{2t}}_{\text{self-adjoint and non-negative on }L^2(\R;dt)}\\
\notag&\qquad\qquad+\underbrace{i\beta_ke^{2t}\big(\sigma(e^t)-\nu_k\big)}_{\text{skew-adjoint on }L^2(\R;dt)}-\underbrace{i\beta_ke^{2t}g(e^t) (k^2+D_t^2)^{-1} e^{2t}g(e^t)}_{\text{skew-adjoint on }L^2(\R;dt)}.
\end{align}
After the change of variables $r=e^t$ and the multiplication by the weight $e^{2t}$, the self-adjoint (resp. skew-adjoint) part of $\h_k$ in \eqref{eq.hk} does not lose its self-adjointness (resp. skew-adjointness), and in particular, the non-local term $i\beta_kg{\mathcal K}_kg$ stays skew-adjoint.
Moreover, the power 2 in the weight is the only power to keep these properties unchanged.

In view of \eqref{eq.change1} and \eqref{eq.change2} in Lemma \ref{lem.change}, the problem is reduced to prove estimates for the operator $\widetilde\SL_k$ in \eqref{eq.tilde.lk} of type
\begin{equation}\label{eq.u}
\|e^{-t}\widetilde\SL_ku\|_{L^2(\R;dt)}\geq C|\beta_k|^a\|e^tu\|_{L^2(\R;dt)}
\end{equation}
for some $a>0$, which correspond to the estimates for the operator $\h_k$ given in \eqref{eq.hk}
\begin{equation}\label{eq.v}
\|\h_kv\|_{L^2(\R_+;rdr)}\geq C|\beta_k|^a\|v\|_{L^2(\R_+;rdr)},
\end{equation}
where $u(t)=v(e^t)$, since we have exactly
$$\|e^{-t}\widetilde\SL_ku\|_{L^2(\R;dt)}=\|\h_kv\|_{L^2(\R_+;rdr)},\quad \|e^tu\|_{L^2(\R;dt)}=\|v\|_{L^2(\R_+;rdr)}.$$
Furthermore, we need only to prove estimates \eqref{eq.u} for $u\in C_0^\infty(\R)$, since it is enough to get \eqref{eq.v} for $v\in C_0^\infty((0,+\infty))$.

As in \cite{D}, we divide our discussion into different cases, according to the change-of-sign situation of $\sigma(e^t)-\nu_k$, where the function $\sigma$ is given in \eqref{eq.sigma}. Note that $\sigma(e^t)$ is a decreasing function of the variable $t$ and has range $(0,1)$. When $\sigma(e^t)-\nu_k$ does not change sign, it is easy to deal with by using the multipliers ${\rm Id}$, $\pm i{\rm Id}$ (see Section \ref{sec.easycase}). If $\sigma(e^t)-\nu_k$ changes sign at one point, it is more complicated (see Section \ref{sec.nontrivialcase}, \ref{sec.nontrivial.c}). In this case, we will construct a multiplier well-adapted to this change-of-sign situation, which is a pseudodifferential operator depending on a H\"ormander metric on the phase space.
Compared with the method in \cite{D}, the multiplier that we shall construct is a global one, because of the existence of the non-local term, which possesses a large coefficient and would produce a commutator of size $|\beta_k|$ if we just used a partition of unity on $\R_t$ as done in \cite{D}.

\subsubsection{Notations}

In Section \ref{sec.easycase}, \ref{sec.nontrivialcase} and \ref{sec.nontrivial.c}, we shall always assume that $k\geq1$ hence $\beta_k>0$, and we denote by $\|\cdot\|$, $\poscal{\cdot}{\cdot}$ the $L^2(\R;dt)$-norm, inner-product respectively. 
We shall also be able to neglect the term $-\frac{1}{2}e^{2t}$ in the real part of $\widetilde\SL_k$ and by introducing two notations,
\begin{equation}\label{eq.Dk}
\Dk=(D_t^2+k^2)^{-1},\quad \gamma(t)=e^{2t}g(e^t)=e^{2t}e^{-e^{2t}/8},
\end{equation}
we shall study
\begin{align}\label{eq.lk}
\SL_k&=D_t^2+k^2+\frac{1}{16}e^{4t}+i\beta_ke^{2t}\big(\sigma(e^t)-\nu_k\big)-i\beta_k\gamma(t)\Dk\gamma(t).
\end{align}
In fact, as soon as we prove \eqref{eq.u} for $\SL_k$ in \eqref{eq.lk} with $a>0$, we have for the operator $\widetilde\SL_k$ given in \eqref{eq.tilde.lk}
\begin{align*}
\|e^{-t}\widetilde\SL_ku\|_{L^2(\R;dt)}&=\|e^{-t}\big(\SL_k-\frac{1}{2}e^{2t}\big)u\|_{L^2(\R;dt)}\\
&\geq C\beta_k^a\|e^tu\|_{L^2(\R;dt)}-\frac{1}{2}\|e^tu\|_{L^2(\R;dt)},
\end{align*}
so that it suffices to let $\alpha$ large enough since $k\geq1$, $\beta_k\geq\alpha/8\pi$.

\medskip

We present in Appendix \ref{sec.ineq} some inequalities concerning the functions $\sigma$ and $g$ given in \eqref{eq.sigma} that will be used in the proof.
We have for all $u\in C_0^\infty(\R)$,
\begin{equation}\label{eq.real}
{\rm Re}\poscal{\SL_ku}{u}_{L^2(\R;dt)}=\poscal{\big(D_t^2+k^2+\frac{1}{16}e^{4t}\big)u}{u}_{L^2(\R;dt)},
\end{equation}
\begin{equation}\label{eq.gkg}
0\leq\poscal{\gamma\Dk\gamma u}{u}_{L^2(\R;dt)}\leq k^{-2}\normt{\gamma u}_{L^2(\R;dt)}^2,\quad\text{\small by Lemma \ref{lem.Dk1},}
\end{equation}
where $\SL_k$ is given in \eqref{eq.lk} and $\gamma,\Dk$ are given in \eqref{eq.Dk}.

\subsection{Easy cases}\label{sec.easycase}
In this section, we study the cases where $\sigma(e^t)-\nu_k$ does not change sign, that is $\nu_k\geq1$ or $\nu_k\leq0$. %

\begin{lem}\label{lem.easy1}
Suppose $\nu_k\geq1$. There exists $C>0$ such that for all $k\geq1$, $\alpha\geq8\pi$ and for $u\in C_0^\infty(\R)$,
\begin{equation}\label{est.easy1}
\normt{e^{-t}\SL_ku}\geq C\beta_k^{1/2}\normt{e^tu},
\end{equation}
where $\SL_k$ is given in \eqref{eq.lk} and $\beta_k$ is given in \eqref{eq.beta}.
\end{lem}
\begin{proof}
If $\nu_k\geq1$, then $\sigma(e^t)-\nu_k$ is non-positive. Using the multiplier $-i{\rm Id}$ and by \eqref{eq.gkg}, we have 
\begin{align}
\notag{\rm Re}\poscal{\SL_ku}{-iu}&=\beta_k\poscal{e^{2t}\big(\nu_k-\sigma(e^t)\big)u}{u}+\beta_k\poscal{\gamma\Dk\gamma u}{u}\\
\label{eq.easy1}&\geq\beta_k\poscal{e^{2t}\big(1-\sigma(e^t)\big)u}{u}.
\end{align}
Adding \eqref{eq.real}, \eqref{eq.easy1} together, we obtain
\begin{equation*}
{\rm Re}\poscal{\SL_ku}{(1-i)u}\geq \poscal{\Big(k^2e^{-2t}+\beta_k\big(1-\sigma(e^t)\big)\Big)e^{2t}u}{u}.
\end{equation*}
Then using the second inequality in \eqref{sigma.2}, we get
\begin{equation*}
{\rm Re}\poscal{e^{-t}\SL_ku}{e^t(1-i)u}\geq C\beta_k^{1/2}\poscal{e^{2t}u}{u}=C\beta_k^{1/2}\normt{e^tu}^2.
\end{equation*}
By Cauchy-Schwarz inequality, the estimate \eqref{est.easy1} is proved.
\end{proof}

\begin{lem}\label{lem.easy2}
Suppose $\nu_k\leq0$. There exists $C>0$ such that for all $k\geq2$, $\alpha\geq8\pi$ and for $u\in C_0^\infty(\R)$,
\begin{equation}\label{est.easy2}
\normt{e^{-t}\SL_ku}\geq C\beta_k^{1/2}\normt{e^tu},
\end{equation}
where $\SL_k$ is given in \eqref{eq.lk} and $\beta_k$ in \eqref{eq.beta}.
\end{lem}
\begin{proof}
If $\nu_k\leq0$, then $\sigma(e^t)-\nu_k$ is non-negative. Using the multiplier $i{\rm Id}$ and by \eqref{eq.gkg}, we have 
\begin{align*}
{\rm Re}\poscal{\SL_ku}{iu}&=\beta_k\poscal{e^{2t}\big(\sigma(e^t)-\nu_k\big)u}{u}-\beta_k\poscal{\gamma\Dk\gamma u}{u}\\
&\geq\beta_k\Big(\poscal{e^{2t}\sigma(e^t)u}{u}-k^{-2}\normt{e^{2t}g(e^t)u}^2\Big).
\end{align*}
Using \eqref{sigma.1}, we get for $k\geq2$,
\begin{equation}\label{eq.easy2}
{\rm Re}\poscal{\SL_ku}{iu}\geq \big(1-(4\delta)^{-1}\big)\beta_k\poscal{e^{2t}\sigma(e^t)u}{u},
\end{equation}
with $1-(4\delta)^{-1}>0$.
Adding \eqref{eq.real}, \eqref{eq.easy2} together we obtain
\begin{equation*}
{\rm Re}\poscal{\SL_ku}{(1+i)u}\geq \poscal{\Big(\frac{1}{16}e^{2t}+\big(1-(4\delta)^{-1}\big)\beta_k\sigma(e^t)\Big)e^{2t}u}{u},\quad k\geq2.
\end{equation*}
Using the first inequality in \eqref{sigma.2}, we get 
\begin{equation*}
{\rm Re}\poscal{e^{-t}\SL_ku}{e^t(1+i)u}\geq C\beta_k^{1/2}\poscal{e^{2t}u}{u}=C\beta_k^{1/2}\normt{e^tu}^2,\quad k\geq2.
\end{equation*}
By Cauchy-Schwarz inequality, the estimate \eqref{est.easy2} is proved.
\end{proof}

\begin{rem}\rm
When $k=1$ and $\nu_k=0$, the imaginary part of $\h_1$ vanishes on the function $v(r)=rg(r)\in D(\h_1)$, i.e. we have $g{\mathcal K}_1[gv]=\sigma v$. Consequently, when $\nu_k=0$, the imaginary part of $\SL_1$ vanishes on the function $u(t)=e^tg(e^t)$.
%{\color{red}Question: we can find some constant $C>0$ such that $\normt{e^{-t}\SL_1u}\geq C\beta_1^{1/2}\normt{e^tu}$ holds for $u$ verifying $\poscal{u}{e^{t}g(e^t)}=0$.}
\end{rem}

\subsection{Nontrivial cases}\label{sec.nontrivialcase}

We turn to study the cases where the change-of-sign of $\sigma(e^t)-\nu_k$ takes place, that is $\nu_k\in (0,1)$. We have thus $\nu_k=\sigma(e^{t_k})$ for some $t_k\in\R$. Then the operator $\SL_k$ can be written as 
\begin{align}\label{eq.lk2}
\SL_k&=D_t^2+k^2+\frac{1}{16}e^{4t}+i\beta_ke^{2t}\big(\sigma(e^t)-\sigma(e^{t_k})\big)-i\beta_k\gamma(t)\Dk\gamma(t).
\end{align}
Suppose $\epsilon_0,\epsilon_1\in(0,1)$. We discuss four cases according to the behavior of the function $\sigma$ near the point $e^{t_k}$: 
\begin{equation}\label{4cases}
e^{t_k}>\epsilon_0^{-1}\quad \text{or}\quad e^{t_k}\in\big[\epsilon_1,\epsilon_0^{-1}\big]\quad \text{or}\quad  e^{t_k}\in(\beta_k^{-1/4},\epsilon_1)\quad \text{or}\quad  e^{t_k}\leq \beta_k^{-1/4}.
\end{equation}
Before going through the proofs for each case, let us first choose some functions that will be used to construct the multipliers.
Suppose that $c_0\in(0,1)$ is the constant chosen in Proposition \ref{prop.sigma}.
Let $\chi_j\in C^\infty(\R;[0,1])$, $j=0,1,-1$, satisfying that
\begin{equation}\label{eq.chi}
\begin{cases}
\displaystyle\text{supp}\chi_0\subset[-c_0,c_0],\\ 
\displaystyle\chi_+=1\text{ on }[c_0,+\infty),\qquad\text{supp}\chi_+\subset\big[\frac{c_0}{2},+\infty),\\
\displaystyle\chi_-=1\text{ on }(-\infty,-c_0],\ \quad\text{supp}\chi_-\subset(-\infty,-\frac{c_0}{2}\big],\\
\displaystyle\chi_0(\theta)^2+\chi_+(\theta)^2+\chi_-(\theta)^2=1,\quad\forall  \theta\in\R.
\end{cases}
\end{equation}
See Figure \ref{pic1}.
\begin{figure}[ht]
\begin{center}
\hskip0truecm\scalebox{0.8}{\includegraphics{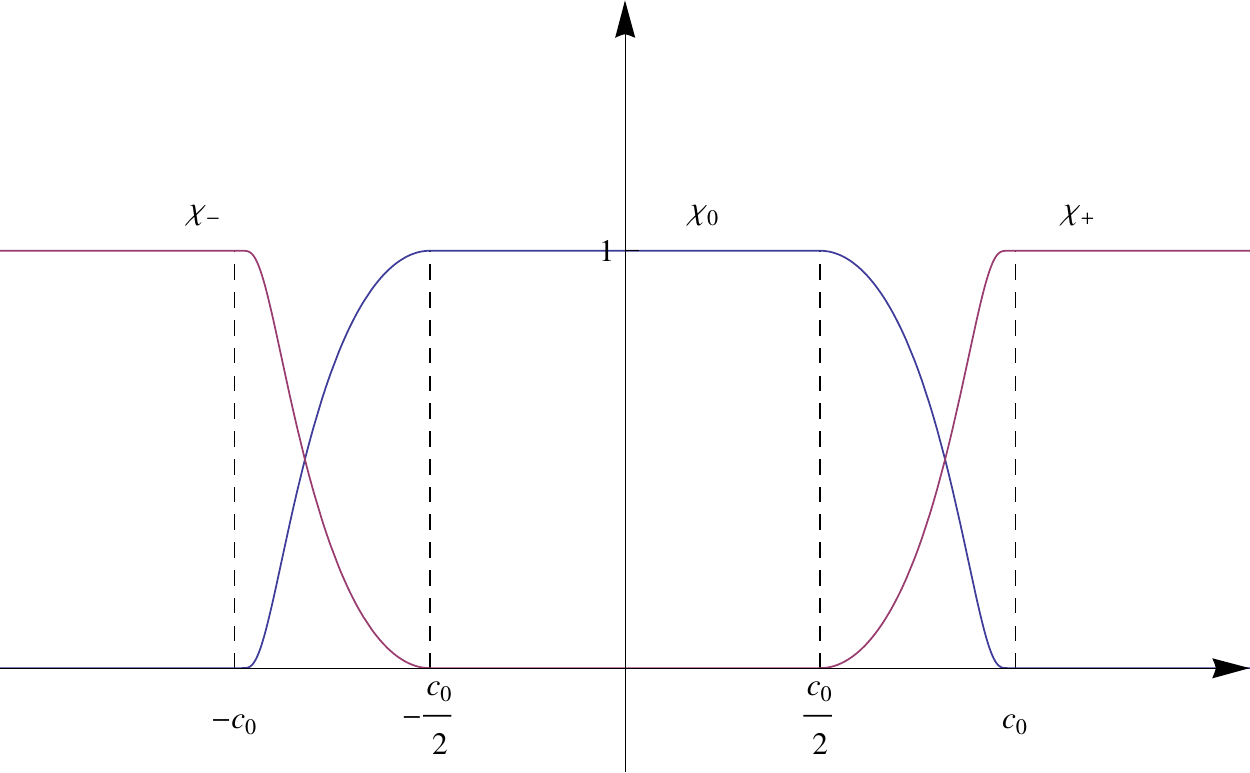}}
%\vskip-48pt
\caption{\sc The functions $\chi_0,\chi_\pm$.}
\label{pic1}
\end{center}
\end{figure}
Choose a function $\widetilde\chi_0\in C_0^\infty(\R;[0,1])$ such that 
\begin{equation}\label{eq.tilde.chi}
\displaystyle\widetilde\chi_0=1\text{ on }[-2c_0,2c_0],\quad\text{supp}\widetilde\chi_0\subset[-3c_0,3c_0].
\end{equation}
Take a decreasing function $\psi\in C^\infty(\R;[-1,1])$ such that
\begin{equation}\label{eq.psi}
\psi=1\text{ on }(-\infty,-2],\quad \psi=-1\text{ on }[2,+\infty),\quad\psi'=-\frac{1}{2}\text{ on }[-1,1].
\end{equation}
We can assume that $\psi$ has a factorization
\begin{equation}\label{eq.factorization}
\psi(\theta)=- e(\theta)\theta,
\end{equation}
where $e\in C_b^\infty(\R;[0,1])$ satisfies\footnote{We denote by $C_b^\infty(\R;[0,1])$ the set of smooth functions defined on $\R$ with values in $[0,1]$ such that all their derivatives are bounded.} that
\begin{equation*}
e(\theta)=\frac{1}{2}\text{ for }\theta\in[-1,1],\quad e(\theta)=|\theta|^{-1}\text{ for }|\theta|\geq2.
\end{equation*}

%%%%%%%%%%%%%%%%%%%%%%%%%%%
\subsubsection{Plan of the paragraph}\label{sec.plan}
The sections \ref{sec.nontrivialcase}, \ref{sec.nontrivial.c} are organized as follows.
Recall the four cases given in \eqref{4cases} and we give in Proposition \ref{prop.sigma} inequalities about the function $\sigma$ that will be used in the proof for the first three cases. 

Section \ref{sec.case1} is devoted to the proof for Case 1 where $e^{t_k}>\epsilon_0^{-1}$. We shall construct a multiplier adapted to the change-of-sign situation. Moreover, there is a special localization effect in this case (see Remark \ref{rem.localization}).

In Section \ref{sec.case2}, we prove estimates for Case 2 where $e^{t_k}\in[\epsilon_1,\epsilon_0^{-1}]$. The multiplier to be used in this case is the same as that in Case 1.

In Section \ref{sec.case3}, we prove estimates for Case 3 where $e^{t_k}\in(\beta_k^{-1/4},\epsilon_1)$. The multiplier will be different from that in the previous cases and the condition $e^{t_k}>\beta_k^{-1/4}$ is required such that the metric verifies  the uncertainty principle.

Finally, Section \ref{sec.case4} is devoted to proving estimates for the last case where $e^{t_k}\leq\beta_k^{-1/4}$ and estimates are easily obtained by using the multipliers ${\rm Id}$, $-i{\rm Id}$.

%%%%%Case 1%%%%
\subsubsection{Case 1: $e^{t_k}>\epsilon_0^{-1}$}\label{sec.case1}
We present in Proposition \ref{prop.sigma},(1) some inequalities about the function $\sigma$ that will be used in this case.

\begin{thm}\label{thm.case1}
Suppose $e^{t_k}>\epsilon_0^{-1}$.
There exist $C>0$, $k_0\geq1$ such that for all $k\geq k_0$, $\alpha\geq8\pi$, $u\in C_0^\infty(\R)$,
\begin{equation}\label{est.case1}
\normt{e^{-t}\SL_ku}\geq C\beta_k^{1/3}\normt{e^{t}u},
\end{equation}
where $\SL_k$ is given in \eqref{eq.lk2} and $\beta_k$ is given in \eqref{eq.beta}.
\end{thm}

\noindent{\bf a. Definition of the multiplier.}
We first give the definition of the H\"ormander-type metric that we shall work with (see Appendix \ref{sec.app.weyl}).
\begin{defi}\label{def.metric}
Define a metric on the phase space $\R_t\times\R_\tau$
\begin{align*}
\Gamma&=|dt|^2+\frac{|d\tau|^2}{\tau^2+\beta_k^{2/3}},%\qquad\widetilde\Gamma=|dt|^2+\frac{|d\tau|^2}{k^2+\tau^2},
\end{align*}
which is admissible with 
\begin{equation}\label{eq.lambda}
\lambda_\Gamma=(\tau^2+\beta_k^{2/3})^{1/2}\geq\beta_k^{1/3}\geq\big(\frac{\alpha}{8\pi}\big)^{1/3}\geq1,\quad\text{provided }\alpha\geq8\pi.
\end{equation}
%\begin{equation}\label{eq.tilde.lambda}
%\text{and}\qquad\lambda_{\widetilde\Gamma}=(\tau^2+k^2)^{1/2}\geq k\geq1.
%\end{equation}
\end{defi}

\begin{rem}\rm
We give a proof for the uniform admissibility (w.r.t. $k\geq1,\alpha\geq8\pi$) of the metric $\Gamma$ in Lemma \ref{lem.metric}.
Moreover, the function $f(\beta_k^{-1/3}\tau)$ belongs to $S(1,\Gamma)$ whenever $f\in S(1,\frac{|d\theta|^2}{1+\theta^2})$, since for any $n\in\N$,
\begin{align*}
\big|\frac{\p^n}{\p\tau^n}&\bigl(f(\beta_k^{-1/3}\tau)\bigr)\big|=\big|f^{(n)}(\beta_k^{-1/3}\tau)\beta_k^{-n/3}\big|\\
&\leq C_n(1+|\beta_k^{-1/3}\tau|^2)^{-n/2}\beta_k^{-n/3}=C_n(\beta_k^{2/3}+\tau^2)^{-n/2}.
\end{align*}
\end{rem}

Now we can construct the multiplier, using the functions that we have chosen in \eqref{eq.chi}, \eqref{eq.psi}.
\begin{defi}\label{def.multiplier}
\begin{equation}\label{eq.Mk}
M_k=m_{0,k}^w+m_{+,k}^w+m_{-,k}^w,
\end{equation}
where
\begin{align*}
 m_{0,k}(t,\tau)&=\chi_0(t-t_k)\sharp\psi(\beta_k^{-1/3}\tau)\sharp\chi_0(t-t_k),\\
m_{+,k}(t,\tau)&=-i\beta_k^{-1/3}\chi_+(t-t_k)^2,\\
m_{-,k}(t,\tau)&=i\beta_k^{-1/3}\chi_-(t-t_k)^2,
\end{align*}
where $a^w$ stands for the Weyl quantization for the symbol $a$ and $\sharp$ denotes the composition law in Weyl calculus. (See Appendix \ref{sec.app.weyl} for Weyl calculus.)
\end{defi}

\begin{rem}\rm
The functions $\chi_0(t-t_k)$, $\chi_{\pm}(t-t_k)$, $\psi(\beta_k^{-1/3}\tau)$ are real-valued symbols in $S(1,\Gamma)$. 
%and we have
%\begin{equation}\label{eq.m0}
%m_{0,k}=\chi_0(t-t_k)^2\psi(\beta_k^{-1/3}\tau)+r_0,\qquad\text{with }r_0\in S(\beta_k^{-2/3},\Gamma),
%\end{equation}
%by \eqref{eq.aba}.
Then $M_k$ given in Definition \ref{def.multiplier} is a bounded operator on $L^2(\R;dt)$.
Moreover, we see that
\begin{equation}\label{eq.m0w}
m_{0,k}^w=\chi_0(t-t_k)\psi(\beta_k^{-1/3}D_t)\chi_0(t-t_k)
\end{equation}
and $M_k$ can be written as
\begin{equation*}
M_k=\chi_0(t-t_k)\psi(\beta_k^{-1/3}D_t)\chi_0(t-t_k)-i\beta_k^{-1/3}\chi_+(t-t_k)^2+i\beta_k^{-1/3}\chi_-(t-t_k)^2.
\end{equation*}
Furthermore, the operator $e^tM_{k}^we^{-t}$ is bounded on $L^2(\R;dt)$, since
\begin{align}\label{eq.m01}
e^tm_{0,k}^we^{-t}&=\big[e^{t-t_k}\chi_0(t-t_k)\big]\psi(\beta_k^{-1/3}D_t)\big[\chi_0(t-t_k)e^{-(t-t_k)}\big],
\end{align}
and $|e^{\pm(t-t_k)}\chi_0(t-t_k)|\leq e^{c_0}$. 

\medskip

The three parts in $M_k$ are used to handle different zones in the phase space. We use $m_{0,k}^w$ to localize near the point $t_k$, where the change-of-sign of $\sigma(e^t)-\sigma(e^{t_k})$ happens. The Fourier multiplier $\psi(\beta_k^{-1/3}D_t)$ allows us to obtain some subelliptic estimate in this zone, acting with the skew-adjoint part of $\SL_k$. As we shall see in the computations, it is important to put the cutoff function $\chi_0(t-t_k)$ on both sides of $\psi(\beta_k^{-1/3}D_t)$, so that we are able to do symbolic calculus with  the exponential functions since they are all localized near $t_k$.

The other two multipliers $m_{\pm,k}^w$ are used for dealing with the zones where there is no change-of-sign of $\sigma(e^t)-\sigma(e^{t_k})$, that is $t$ away from the point $t_k$, and the sign of $m_{+,k}, m_{-,k}$ corresponds exactly to the sign of $\sigma(e^t)-\sigma(e^{t_k})$ on their supports. If the non-local term $i\beta_k\gamma\Dk\gamma$ were not present, then we could remove the factor $\beta_k^{-1/3}$ to get better estimates in these zones, as we have already done in \cite{D}. However, we see that the non-local term has a large coefficient $\beta_k$ and it does not commute with $\chi_{\pm}(t-t_k)$, so that we would obtain a commutator of size $\beta_k$ that we would not know how to control. Our strategy is to weaken the multiplier in these regions by multiplying a factor $\beta_k^{-1/3}$.

The method that we use here is perturbative: the non-local term is treated as a perturbation with respect to the main term $\sigma(e^t)-\sigma(e^{t_k})$. Thanks to the operator $\Dk$ and the nice function $\gamma(t)$ (see \eqref{eq.Dk}), this perturbation is controlled by the main term with an extra factor $k^{-2}$. Letting $k\geq k_0$, with $k_0\geq1$ a constant independent of the parameter $\alpha$, we can get the desired result.

However, it is of course impossible to consider the non-local term as a ``global" perturbation, i.e. to absorb it by a term controlled by $\|e^{-t}\SL_{k,0}u\|_{L^2(\R;dt)}$, where $\SL_{k,0}$ is the unperturbed part of $\SL_k$: in fact the size of that perturbation is $\beta_k$ and the best estimate we can hope is controlling a factor $\beta_k^{1/3}$. We have instead to follow our multiplier method to check the effect of the perturbation.
\end{rem}

\medskip

\noindent{\bf b. Computations.}
Now let us compute 2Re$\poscal{\SL_ku}{M_ku}$. 
\begin{prop}\label{prop.case1}
Suppose $e^{t_k}>\epsilon_0^{-1}$.
There exist $c,C>0$ such that for all $k\geq1$, $\alpha\geq8\pi$, $u\in C_0^\infty(\R)$,
\begin{align}
\notag2{\rm Re}\poscal{&\SL_ku}{M_ku}\geq c\beta_k^{2/3}\poscal{\rho(t,t_k)u}{u}-C\beta_k^{2/3}k^{-2}\kappa(e^{t_k})\normt{e^{2t}g(e^t)^{1/2}u}^2\\
\label{eq.prop.case1}&\quad-2\beta_k^{2/3}k^{-2} \normt{e^{2t}g(e^t)\chi_-u}^2-C\normt{D_tu}^2-Ck^2\normt{u}^2-C\normt{e^{2t}u}^2,
\end{align}
where $\SL_k$ is given in \eqref{eq.lk2}, $M_k$ in Definition \ref{def.multiplier}, $\chi_0,\chi_\pm$ in \eqref{eq.chi}, $\sigma,g$ in \eqref{eq.sigma}, $\beta_k$ in \eqref{eq.beta},
\begin{equation}\label{eq.rho1}
\rho(t,t_k)=\chi_0(t-t_k)^2+e^{2t}\sigma(e^{t_k})\chi_+(t-t_k)^2+e^{2t}\sigma(e^{t})\chi_-(t-t_k)^2
\end{equation}
\begin{equation}
\text{and}\qquad\kappa(e^{t_k})=g(e^{t_k})^{1/2}\max\big(1,\big|2-\frac14e^{2t_k}\big|,\big|4-\frac32e^{2t_k}+\frac{1}{16}e^{4t_k}\big|\big).
\end{equation}
 \end{prop}
 
 \medskip

\noindent{\it Proof of Proposition \ref{prop.case1}.}
First recall that
for all $u\in C_0^\infty(\R)$,
\begin{align}
{\rm Re}\poscal{\SL_ku}{u}%&=\poscal{\big(D_t^2+k^2+\frac{1}{16}e^{4t}\big)u}{u}\\
\label{eq.re}&=\normt{D_tu}^2+k^2\normt{u}^2+\frac{1}{16}\normt{e^{2t}u}^2.
\end{align}
In the following computations, we omit the dependence of $\chi_j(t-t_k)$ on $t-t_k$ for the sake of brevity.

\bigskip
%%%%Estimates for A%%%%%

\noindent{\it Estimates for $2{\rm Re}\poscal{\SL_ku}{m_{0,k}^wu}$.}
\begin{align}
\notag A:=2\text{Re}\poscal{\SL_ku}{m_{0,k}^wu}&=2{\rm Re}\poscal{i\beta_ke^{2t}\big(\sigma(e^t)-\sigma(e^{t_k})\big)u}{m_{0,k}^wu}\\
\notag&\quad-2{\rm Re}\poscal{i\beta_k\gamma\Dk\gamma u}{m_{0,k}^wu}\\
\notag&\quad+2{\rm Re}\poscal{\big(D_t^2+k^2+\frac{1}{16}e^{4t}\big)u}{m_{0,k}^wu}\\
\label{def.A}&=:A_1+A_2+A_3.
\end{align}
%A_1
Noticing $\chi_0\widetilde\chi_0=\chi_0$ and \eqref{eq.m0w}, we have
\begin{align*}
A_1&=2{\rm Re}\poscal{i\beta_ke^{2t}\big(\sigma(e^t)-\sigma(e^{t_k})\big)u}{\chi_0\psi(\beta_k^{-1/3}D_t)\chi_0u}\\
%&=2{\rm Re}\poscal{i\beta_k\widetilde\chi_0e^{2t}\big(\sigma(e^t)-\sigma(e^{t_k})\big)u}{\chi_0\psi(\beta_k^{-1/3}D_t)\chi_0u}\\
&=2{\rm Re}\poscal{i\beta_k\widetilde\chi_0e^{2t}\big(\sigma(e^t)-\sigma(e^{t_k})\big)\chi_0u}{\psi(\beta_k^{-1/3}D_t)\chi_0u}\\
&=\poscal{\Big[\psi(\beta_k^{-1/3}D_t),i\beta_k\widetilde\chi_0e^{2t}\big(\sigma(e^t)-\sigma(e^{t_k})\big)\Big]\chi_0u}{\chi_0u}.
\end{align*}
By \eqref{eq.case1sigma.1}, we know that the symbol $\beta_k\widetilde\chi_0e^{2t}\big(\sigma(e^t)-\sigma(e^{t_k})\big)$ belongs to $S(\beta_k,\Gamma)$ and we get 
$$\Big[\psi(\beta_k^{-1/3}D_t),\ i\beta_k\widetilde\chi_0e^{2t}\big(\sigma(e^t)-\sigma(e^{t_k})\big)\Big]=b_1^w+r_1^w,$$
where $b_1\in S(\beta_k\lambda_\Gamma^{-1},\Gamma)$ is a Poisson bracket and $r_1\in S(\beta_k\lambda_\Gamma^{-3},\Gamma)\subset S(1,\Gamma)$, with $\lambda_\Gamma$ given in \eqref{eq.lambda} (see \eqref{eq.r}).
More precisely, we have
\begin{align*}
b_1(t,\tau)&=\frac{1}{i}\Big\{\psi(\beta_k^{-1/3}\tau),i\beta_k\widetilde\chi_0e^{2t}\big(\sigma(e^t)-\sigma(e^{t_k})\big)\Big\}\\
&=\beta_k^{2/3}\psi'(\beta_k^{-1/3}\tau)\frac{d}{dt}\Big(\widetilde\chi_0e^{2t}\big(\sigma(e^t)-\sigma(e^{t_k})\big)\Big)\in S(\beta_k^{2/3},\Gamma).
\end{align*}
By \eqref{eq.tilde.chi}, \eqref{eq.psi} and \eqref{eq.case1sigma}, we have in the zone $\{|t-t_k|\leq 2c_0,|\tau|\leq\beta_k^{1/3}\}$ 
\begin{align*}
\label{eq.b1}b_1(t,\tau)&=\beta_k^{2/3}\psi'(\beta_k^{-1/3}\tau)\frac{d}{dt}\Big(e^{2t}\big(\sigma(e^t)-\sigma(e^{t_k})\big)\Big)\geq \frac{C_1}{2}\beta_k^{2/3}.
\end{align*}
This implies for all $t,\tau\in\R$,
\begin{equation}\label{eq.fp}
\frac{C_1}{2}\beta_k^{2/3}\leq b_1(t,\tau)+\frac{C_1}{2}\tau^2+\tilde C_1\beta_k^{2/3}\Big(1-\widetilde\chi_0\big(2(t-t_k)\big)\Big)\in S(\lambda_\Gamma^2,\Gamma),
\end{equation}
where $\tilde C_1=2\|b_1\|_{0,S(\beta_k^{2/3},\Gamma)}$.
Indeed, the function
$$b_1(t,\tau)+\tilde C_1\beta_k^{2/3}\Big(1-\widetilde\chi_0\big(2(t-t_k)\big)\Big)\geq \frac{C_1}{2}\beta_k^{2/3}\text{ for all $t\in\R$ and $|\tau|\leq\beta_k^{1/3}$},$$
and it is non-negative for all $t,\tau\in\R$ ; if $|\tau|\geq\beta_k^{1/3}$, then $\tau^2\geq\beta_k^{2/3}$, which proves the inequality in \eqref{eq.fp}. Moreover, each term in the right hand side of \eqref{eq.fp} is in $S(\lambda_\Gamma^2,\Gamma)$.
The Fefferman-Phong inequality (Proposition \ref{prop.fp}) implies 
$$ b_1(t,\tau)^w+\frac{C_1}{2}D_t^2+\tilde C_1\beta_k^{2/3}\Big(1-\widetilde\chi_0\big(2(t-t_k)\big)\Big)\geq \frac{C_1}{2}\beta_k^{2/3}-C'.$$
Applying to $\chi_0u$ and noting $\chi_0(\cdot)\widetilde\chi_0(2\cdot)=\chi_0(\cdot)$, we obtain
\begin{align*}
A_1+\frac{C_1}{2}\poscal{D_t^2\chi_0u}{\chi_0u}&=\poscal{\big(\frac{C_1}{2}D_t^2+b_1^w\big)\chi_0u}{\chi_0u}+\poscal{r_1^w\chi_0u}{\chi_0u}\\
&\geq \frac{C_1}{2}\beta_k^{2/3}\normt{\chi_0u}^2-C''\normt{\chi_0u}^2.
\end{align*}
On the other hand, we have
\begin{align*}
\poscal{D_t^2\chi_0u}{\chi_0u}=\normt{D_t\chi_0u}^2&\leq2\normt{\chi_0D_tu}^2+2\normt{\chi_0'u}^2\\
&\leq C\normt{D_tu}^2+C\normt{u}^2,
\end{align*}
which gives
\begin{align}
\label{eq.A1}A_1\geq \frac{C_1}{2}\beta_k^{2/3}\normt{\chi_0u}^2-C\normt{D_tu}^2-C\normt{u}^2.
\end{align}

%A_2
For the term $A_2$ defined in \eqref{def.A}, we have
\begin{align}
\notag A_2&=-2{\rm Re}\poscal{i\beta_k\gamma\Dk\gamma u}{m_{0,k}^wu}\\
\notag&=-2{\rm Re}\poscal{i\beta_k \Dk\gamma u}{\gamma m_{0,k}^wu}\\
\notag&=-2{\rm Re}\poscal{i\beta_k \Dk\gamma u}{m_{0,k}^w\gamma u}-2{\rm Re}\poscal{i\beta_k \Dk \gamma u}{\big[\gamma ,m_{0,k}^w\big]u}\\
\label{def.A2}&=:A_{21}+A_{22}.
\end{align}
%A_{21}
For $A_{21}$ in \eqref{def.A2}, since $i\Dk$ is skew-adjoint and $m_{0,k}^w$ is self-adjoint, we have
\begin{align}
\label{def.A21} A_{21}&=i\beta_k\poscal{ \big[\Dk,m_{0,k}^w\big] \gamma u}{\gamma u}.
\end{align}
Recalling \eqref{eq.m0w} and noting that $\Dk$ commutes with $\psi(\beta_k^{-1/3}D_t)$, we get
\begin{align}
\big[\Dk,m_{0,k}^w\big]%&=\big[\Dk,\chi_0\psi(\beta_k^{-1/3}D_t)\chi_0\big] \\
\label{eq.A21'}&=\big[\Dk,\chi_0\big]\psi(\beta_k^{-1/3}D_t)\chi_0+\chi_0\psi(\beta_k^{-1/3}D_t)\big[\Dk,\chi_0\big].
\end{align}
We compute the commutator as follows 
\begin{align*}
\big[\Dk,\chi_0\big]&=\Dk\chi_0-\chi_0\Dk\\
&=\Dk\Big(\chi_0(D_t^2+k^2)-(D_t^2+k^2)\chi_0\Big)\Dk\\
&=\Dk\Big([\chi_0,D_t]D_t+D_t[\chi_0,D_t]\Big)\Dk\\
&=i\Dk\chi_0'D_t\Dk+i\Dk D_t\chi_0'\Dk,
\end{align*}
which implies
$$\big[\Dk,\chi_0\big]D_t=i\underbrace{\Dk}_{\leq k^{-2}}\chi_0'\underbrace{D_t\Dk D_t}_{\leq1}+i\underbrace{\Dk D_t}_{\leq(2k)^{-1}}\chi_0'\underbrace{\Dk D_t}_{\leq(2k)^{-1}},$$
\begin{align*}
\text{and}\qquad
\|\big[\Dk,\chi_0\big]D_t\|_{{\mathcal L}(L^2(\R;dt))}\leq \frac54\|\chi'_0\|_{L^\infty}k^{-2}.
\end{align*}
The factorization \eqref{eq.factorization} of $\psi$ gives 
$$\psi(\beta_k^{-1/3}D_t)=-\beta_k^{-1/3}D_te(\beta_k^{-1/3}D_t)=-\beta_k^{-1/3}e(\beta_k^{-1/3}D_t)D_t,$$
so that
\begin{align*}
\big[\Dk,\chi_0\big]\psi(\beta_k^{-1/3}D_t)\chi_0&=-\beta_k^{-1/3}\underbrace{\big[\Dk,\chi_0\big]D_t}_{\text{has norm }\leq Ck^{-2}}\underbrace{e(\beta_k^{-1/3}D_t)}_{\text{bounded}}\chi_0,
\end{align*}
\begin{align*}
\text{and}\qquad&\big|\poscal{\big[\Dk,\chi_0\big]\psi(\beta_k^{-1/3}D_t)\chi_0\gamma u}{\gamma u}\big|\leq C\beta_k^{-1/3}k^{-2}\|\chi_0\gamma u\|\|\gamma u\|.
\end{align*}
Similarly we can get
\begin{align*}
&\big|\poscal{\chi_0\psi(\beta_k^{-1/3}D_t)\big[\Dk,\chi_0\big]\gamma u}{\gamma u}\big|\leq C\beta_k^{-1/3}k^{-2}\|\gamma u\|\|\chi_0\gamma u\|.
\end{align*}
By \eqref{def.A21} and \eqref{eq.A21'}, we obtain
\begin{equation}\label{eq.A21}
|A_{21}|\leq C\beta_k^{2/3}k^{-2}\|\chi_0\gamma u\|\|\gamma u\|.
\end{equation}

%A_{22}
For the term $A_{22}$ defined in \eqref{def.A2}, we have, using \eqref{eq.m0w}
\begin{align*}
A_{22}&=-2{\rm Re}\poscal{i\beta_k \Dk \gamma u}{\big[\gamma ,m_{0,k}^w\big]u}\\
&=-2{\rm Re}\poscal{i\beta_k \Dk \gamma u}{\chi_0\big[\gamma,\psi(\beta_k^{-1/3}D_t)\big]\chi_0u},
\end{align*}
so that we should compute the commutator $[\gamma,\psi(\beta_k^{-1/3}D_t)]$, for which we will do some symbolic calculus with the metric $\Gamma$ given in Definition \ref{def.metric}.
The symbol $\gamma(t)$ is in $S(1,\Gamma)$ since $\gamma(t)=e^{2t}e^{-e^{2t}/8}$ belongs to $C_b^\infty(\R)$, and we get
$$\big[\gamma,\psi(\beta_k^{-1/3}D_t)\big]=b_2^w+r_2^w,$$
where $b_2\in S(\lambda_\Gamma^{-1},\Gamma)$ is a Poisson bracket and $ r_2\in S(\lambda_\Gamma^{-3},\Gamma)\subset S(\beta_k^{-1},\Gamma)$, with $\lambda_\Gamma$ given in \eqref{eq.lambda} (see \eqref{eq.r}). By direct computation, we have
\begin{align*}
b_2&=\frac{1}{i}\big\{\gamma(t),\psi(\beta_k^{-1/3}\tau)\big\}\\
&=-\frac{1}{i}\beta_k^{-1/3}\psi'(\beta_k^{-1/3}\tau)\gamma'(t)\quad\in S(\beta_k^{-1/3},\Gamma)\\
&=-\frac{1}{i}\beta_k^{-1/3}\psi'(\beta_k^{-1/3}\tau)\sharp \gamma'(t)+b_3+r_3,
\end{align*}
where $b_3\in S(\beta_k^{-1/3}\lambda_\Gamma^{-1},\Gamma)$ is again a Poisson bracket and $r_3$ belongs to $S(\beta_k^{-1/3}\lambda_\Gamma^{-2},\Gamma)$ thus to $S(\beta_k^{-1},\Gamma)$, since $\lambda_\Gamma\geq\beta_k^{1/3}$. We continue to expand $b_3$
\begin{align*}
b_3&=-\frac{1}{2i}\big\{-\frac{1}{i}\beta_k^{-1/3}\psi'(\beta_k^{-1/3}\tau), \gamma'(t)\big\}\\
&=-\frac{1}{2}\beta_k^{-2/3}\psi''(\beta_k^{-1/3}\tau)\gamma''(t)\quad\in S(\beta_k^{-2/3},\Gamma)\\
&=-\frac{1}{2}\beta_k^{-2/3}\psi''(\beta_k^{-1/3}\tau)\sharp\gamma''(t)+r_4,
\end{align*}
with $r_4\in S(\beta_k^{-2/3}\lambda_\Gamma^{-1},\Gamma)\subset S(\beta_k^{-1},\Gamma)$. Thus we get for $w\in C_0^\infty(\R)$,
\begin{align*}
\big[\gamma,\psi(\beta_k^{-1/3}D_t)\big]w&= -\frac{1}{i}\beta_k^{-1/3}\psi'(\beta_k^{-1/3}D_t)\gamma'(t)w-\frac{1}{2}\beta_k^{-2/3}\psi''(\beta_k^{-1/3}D_t)\gamma''(t)w\\
&\quad+\big(r_2^w+r_3^w+r_4^w\big)w,
\end{align*}
where $r_2,r_3,r_4\in S(\beta_k^{-1},\Gamma)$.
Since $\psi'(\beta_k^{-1/3}D_t)$ and $\psi''(\beta_k^{-1/3}D_t)$ are bounded on $L^2(\R;dt)$, we deduce that for $w\in C_0^\infty(\R)$,
\begin{align*}
\normt{\big[\gamma(t),\psi(\beta_k^{-1/3}D_t)\big]w}&\leq C\beta_k^{-1/3}\normt{\gamma'w}+C\beta_k^{-2/3}\normt{\gamma''w}+C\beta_k^{-1}\normt{w}.
\end{align*}
Applying the above inequality to $\chi_0u$, we get the estimate for $A_{22}$ defined in \eqref{def.A2}:
\begin{align}
\notag\big|A_{22}\big|&\leq 2\beta_k \normt{\Dk \gamma u}\normt{\chi_0\big[\gamma,\psi(\beta_k^{-1/3}D_t)\big]\chi_0u}\\
\label{eq.A22}&\leq 2\beta_kk^{-2} \normt{\gamma u}\times\Big(C\beta_k^{-1/3}\normt{\gamma'\chi_0u}+C\beta_k^{-2/3}\normt{\gamma''\chi_0u}+C\beta_k^{-1}\normt{\chi_0u}\Big).
\end{align}
It follows from \eqref{def.A2}, \eqref{eq.A21} and \eqref{eq.A22} that
\begin{align}
\notag\big|A_2\big|&\leq C\beta_k^{2/3}k^{-2} \normt{\gamma u}\times\big(\normt{\chi_0\gamma u}+\normt{\chi_0\gamma'u}\big)\\
\label{eq.A2'}&\qquad+C\beta_k^{1/3}k^{-2}\normt{\gamma u}\normt{\chi_0\gamma''u}+Ck^{-2}\normt{\gamma u}\normt{\chi_0u}.
\end{align}
Recall $\gamma(t)=e^{2t}g(e^t)=e^{2t}e^{-e^{2t}/8}\leq 8/e$ and $g(e^t)=e^{-e^{2t}/8}$, then
\begin{equation}\label{eq.gamma.d}
\gamma'(t)=e^{2t}g(e^t)\big(2-\frac14e^{2t}\big),\quad \gamma''(t)=e^{2t}g(e^t)\big(4-\frac32e^{2t}+\frac{1}{16}e^{4t}\big).
\end{equation}
Letting
\begin{equation}\label{eq.kappak}
\kappa(r)=g(r)^{1/2}\max\big(1,\big|2-\frac{1}{4}r^{2}\big|,\big|4-\frac32r^{2}+\frac{1}{16}r^{4}\big|\big),
\end{equation}
we deduce from \eqref{eq.A2'} that
\begin{align}
\label{eq.A2}\big|A_2\big|&\leq C\beta_k^{2/3}k^{-2} \kappa(e^{t_k})\normt{e^{2t}g(e^t)^{1/2} u}^2+Ck^{-2}\normt{u}^2.
\end{align}

\begin{rem}\label{rem.kappak}\rm
When $e^{t_k}$ is taken large (and we do not need  $k$ large), $\kappa(e^{t_k})$ is very small. In particular, if $\epsilon_0$ is small, since $e^{t_k}>\epsilon_0^{-1}$, $\kappa(e^{t_k})$ is bounded above by $\kappa(\epsilon_0^{-1})$.
\end{rem}

For the term $A_3$ defined in \eqref{def.A}, we have
\begin{align}
\notag A_3&=2{\rm Re}\poscal{D_t^2u}{m_{0,k}^wu}+2{\rm Re}\poscal{k^2u}{m_{0,k}^wu}+\frac{1}{8}{\rm Re}\poscal{e^{4t}u}{m_{0,k}^wu}\\
\label{def.A3}&=:A_{31}+A_{32}+A_{33}.
\end{align}
%A_{31}
For $A_{31}$ in \eqref{def.A3},
\begin{align*}
A_{31}&=2{\rm Re}\poscal{D_tu}{D_tm_{0,k}^wu}\\
&=2{\rm Re}\poscal{D_tu}{m_{0,k}^wD_tu}+2{\rm Re}\poscal{D_tu}{\big[D_t,m_{0,k}^w\big]u}\\
&=2{\rm Re}\poscal{D_tu}{m_{0,k}^wD_tu}+\poscal{\big[D_t,[D_t,m_{0,k}^w]\big]u}{u}.
\end{align*}
Since $m_{0,k}\in S(1,\Gamma)$ and $\tau\in S(\lambda_\Gamma,\Gamma)$, the double commutator $\big[D_t,[D_t,m_{0,k}^w]\big]$ has a symbol in $S(1,\Gamma)$. We get
\begin{equation}\label{eq.A31}
\big|A_{31}\big|\leq C\normt{D_tu}^2+C\normt{u}^2.
\end{equation}
%A_{32}
Using the $L^2(\R;dt)$-boundedness of $m_{0,k}^w$, we get for $A_{32}$ defined in \eqref{def.A3}
\begin{equation}\label{eq.A32}
\big|A_{32}\big|\leq Ck^2\normt{u}^2.
\end{equation}
%A_{33}
For $A_{33}$ in \eqref{def.A3}, %using $\widetilde\chi_0\chi_0=\chi_0$, 
we have by \eqref{eq.m01}
\begin{align*}
8A_{33}&={\rm Re}\poscal{e^{4t}u}{\chi_0\psi(\beta_k^{-1/3}D_t)\chi_0u}\\
&={\rm Re}\poscal{e^{2t}u}{\underbrace{e^{2t}\chi_0\psi(\beta_k^{-1/3}D_t)\chi_0e^{-2t}}_{\text{bounded on }L^2(\R;dt)}e^{2t}u}.
%&=2{\rm Re}\poscal{\widetilde\chi_0^2e^{4t}u}{\chi_0\psi(\beta_k^{-1/3}D_t)\chi_0u}\\
%&=2{\rm Re}\poscal{\widetilde\chi_0e^{2t}\chi_0u}{\widetilde\chi_0e^{2t}\psi(\beta_k^{-1/3}D_t)\chi_0u}\\
%&=2{\rm Re}\poscal{\widetilde\chi_0e^{2t}\chi_0u}{\psi(\beta_k^{-1/3}D_t)\widetilde\chi_0e^{2t}\chi_0u}+2{\rm Re}\poscal{\widetilde\chi_0e^{2t}\chi_0u}{\big[\widetilde\chi_0e^{2t},\psi(\beta_k^{-1/3}D_t)\big]\chi_0u}\\
%&=2{\rm Re}\poscal{\widetilde\chi_0e^{2t}\chi_0u}{\psi(\beta_k^{-1/3}D_t)\widetilde\chi_0e^{2t}\chi_0u}+\poscal{\Big[\widetilde\chi_0e^{2t},\big[\widetilde\chi_0e^{2t},\psi(\beta_k^{-1/3}D_t)\big]\Big]\chi_0u}{\chi_0u}.
\end{align*}
%since $|e^{\pm 2t}\chi_0(t-t_k)|\leq e^{\pm2t_k+2c_0}$. 
Hence
%Since $\psi(\beta_k^{-1/3}\tau)\in S(1,\Gamma)$ and $\widetilde\chi_0e^{2t}\in S(e^{2t_{k}},\Gamma)$, the double commutator in the last equality has a symbol in $S(e^{4t_k}\lambda_\Gamma^{-2},\Gamma)\subset S(e^{4t_k}\beta_k^{-2/3},\Gamma)$, with $\lambda_\Gamma\geq\beta_k^{1/3}$, hence
\begin{equation}\label{eq.A33}
\big|A_{33}\big|%\leq C\normt{e^{2t}\chi_0u}^2+C\beta_k^{-2/3}e^{4t_k}\normt{\chi_0u}^2
\leq C\normt{e^{2t}u}^2.
\end{equation}
By \eqref{def.A3}, \eqref{eq.A31}, \eqref{eq.A32} and \eqref{eq.A33} we get
\begin{equation}\label{eq.A3}
\big|A_3\big|\leq C\normt{D_tu}^2+Ck^2\normt{u}^2+C\normt{e^{2t}u}^2.
\end{equation}

%A
We deduce from \eqref{def.A}, \eqref{eq.A1}, \eqref{eq.A2} and \eqref{eq.A3} that
\begin{align}
\notag A&\geq \frac{C_1}{2}\beta_k^{2/3}\normt{\chi_0u}^2
-C\beta_k^{2/3}k^{-2}\kappa(e^{t_k}) \normt{e^{2t}g(e^t)^{1/2}u}^2\\
\label{eq.A}&\qquad-C\normt{D_tu}^2-Ck^2\normt{u}^2-C\normt{e^{2t}u}^2,
\end{align}
with $\kappa(e^{t_k})$ given by \eqref{eq.kappak}.

\medskip

%%%%%Estimates for B^+%%%%%
\noindent{\it Estimates for $2{\rm Re}\poscal{\SL_ku}{m_{+,k}^wu}$.}
Recall that $m_{+,k}^w=-i\beta_k^{-1/3}\chi_+(t-t_k)^2$.
\begin{align}
\notag B^+&:=2{\rm Re}\poscal{\SL_ku}{m_{+,k}^wu}=2{\rm Re}\poscal{\SL_ku}{-i\beta_k^{-1/3}\chi_+^2u}\\
\notag&=2{\rm Re}\poscal{i\beta_ke^{2t}\big(\sigma(e^{t})-\sigma(e^{t_k})\big)u}{-i\beta_k^{-1/3}\chi_+^2u}\\
\notag&\quad-2{\rm Re}\poscal{i\beta_k\gamma\Dk\gamma u}{-i\beta_k^{-1/3}\chi_+^2u}\\
\notag&\quad+2{\rm Re}\poscal{D_t^2u}{-i\beta_k^{-1/3}\chi_+^2u}\\
\label{def.B+}&=:B_{1}^{+}+B^{+}_{2}+B^{+}_{3}.
\end{align}
%B^{+}_{1}
The support of $\chi_+(t-t_k)$ is included in the set $\{t-t_k\geq c_0/2\}$. By \eqref{eq.case1sigma.2} we have
\begin{align}
\notag B^{+}_{1}&=2\beta_k^{2/3}\poscal{e^{2t}\big(\sigma(e^{t_k})-\sigma(e^t)\big)u}{\chi_+^2u}\\
\label{eq.B+1}&\geq 2c_1\beta_k^{2/3}\poscal{e^{2t}\sigma(e^{t_k})\chi_+^2u}{u}.
\end{align}
%B_{+2}
For $B^{+}_{2}$ in \eqref{def.B+} we have
\begin{align*}
 B^{+}_{2}&=\beta_k^{2/3}2{\rm Re}\poscal{\chi_+\Dk\gamma u}{\chi_+\gamma u}\\
&=\beta_k^{2/3}2{\rm Re}\poscal{\big[\chi_+,\Dk\big]\gamma u}{\chi_+\gamma u}+ \beta_k^{2/3}2{\rm Re}\poscal{\Dk\chi_+\gamma u}{\chi_+\gamma u}\\
&=\beta_k^{2/3}\poscal{\big[\chi_+,[\chi_+,\Dk]\big]\gamma u}{\gamma u}+ \beta_k^{2/3}2{\rm Re}\underbrace{\poscal{\Dk\chi_+\gamma u}{\chi_+\gamma u}}_{\geq0}.
\end{align*}
%since $[\chi_+,\Dk]$ is skew-adjoint. 
The kernel of $\big[\chi_+,[\chi_+,\Dk]\big]$ is
$$\frac{1}{2k}e^{-k|t-s|}\big[\chi_+(t-t_k)-\chi_+(s-t_k)\big]^2,$$
vanishing if $\max(t,s)\leq t_k+c_0/2$ and also if $\min(t,s)\geq t_k+c_0$.
Then we have
\begin{align*}
&\big|\poscal{\big[\chi_+,[\chi_+,\Dk]\big]\gamma u}{\gamma u}\big|\\
&=\Big|\iint_{t-t_k\geq c_0/2} \frac{1}{2k}e^{-k|t-s|}\big[\chi_+(t-t_k)-\chi_+(s-t_k)\big]^2(\gamma u)(s)\overline{(\gamma u)(t)}dtds\\
&\qquad\ +\iint_{t-t_k\leq c_0/2\atop s-t_k\geq c_0/2} \frac{1}{2k}e^{-k|t-s|}\big[\chi_+(t-t_k)-\chi_+(s-t_k)\big]^2(\gamma u)(s)\overline{(\gamma u)(t)}dtds\Big|\\
&\leq g(e^{t_k})^{1/2}\iint_{t-t_k\geq c_0/2} \frac{1}{2k}e^{-k|t-s|}\Big(e^{2s}g(e^s)|u(s)|\Big)\Big( e^{2t}g(e^t)^{1/2} |u(t)|\Big)dtds\\
&\qquad\ +g(e^{t_k})^{1/2}\iint_{t-t_k\leq c_0/2\atop s-t_k\geq c_0/2} \frac{1}{2k}e^{-k|t-s|}\Big(e^{2s}g(e^s)^{1/2}|u(s)|\Big) \Big(e^{2t}g(e^t) |u(t)|\Big)dtds\\
&\leq g(e^{t_k})^{1/2}k^{-2}\|e^{2t}g(e^t)^{1/2}u\|^2,%\leq k^{-2}\kappa(e^{t_k})\normt{e^{2t}g(e^t)^{1/2}u}^2,
\end{align*}
so that we obtain
\begin{equation}\label{eq.B+2}
B^{+}_{2}\geq  -\beta_k^{2/3}k^{-2}\kappa(e^{t_k})\normt{e^{2t}g(e^t)^{1/2}u}^2,
\end{equation}
where $\kappa(e^{t_k})$ is given in \eqref{eq.kappak}.
%\begin{align}
%\notag \big|B_{+2}\big|&=2\beta_k^{2/3}\big|{\rm Re}\poscal{\gamma\Dk\gamma u}{\chi_+^2u}\big|\\
%&= 2\beta_k^{2/3}\big|{\rm Re}\poscal{\Dk\gamma u}{\chi_+^2\gamma u}\big|%+2\beta_k^{2/3}\poscal{(k^2+D_t^2)^{-1}\chi_+Gu}{\chi_+Gu}
%\label{eq.B+2}\leq 2\beta_k^{2/3}k^{-2}\normt{\gamma u}^2.
%\end{align}
%B_{+3}
For $B^{+}_{3}$ in \eqref{def.B+}, we have
\begin{align*}
B^{+}_{3}&=\beta_k^{-1/3}\poscal{\big[D_t^2,-i\chi_+^2\big]u}{u}\\
&=i\beta_k^{-1/3}\poscal{(2\chi_+\chi_+''+2\chi_+'^2)u}{u}+i\beta_k^{-1/3}\poscal{4\chi_+\chi_+'\p_tu}{u},
\end{align*}
which implies
\begin{equation}\label{eq.B+3}
\big|B^{+}_{3}\big|\leq C\beta_k^{-1/3}\normt{u}^2+C\beta_k^{-1/3}\normt{\p_tu}\normt{u}\leq C\normt{D_tu}^2+C\normt{u}^2.
\end{equation}

%B_+
We get from \eqref{def.B+}, \eqref{eq.B+1}, \eqref{eq.B+2} and \eqref{eq.B+3} that
\begin{align}
\notag B^+&\geq 2c_1\beta_k^{2/3}\poscal{e^{2t}\sigma(e^{t_k})\chi_+^2u}{u}\\
\label{eq.B+}&\qquad- \beta_k^{2/3}k^{-2}\kappa(e^{t_k})\normt{e^{2t}g(e^t)^{1/2}u}^2-C\normt{u}^2-C\normt{D_tu}^2.
\end{align}

\medskip

%%%%%Estimates for B_-%%%%%
\noindent{\it Estimates for $2{\rm Re}\poscal{\SL_ku}{m_{-,k}^wu}$.}
Recall that $m_{-,k}^w=i\beta_k^{-1/3}\chi_-(t-t_k)^2$.
\begin{align}
\notag B^-&:=2{\rm Re}\poscal{\SL_ku}{m_{-,k}^wu}=2{\rm Re}\poscal{\SL_ku}{i\beta_k^{-1/3}\chi_-^2u}\\
\notag&=2{\rm Re}\poscal{i\beta_ke^{2t}\big(\sigma(e^t)-\sigma(e^{t_k})\big)u}{i\beta_k^{-1/3}\chi_-^2u}\\
\notag&\qquad-2{\rm Re}\poscal{i\beta_k\gamma\Dk\gamma u}{i\beta_k^{-1/3}\chi_-^2u}\\
\notag&\qquad+2{\rm Re}\poscal{D_t^2u}{i\beta_k^{-1/3}\chi_-^2u}\\
\label{def.B-}&=B^{-}_{1}+B^{-}_{2}+B^{-}_{3}.
\end{align}
%B_{-1}
Recall \eqref{eq.case1sigma.2} and note that the support of $\chi_-(t-t_k)$ is included in $\{t-t_k\leq -c_0/2\}$, then we get
\begin{align}
\notag B^{-}_{1}&=2\beta_k^{2/3}\poscal{e^{2t}\big(\sigma(e^{t})-\sigma(e^{t_k})\big)u}{\chi_-^2u}\\
\label{eq.B-1}&\geq 2c_1\poscal{\beta_k^{2/3}e^{2t}\sigma(e^{t})\chi_-^2u}{u}.
\end{align}
%B_{-2}
For $B^{-}_{2}$ in \eqref{def.B-} we have
\begin{align*}
B^{-}_{2}&=-2\beta_k^{2/3}{\rm Re}\poscal{\chi_-\Dk\gamma u}{\chi_-\gamma u}\\
&= -2\beta_k^{2/3}{\rm Re}\poscal{\big[\chi_-,\Dk\big]\gamma u}{\chi_-\gamma u}-2\beta_k^{2/3}{\rm Re}\poscal{\Dk\chi_-\gamma u}{\chi_-\gamma u}\\
&= -\beta_k^{2/3}\poscal{\big[\chi_-,[\chi_-,\Dk]\big]\gamma u}{\gamma u}-2\beta_k^{2/3}{\rm Re}\poscal{\Dk\chi_-\gamma u}{\chi_-\gamma u}\\
&=:B^{-}_{21}+B^{-}_{22}.
\end{align*}
For $B^{-}_{21}$, we have 
\begin{equation*}
0\geq B^{-}_{22}=-2\beta_k^{2/3}{\rm Re}\poscal{\Dk\chi_-\gamma u}{\chi_-\gamma u}\geq -2\beta_k^{2/3}k^{-2}\normt{\chi_-\gamma u}^2.
\end{equation*}
By using the method that is used to estimate the double commutator in $B^{+}_{2}$, we find
\begin{equation*}
|B^{-}_{21}|\leq C \beta_k^{2/3}k^{-2}\kappa(e^{t_k})\normt{e^{2t}g(e^t)^{1/2}u}^2,
\end{equation*}
where $\kappa(e^{t_k})$ is given in \eqref{eq.kappak}, so that
\begin{equation}\label{eq.B-2}
B^{-}_{2}\geq -2\beta_k^{2/3}k^{-2}\normt{\chi_-\gamma u}^2- C\beta_k^{2/3}k^{-2}\kappa(e^{t_k})\normt{e^{2t}g(e^t)^{1/2}u}^2.
\end{equation}
%\begin{align}
%\notag\big|B_{-2}\big|&=\big|2\beta_k^{2/3}{\rm Re}\poscal{\gamma\Dk\gamma u}{\chi_-^2u}\big|\\
%&= \big|2\beta_k^{2/3}{\rm Re}\poscal{\Dk\gamma u}{\chi_-^2\gamma u}\big|
%\label{eq.B-2}\leq 2\beta_k^{2/3}k^{-2}\normt{\gamma u}^2.
%\end{align}
%B_{-3}
For $B^{-}_{3}$ in \eqref{def.B-}, we have
\begin{align*}
B^{-}_{3}&=\beta_k^{-1/3}\poscal{\big[D_t^2,i\chi_-^2\big]u}{u}\\
&=-i\beta_k^{-1/3}\poscal{(2\chi_-\chi_-''+2\chi_-'^2)u}{u}-i\beta_k^{-1/3}\poscal{4\chi_-\chi_-'\p_tu}{u},
\end{align*}
which implies
\begin{equation}\label{eq.B-3}
\big|B^{-}_{3}\big|\leq C\beta_k^{-1/3}\normt{u}^2+C\beta_k^{-1/3}\normt{\p_tu}\normt{u}\leq C\normt{u}^2+C\normt{D_tu}^2.
\end{equation}

%B_-
We get from \eqref{def.B-}, \eqref{eq.B-1}, \eqref{eq.B-2} and \eqref{eq.B-3} that
\begin{align}
\notag B^-&\geq 2c_1\beta_k^{2/3}\poscal{e^{2t}\sigma(e^{t})\chi_-^2u}{u}-C\normt{u}^2-C\normt{D_tu}^2\\
\label{eq.B-}&\qquad-2\beta_k^{2/3}k^{-2}\normt{\chi_-e^{2t}g(e^t) u}^2-\beta_k^{2/3}k^{-2}\kappa(e^{t_k})\normt{e^{2t}g(e^t)^{1/2}u}^2.
\end{align}

\medskip

%%%%%Conclusion%%%%%
\begin{proof}[End of the proof of Proposition \ref{prop.case1}]
By \eqref{eq.A}, \eqref{eq.B+}, \eqref{eq.B-} and the definition \ref{def.metric} of $M_k$, we get
\begin{align*}
\notag2{\rm Re}\poscal{\SL_ku}{&M_ku}=A+B^++B^-\\
\notag&\geq \frac{C_1}{2}\beta_k^{2/3}\normt{\chi_0u}^2+2c_1\beta_k^{2/3}\poscal{e^{2t}\sigma(e^{t_k})\chi_+^2u}{u}+2c_1\beta_k^{2/3}\poscal{e^{2t}\sigma(e^{t})\chi_-^2u}{u}\\
\notag&\quad-C\beta_k^{2/3}k^{-2}\kappa(e^{t_k})\normt{e^{2t}g(e^t)^{1/2}u}^2-2\beta_k^{2/3}k^{-2} \normt{\chi_-e^{2t}g(e^t)u}^2\\
\notag&\qquad-C\normt{D_tu}^2-Ck^2\normt{u}^2-C\normt{e^{2t}u}^2,
\end{align*}
where $\kappa(e^{t_k})$ is given in \eqref{eq.kappak}. This completes the proof of \eqref{eq.prop.case1} in Proposition \ref{prop.case1}.
\end{proof}

Recall the definition \eqref{eq.rho1} of $\rho(t,t_k)$, then \eqref{eq.prop.case1} implies
\begin{align}
\notag2{\rm Re}\poscal{\SL_ku}{M_ku}&\geq \beta_k^{2/3}\poscal{\Big(c\rho(t,t_k)-Ck^{-2}e^{4t}g(e^t)\Big)u}{u}\\
\label{eq.prop.case1.1}
&\qquad-C\normt{D_tu}^2-Ck^2\normt{u}^2-C\normt{e^{2t}u}^2,
\end{align}
since $\kappa(e^{t_k})$ is bounded above by a constant depending on $\epsilon_0$ (see Remark \ref{rem.kappak}).
We have the following two estimates for $\rho(t,t_k)$.

\begin{lem}\label{lem.rho1}
There exist $C_4,C_5>0$ such that for $e^{t_k}>\epsilon_0^{-1}$, $t\in\R$, $\alpha\geq8\pi$, $k\geq1$,
\begin{equation}\label{eq.rho1.1}
\rho(t,t_k)\geq C_4e^{4t}g(e^t),
\end{equation}
\begin{equation}\label{eq.rho1.2}
\beta_k^{2/3}\rho(t,t_k)+e^{4t}\geq C_5\beta_k^{1/3}e^{2t},
\end{equation}
where $\rho(t,t_k)$ is given in \eqref{eq.rho1}, $g$ is given in \eqref{eq.sigma} and $\beta_k$ is given in \eqref{eq.beta}.
\end{lem}

\begin{proof}[Proof of Lemma \ref{lem.rho1}]
Suppose $e^{t_k}>\epsilon_0^{-1}$, then $\sigma(e^{t_k})\geq \delta e^{-2t_k}$ for some $\delta >0$. Note also that the function $r^{4}g(r)$ is bounded.
If $t$ is in the support of $\chi_0(\cdot-t_k)$, i.e. $|t-t_k|\leq c_0$, we have
$$e^{4t}g(e^t)\leq Ce^{t_k}g(e^{t_k})\leq C,$$ 
$$\beta_k^{2/3}+e^{4t}= \big(\beta_k^{2/3}e^{-2t}+e^{2t}\big)e^{2t}\geq \beta_k^{1/3}e^{2t}.$$
When $t$ is in the support of $\chi_+(\cdot-t_k)$, i.e. $t\geq t_k+c_0/2$, we have
$$e^{2t}\sigma(e^{2t_k})\geq \delta e^{2t}e^{-2t_k}\geq  \delta e^{c_0}\geq Ce^{4t}g(e^t),$$
$$\beta_k^{2/3}e^{2t}\sigma(e^{t_k})+e^{4t}\geq ( \delta\beta_k^{2/3}e^{-2t_k}+e^{2t_k})e^{2t}\geq  \delta^{1/2}\beta_k^{1/3}e^{2t}.$$
When $t$ is in the support of $\chi_-(\cdot-t_k)$, i.e. $t\leq t_k-c_0/2$, we have, using \eqref{sigma.3} and the first inequality in \eqref{sigma.2}
$$e^{4t}g(e^t)\leq 16e^{2t}\sigma(e^{t}),$$
$$\beta_k^{2/3}e^{2t}\sigma(e^{t})+e^{4t}\geq \big(\beta_k^{2/3}\sigma(e^{t})+e^{2t}\big)e^{2t}\geq (2\log2)^{-1}\beta_k^{1/3}e^{2t}.$$
This completes the proof of \eqref{eq.rho1.1} and \eqref{eq.rho1.2}.
\end{proof}

\begin{proof}[Proof of Theorem \ref{thm.case1}]
The estimates \eqref{eq.prop.case1.1} and \eqref{eq.rho1.1} imply that there exists $k_0\geq1$, for all $k\geq k_0$, 
\begin{align*}
2{\rm Re}\poscal{\SL_ku}{M_ku}&\geq \frac{c}{2}\beta_k^{2/3}\poscal{\rho(t,t_k)u}{u}-C\normt{D_tu}^2-Ck^2\normt{u}^2-C\normt{e^{2t}u}^2.
\end{align*}
Together with \eqref{eq.re}, by choosing $C_6>0$ large enough, we have for $k\geq k_0$,
\begin{align}
{\rm Re}\poscal{\SL_ku}{(C_6+2M_k)u}&\geq \frac{c}{2}\poscal{\Big(\beta_k^{2/3}\rho(t,t_k)+D_t^2+k^2+e^{4t}\Big)u}{u}.
\label{eq.case1.1}
\end{align}
It follows from \eqref{eq.case1.1} and \eqref{eq.rho1.2} that for $k\geq k_0$,
\begin{align}\label{eq.case1.2}
{\rm Re}\poscal{\SL_ku}{(C_6+2M_k)u}&\geq C\poscal{\beta_k^{1/3}e^{2t}u}{u}.
\end{align}
Noticing that $e^t(C_6+2M_k)e^{-t}$ is bounded on $L^2(\R;dt)$ by \eqref{eq.m01} and that
$$\poscal{\SL_ku}{(C_6+2M_k)u}=\poscal{e^{-t}\SL_ku}{\big(e^t(C_6+2M_k)e^{-t}\big)(e^tu)},$$
we deduce from \eqref{eq.case1.2} and Cauchy-Schwarz inquality that
$$\normt{e^{-t}\SL_ku}\normt{e^tu}\geq C\beta_k^{1/3}\normt{e^tu}^2,\quad\text{for }k\geq k_0,$$
which is
\begin{equation*}
\normt{e^{-t}\SL_ku}\geq C\beta_k^{1/3}\normt{e^tu}, \quad k\geq k_0,
\end{equation*}
completing the proof of Theorem \ref{thm.case1}.
\end{proof}

\begin{rem}\label{rem.localization}\rm
There is a localization effect taking place in this case. 
We see in \eqref{eq.prop.case1} of Proposition \ref{prop.case1} that the coefficient of the term $\normt{e^{2t}g(e^t)^{1/2}u}^2$ has a factor $\kappa(e^{t_k})$, which is small if $e^{t_k}$ is taken very large (see Remark \ref{rem.kappak}). As a result, if we suppose $e^{t_k}$ large enough, this term is negligible, and the only bad term coming from the nonlocal operator that we need to control is 
$$2\beta_k^{2/3}k^{-2}\normt{e^{2t}g(e^t)\chi_-u}^2.$$ 
On the other hand, we can prove that there exists $\epsilon_2>0$ such that for all $e^{t_k}>\epsilon_2^{-1}$,
$$\forall k\geq2,\ \forall t\leq t_k-\frac{c_0}{2},\quad 2e^{2t}\big(\sigma(e^t)-\sigma(e^{t_k})\big)>2k^{-2} e^{4t}g(e^t)^2.$$
This implies that it suffices to take $k\geq2$ to absorb the remainders and thus Theorem \ref{thm.case1} holds for $k_0=2$ and $e^{t_k}>\epsilon_2^{-1}$.
Furthermore, we shall see that this localization effect does not present in Case 2 and Case 3. 
\end{rem}

%%%%%Case 2%%%%%
\subsubsection{Case 2: $e^{t_k}\in[\epsilon_1,\epsilon_0^{-1}]$}\label{sec.case2}
%Given $\epsilon_0\in(0,1)$, we suppose $c_0\in(0,1)$ to ensure Proposition \ref{prop.sigma.case2}, where the inequalities concerning $\sigma$ to be used in Case 2 are presented.

\begin{thm}\label{thm.case2}
Suppose $e^{t_k}\in[\epsilon_1,\epsilon_0^{-1}]$. Then there exist $C>0$, $k_0\geq1$ such that for all $k\geq k_0$, $\alpha\geq 8\pi$, $u\in C_0^\infty(\R)$,
\begin{equation}\label{est.case2}
\normt{e^{-t}\SL_ku}\geq C\beta_k^{2/3}\normt{e^tu},
\end{equation}
where $\SL_k$ is given in \eqref{eq.lk2} and $\beta_k$ is given in \eqref{eq.beta}.
\end{thm}

%\medskip
We present some inequalities concerning $\sigma$ that will be used in Case 2 in Proposition \ref{prop.sigma},(2). 
Note that they are similar to those in the Case 1 (given in Proposition \ref{prop.sigma},(1)). 

We use the metric $\Gamma$ and the multiplier $M_k$ in Definition \ref{def.metric}, \ref{def.multiplier}, and we use the notations $A_1,A_2,A_3,B^+,B^-$ in \eqref{def.A}, \eqref{def.B+}, \eqref{def.B-}. The estimate \eqref{eq.A1} for $A_{1}$ is valid with constant $C_1$ replaced by $C_2$ (which is given in \eqref{eq.case2sigma}) and the estimate \eqref{eq.A3} for $A_3$ holds in Case 2. 
For $A_2$, the estimate \eqref{eq.A2'} remains true:
\begin{align*}
|A_2|&=\big|2{\rm Re}\poscal{i\beta_k\gamma\Dk\gamma u}{m_{0,k}^wu}\big|\\
&\leq C\beta_k^{2/3}k^{-2}\normt{\gamma u}\big(\normt{\chi_0\gamma u}+\normt{\chi_0\gamma'u}\big)\\
&\qquad+C\beta_k^{2/3}k^{-2}\normt{\gamma u}\normt{\chi_0\gamma''u}+Ck^{-2}\normt{\gamma u}\normt{\chi_0u}.
\end{align*}
In the case where $e^{t_k}\in[\epsilon_1,\epsilon_0^{-1}]$, we have by \eqref{eq.gamma.d}:
$$|\chi_0(t-t_k)\gamma'(t)|\leq C\gamma(t),\quad |\chi_0(t-t_k)\gamma''(t)|\leq C\gamma(t),$$
for some $C$ depending on $\epsilon_1,\epsilon_0^{-1}$, so that 
\begin{equation}
|A_2|\leq C\beta_k^{2/3}k^{-2}\normt{\gamma u}^2+Ck^{-2}\normt{u}^2.
\end{equation}
For the terms $B^+, B^-$, we have by \eqref{eq.case2sigma.2},
\begin{align*}
B^+&=2{\rm Re}\poscal{\SL_ku}{m_{+,k}^wu}\\
&\geq 2c_2\beta_k^{2/3}\poscal{e^{2t}\sigma(e^{t_k})\chi_+^2u}{u}-C\normt{u}^2-C\normt{D_tu}^2-2\beta_k^{2/3}k^{-2}\normt{\gamma u}^2,
\end{align*}
\begin{align*}
B^-&=2{\rm Re}\poscal{\SL_ku}{m_{-,k}^wu}\\
&\geq 2c_2\beta_k^{2/3}\poscal{e^{2t}\sigma(e^{t})\chi_-^2u}{u}-C\normt{u}^2-C\normt{D_tu}^2-2\beta_k^{2/3}k^{-2}\normt{\gamma u}^2.
\end{align*}
Summarizing, we get that for all $k\geq1$, $\alpha\geq8\pi$, $u\in C_0^\infty(\R)$,
\begin{align}
\notag2{\rm Re}\poscal{&\SL_ku}{M_ku}=A+B^++B^-\\
\notag&\geq \frac{C_2}{2}\beta_k^{2/3}\normt{\chi_0u}^2+2c_2\beta_k^{2/3}\poscal{e^{2t}\sigma(e^{t_k})\chi_+^2u}{u}+2c_2\beta_k^{2/3}\poscal{e^{2t}\sigma(e^{t})\chi_-^2u}{u}\\
\notag&\quad-C\beta_k^{2/3}k^{-2}\normt{\gamma u}^2-C\normt{D_tu}^2-Ck^2\normt{u}^2-C\normt{e^{2t}u}^2,
\end{align}
so that the following proposition is proved:
\begin{prop}\label{prop.case2}
Suppose $e^{t_k}\in[\epsilon_1,\epsilon_0^{-1}]$. There exist $c>0,C>0$ such that for $k\geq1$, $\alpha\geq8\pi$, $u\in C_0^\infty(\R)$,
\begin{align}
\notag2{\rm Re}\poscal{\SL_ku}{M_ku}&\geq \beta_k^{2/3}\poscal{\Big(c\rho(t,t_k)-Ck^{-2}e^{4t}g(e^t)^2\Big)u}{u}\\
\label{eq.case2.1}&\quad-C\normt{D_tu}^2-Ck^2\normt{u}^2-C\normt{e^{2t}u}^2,
\end{align}
where $\SL_k$ is given in \eqref{eq.lk2}, $M_k$ in Definition \ref{def.multiplier}, $g$ in \eqref{eq.sigma}, $\beta_k$ in \eqref{eq.beta} and
\begin{equation}\label{eq.rho2}
\rho(t,t_k)=\chi_0(t-t_k)^2+e^{2t}\sigma(e^{t_k})\chi_+(t-t_k)^2+e^{2t}\sigma(e^{t})\chi_-(t-t_k)^2,
\end{equation}
with $\chi_0,\chi_\pm$ defined in \eqref{eq.chi} and $\sigma$ in \eqref{eq.sigma}.
\end{prop}

We have the following estimates for $\rho(t,t_k)$.
\begin{lem}\label{lem.rho2}
There exist $C_7,C_8>0$ such that for all $e^{t_k}\in[\epsilon_1,\epsilon_0^{-1}]$, $t\in\R$, 
\begin{equation}\label{eq.rho2.1}
\rho(t,t_k)\geq C_7e^{4t}g(e^t)^2,
\end{equation}
\begin{equation}\label{eq.rho2.2}
\rho(t,t_k)\geq C_8e^{2t},
\end{equation}
where $\rho(t,t_k)$ is given in \eqref{eq.rho2} and $g$ is given in \eqref{eq.sigma}.
\end{lem}
\begin{proof}[Proof of Lemma \ref{lem.rho2}]
Indeed, we have for $e^{t_k}\in[\epsilon_1,\epsilon_0^{-1}]$, $\sigma(e^{t_k})\geq \delta$.
 If $t$ is in the support of $\chi_0(\cdot-t_k)$, we have $|t-t_k|\leq c_0$, then
$$e^{4t}g(e^t)^2\leq C,\quad e^{2t}\leq \epsilon_0^{-2}e^{2c_0}.$$
 If $t$ is in the support of $\chi_+(\cdot-t_k)$, we have $t\geq t_k+ c_0/2$,
 $$e^{4t}g(e^t)^2\leq Ce^{2t}\sigma(e^{t_k}),\quad e^{2t}\leq \delta^{-1}e^{2t}\sigma(e^{t_k})$$
  If $t$ is in the support of $\chi_-(\cdot-t_k)$, we have $t\leq t_k-c_0/2$, then $\sigma(e^t)\geq \delta$,
  $$ e^{4t}g(e^t)^2\leq 3e^{2t}\sigma(e^{t})\ \text{\small (by \eqref{sigma.1})},\quad e^{2t}\leq \delta^{-1}e^{2t}\sigma(e^t).$$
Thus \eqref{eq.rho2.1}, \eqref{eq.rho2.2} are proved.
\end{proof}

\begin{proof}[Proof of Theorem \ref{thm.case2}]
\eqref{eq.case2.1} and \eqref{eq.rho2.1} imply that there exists $k_0\geq1$, for $k\geq k_0$,
\begin{align}
2{\rm Re}\poscal{\SL_ku}{M_ku}&\geq \frac{c}{2}\beta_k^{2/3}\poscal{\rho(t,t_k)u}{u}
\label{eq.case2.2}-C\normt{D_tu}^2-Ck^2\normt{u}^2-C\normt{e^{2t}u}^2.
\end{align}
Hence by choosing $C_9>0$ large enough, we get for $k\geq k_0$ 
\begin{equation*}
{\rm Re}\poscal{\SL_ku}{(C_9+2M_k)u}\geq \frac{c}{2}\poscal{\Big(\beta_k^{2/3}\rho(t,t_k)+D_t^2+k^2+e^{4t}\Big)u}{u}.
\end{equation*}
and in particular by \eqref{eq.rho2.2}
$${\rm Re}\poscal{\SL_ku}{(C_9+2M_k)u}\geq C\poscal{\beta_k^{2/3}e^{2t}u}{u},\quad k\geq k_0.$$
Finally we obtain the inequality \eqref{est.case2} by using the $L^2(\R;dt)$-boundedness of the operator $e^{t}(C_9+M_k)e^{-t}$ and Cauchy-Schwarz inequality.
\end{proof}

%%%%%%%%%%%%%%%%%%%%%%%%%%
%%%%%%%%%%%%%%%%%%%%%%%%%%
%%%%%%%%%%%%%%%%%%%%%%%%%%

%%%%%Case 3%%%%%
\subsection{Nontrivial cases, continued}\label{sec.nontrivial.c}

\subsubsection{Case 3: $\beta_k^{-1/4}< e^{t_k}<\epsilon_1$}\label{sec.case3}
We present in Proposition \ref{prop.sigma},(3) the inequalities about the function $\sigma$ to be used in this case.
Moreover, we assume $\alpha_0\geq8\pi$ such that the interval $(\beta_k^{-1/4},\epsilon_1)$ is not empty for any $k\geq1$, $\alpha\geq\alpha_0$.

\begin{thm}\label{thm.case3}
Suppose $e^{t_k}\in(\beta_k^{-1/4},\epsilon_1)$. Then there exist $C>0$, $k_0\geq3$ such that for all $k\geq k_0$, $\alpha\geq \alpha_0$, $u\in C_0^\infty(\R)$,
\begin{equation}\label{est.case3}
\normt{e^{-t}\SL_ku}\geq C\beta_k^{1/2}\normt{e^tu},
\end{equation}
where $\SL_k$ is given in \eqref{eq.lk2} and $\beta_k$ is given in \eqref{eq.beta}.
\end{thm}

\medskip

We shall modify the metric $\Gamma$ and the multiplier $M_k$ as follows.
\begin{defi}\label{def.case3}
\begin{equation}
\Gamma=|dt|^2+\frac{|d\tau|^2}{\tau^2+(\beta_ke^{4t_k})^{2/3}},\quad (t,\tau)\in \R_t\times\R_\tau,
\end{equation}
\begin{equation}
M_k=m_{0,k}^w+m_{+,k}^w+m_{-,k}^w,
\end{equation}
where
\begin{align*}
&m_{0,k}(t,\tau)=\chi_0(t-t_k)\sharp\psi\big((\beta_ke^{4t_k})^{-1/3}\tau\big)\sharp\chi_0(t-t_k),\\
&m_{+,k}(t,\tau)=-i(\beta_ke^{4t_k})^{-1/3}\chi_+(t-t_k)^2,\\
&m_{-,k}(t,\tau)=i(\beta_ke^{4t_k})^{-1/3}\chi_-(t-t_k)^2,
\end{align*}
with $\chi_0$, $\chi_\pm$, $\psi$ given in \eqref{eq.chi}, \eqref{eq.psi}.
\end{defi}
\begin{rem}\rm
Since we are in the region $e^{t_k}\in (\beta_k^{-1/4},\epsilon_1)$, we have
\begin{equation}\label{eq.lambda3}
\lambda_\Gamma=\big(\tau^2+(\beta_ke^{4t_k})^{2/3}\big)^{1/2}\geq (\beta_ke^{4t_k})^{1/3}\geq1,
\end{equation}
so that the metric $\Gamma$ verifies the uncertainty principle and moreover, $\Gamma$ is uniformly admissible (see Lemma \ref{lem.metric}).
%Note also by \eqref{eq.aba}
%\begin{equation}\label{eq.m0.3}
%m_{0,k}=\chi_0(t-t_k)^2\psi\big((\beta_ke^{4t_k})^{-1/3}\tau\big)+r_0,\qquad\text{with }r_0\in S((\beta_ke^{4t_k})^{-2/3},\Gamma),
%\end{equation}
%since $\chi_j(t-t_k)$, $\psi\big((\beta_ke^{4t_k})^{-1/3}\tau\big)$ are real-valued in $S(1,\Gamma)$.
Furthermore, the operator $M_k$ is bounded on $L^2(\R;dt)$.
\end{rem}

%Now let us compute 2Re$\poscal{\SL_ku}{M_ku}$. 
\begin{prop}\label{prop.case3}
Suppose $e^{t_k}\in(\beta_k^{-1/4},\epsilon_1)$. There exist $c,C>0$ such that for all $k\geq3$, $\alpha\geq\alpha_0$, $u\in C_0^\infty(\R)$,
\begin{align}
\notag2{\rm Re}\poscal{\SL_ku}{M_ku}&\geq \beta_k(\beta_ke^{4t_k})^{-1/3}\poscal{\Big(c\widetilde\rho(t,t_k)-Ck^{-2}e^{4t}g(e^t)^2\Big)u}{u}\\
\label{eq.prop.case3}&\qquad-C\normt{D_tu}^2-Ck^2\normt{u}^2-C\normt{e^{2t}u}^2,
\end{align}
where $\SL_k$ is given in \eqref{eq.lk2}, $M_k$ in Definition \ref{def.case3}, $g$ in \eqref{eq.sigma}, $\beta_k$ in \eqref{eq.beta} and
\begin{equation}\label{eq.rho3}
\widetilde\rho(t,t_k)=e^{4t_k}\chi_0(t-t_k)^2+e^{2t}\big(1-\sigma(e^{t})\big)\chi_+(t-t_k)^2+e^{2t}\big(1-\sigma(e^{t_k})\big)\chi_-(t-t_k)^2,
\end{equation}
with $\chi_0,\chi_\pm$ defined in \eqref{eq.chi} and $\sigma$ given in \eqref{eq.sigma}.
\end{prop}

\smallskip

\noindent{\it Proof of Proposition \ref{prop.case3}.}

\medskip

%%%%%Estimates for A
\noindent{\it Estimates for $2{\rm Re}\poscal{\SL_ku}{m_{0,k}^wu}$.}
\begin{align}
\notag A:=2\text{Re}\poscal{\SL_ku}{m_{0,k}^wu}&=2{\rm Re}\poscal{i\beta_ke^{2t}\big(\sigma(e^t)-\sigma(e^{t_k})\big)u}{m_{0,k}^wu}\\
\notag&\quad-2{\rm Re}\poscal{i\beta_k\gamma\Dk\gamma u}{m_{0,k}^wu}\\
\notag&\quad+2{\rm Re}\poscal{(D_t^2+k^2+\frac{1}{16}e^{4t})u}{m_{0,k}^wu}\\
\label{def.A.case3}&=:A_1+A_2+A_3.
\end{align}
%A_1
For $A_1$ in \eqref{def.A.case3}, we get a commutator
\begin{align*}
A_1&=2{\rm Re}\poscal{i\beta_ke^{2t}\big(\sigma(e^t)-\sigma(e^{t_k})\big)u}{\chi_0\psi\big((\beta_ke^{4t_k})^{-1/3}D_t\big)\chi_0u}\\
&=\poscal{\Big[\psi\big((\beta_ke^{4t_k})^{-1/3}D_t\big),i\beta_k\widetilde\chi_0e^{2t}\big(\sigma(e^t)-\sigma(e^{t_k})\big)\Big]\chi_0u}{\chi_0u},
\end{align*}
where $\widetilde\chi_0$ is given in \eqref{eq.tilde.chi}. % so that $\widetilde\chi_0\chi_0=\chi_0$ 
We know that, with $\Gamma$ given in Definition \ref{def.case3}
$$\Big[\psi\big((\beta_ke^{4t_k})^{-1/3}D_t\big),\ \underbrace{i\beta_k\widetilde\chi_0e^{2t}\big(\sigma(e^t)-\sigma(e^{t_k})\big)}_{\in S(\beta_ke^{4t_k},\Gamma)\text{ by \eqref{eq.case3sigma.1}}}\Big]=b_1^w+r_1^w,$$
where $b_1$ is a Poisson bracket and $r_1\in S(\beta_ke^{4t_k}\lambda_\Gamma^{-3},\Gamma)\subset S(1,\Gamma)$, with $\lambda_\Gamma$ given in \eqref{eq.lambda3} (see \eqref{eq.r}). More precisely,
\begin{align*}
b_1(t,\tau)&=\frac{1}{i}\Big\{\psi\big((\beta_ke^{4t_k})^{-1/3}\tau\big),i\beta_k\widetilde\chi_0e^{2t}\big(\sigma(e^t)-\sigma(e^{t_k})\big)\Big\}\\
&=\beta_k(\beta_ke^{4t_k})^{-1/3}\psi'\big((\beta_ke^{4t_k})^{-1/3}\tau\big)\frac{d}{dt}\Big(\widetilde\chi_0e^{2t}\big(\sigma(e^t)-\sigma(e^{t_k})\big)\Big)\\
&\quad\in S((\beta_ke^{4t_k})^{2/3},\Gamma)\subset S(\lambda_\Gamma^2,\Gamma).
\end{align*}
By \eqref{eq.tilde.chi}, \eqref{eq.psi} and \eqref{eq.case3sigma}, we have in the zone $\{|t-t_k|\leq 2c_0,|\tau|\leq(\beta_ke^{4t_k})^{1/3}\}$
\begin{align}
\notag b_1(t,\tau)&=\beta_k(\beta_ke^{4t_k})^{-1/3}\psi'\big((\beta_ke^{4t_k})^{-1/3}\tau\big)\frac{d}{dt}\Big(e^{2t}\big(\sigma(e^t)-\sigma(e^{t_k})\big)\Big)\\
\label{eq.b1.case3}&\geq \beta_k(\beta_ke^{4t_k})^{-1/3}\times\frac{1}{2}C_3 e^{4t_k}\geq \frac{C_3}{2}(\beta_ke^{4t_k})^{2/3}.
\end{align}
This implies for all $t,\tau\in\R$,
\begin{equation}\label{eq.fp.case3}
\frac{C_3}{2}(\beta_ke^{4t_k})^{2/3}\leq b_1(t,\tau)+\frac{C_3}{2}\tau^2+\tilde C_3(\beta_ke^{4t_k})^{2/3}\Big(1-\widetilde\chi_0\big(2(t-t_k)\big)\Big)\in S(\lambda_\Gamma^2,\Gamma),
\end{equation}
where $\tilde C_3=2\|b_1\|_{0,S((\beta_ke^{4t_k})^{2/3},\Gamma)}$.
Indeed, the function 
$$b_1(t,\tau)+\tilde C_3(\beta_ke^{4t_k})^{2/3}\Big(1-\widetilde\chi_0\big(2(t-t_k)\big)\Big)\geq \frac{C_3}{2}(\beta_ke^{4t_k})^{2/3}$$
for all $t\in\R$ and $|\tau|\leq(\beta_ke^{4t_k})^{1/3}$, 
and it is non-negative for all $t,\tau\in\R$ ; if $|\tau|\geq(\beta_ke^{4t_k})^{1/3}$, then $\tau^2\geq(\beta_ke^{4t_k})^{2/3}$, which proves the inequality in \eqref{eq.fp.case3}. Moreover, each term in the right hand side of \eqref{eq.fp.case3} is in $S(\lambda_\Gamma^2,\Gamma)$.
The Fefferman-Phong inequality (Proposition \ref{prop.fp}) implies
$$b_1(t,\tau)^w+\frac{C_3}{2}D_t^2+\tilde C_3(\beta_ke^{4t_k})^{2/3}\Big(1-\widetilde\chi_0\big(2(t-t_k)\big)\Big)\geq \frac{C_3}{2}(\beta_ke^{4t_k})^{2/3}-C'.$$
Applying to $\chi_0u$, we get
\begin{align*}
A_1+\frac{C_3}{2}\poscal{D_t^2\chi_0u}{\chi_0u}&=\poscal{\big(\frac{C_3}{2}D_t^2+b_1^w\big)\chi_0u}{\chi_0u}+\poscal{r_1^w\chi_0u}{\chi_0u}\\
&\geq \frac{C_3}{2}(\beta_ke^{4t_k})^{2/3}\normt{\chi_0u}^2-C''\normt{\chi_0u}^2.
\end{align*}
Hence we get the estimate for $A_1$:
\begin{align}
\label{eq.A1.case3}A_1\geq \frac{C_3}{2}(\beta_ke^{4t_k})^{2/3}\normt{\chi_0u}^2-C\normt{D_tu}^2-C\normt{u}^2.
\end{align}

%A_2
For $A_2$ defined in \eqref{def.A.case3}, we have
\begin{align}
\notag A_2&=-2{\rm Re}\poscal{i\beta_k\gamma\Dk\gamma u}{m_{0,k}^wu}\\
\notag&=-2{\rm Re}\poscal{i\beta_k \Dk\gamma u}{m_{0,k}^w\gamma u}
-2{\rm Re}\poscal{i\beta_k \Dk \gamma u}{\big[\gamma,m_{0,k}^w\big]u}\\
\label{def.A2.case3}&=:A_{21}+A_{22}.
\end{align}
%A_{21}
For $A_{21}$ in \eqref{def.A2.case3}, since $i\Dk$ is skew-adjoint and $m_{0,k}^w$ is self-adjoint, we get
\begin{equation*}
A_{21}=i\beta_k\poscal{\big[\Dk,m_{0,k}^w\big]\gamma u}{\gamma u}.
\end{equation*}
Noting that $\Dk$ commutes with $\psi((\beta_ke^{4t_k})^{-1/3}D_t)$, we have
\begin{align*}
\notag&\big[\Dk,m_{0,k}^w\big]=\big[\Dk,\chi_0\psi\big((\beta_ke^{4t_k})^{-1/3}D_t\big)\chi_0\big]\\
&=\big[\Dk,\chi_0\big]\psi\big((\beta_ke^{4t_k})^{-1/3}D_t\big)\chi_0+\chi_0\psi\big((\beta_ke^{4t_k})^{-1/3}D_t\big)\big[\Dk,\chi_0\big].
\end{align*}
By using the method that is used in Case 1, we can get 
\begin{align*}
\big|\poscal{\big[\Dk,\chi_0\big]\psi\big((\beta_ke^{4t_k})^{-1/3}D_t\big)\chi_0\gamma u}{\gamma u}\big|\leq C(\beta_ke^{4t_k})^{-1/3}k^{-2}\normt{\chi_0\gamma u}\normt{\gamma u},\\
\big|\poscal{\chi_0\psi\big((\beta_ke^{4t_k})^{-1/3}D_t\big)\big[\Dk,\chi_0\big]\gamma u}{\gamma u}\big|\leq C(\beta_ke^{4t_k})^{-1/3}k^{-2}\normt{\gamma u}\normt{\chi_0\gamma u},
\end{align*}
so that 
\begin{equation}\label{eq.A21.case3}
|A_{21}|\leq C\beta_k(\beta_ke^{4t_k})^{-1/3}k^{-2}\normt{\chi_0\gamma u}\normt{\gamma u}.
\end{equation}
For $A_{22}$ in \eqref{def.A2.case3}, we have
\begin{align*}
A_{22}&=-2{\rm Re}\poscal{i\beta_k \Dk\gamma u}{\big[\gamma ,m_{0,k}^w\big]u}\\
&=-2{\rm Re}\poscal{i\beta_k \Dk \gamma u}{\chi_0\Big[\widetilde\chi_0\gamma,\psi\big((\beta_ke^{4t_k})^{-1/3}D_t\big)\Big]\chi_0u},
\end{align*}
where $\tilde\chi_0$ is given in \eqref{eq.tilde.chi}.
%Remark here that we need to put a cutoff function $\widetilde\chi_0$ in the commutator to localize the symbol $\gamma$ near $t_k$, in order that the factor $e^{-4t_k}$ in the remainder coming from the symbolic calculus can be compensated.
Since $\widetilde\chi_0\gamma=\widetilde\chi_0(t-t_k)e^{2t}g(e^t)\in S(e^{2t_k},\Gamma)$, we get
$$\Big[\widetilde\chi_0\gamma,\psi\big((\beta_ke^{4t_k})^{-1/3}D_t\big)\Big]=b_2^w+r_2^w,$$
where $b_2\in S(e^{2t_k}\lambda_\Gamma^{-1},\Gamma)$ is a Poisson bracket and $r_2$ belongs to $S(e^{2t_k}\lambda_\Gamma^{-3},\Gamma)\subset S(e^{2t_k}(\beta_ke^{4t_k})^{-1},\Gamma)$, with $\lambda_\Gamma$ given in \eqref{eq.lambda3} (see \eqref{eq.r}). We compute $b_2$ as follows
\begin{align*}
b_2&=\frac{1}{i}\Big\{\widetilde\chi_0\gamma,\psi\big((\beta_ke^{4t_k})^{-1/3}\tau\big)\Big\}\\
&=-\frac{1}{i}(\beta_ke^{4t_k})^{-1/3}\psi'\big((\beta_ke^{4t_k})^{-1/3}\tau\big)(\widetilde\chi_0\gamma)'(t)\quad\in S(e^{2t_k}(\beta_ke^{4t_k})^{-1/3},\Gamma)\\
&=-\frac{1}{i}(\beta_ke^{4t_k})^{-1/3}\psi'\big((\beta_ke^{4t_k})^{-1/3}\tau\big)\sharp (\widetilde\chi_0\gamma)'(t)+b_3+r_3,
\end{align*}
%$$\text{where}\quad \varphi_1(t)=2\widetilde\chi_0-\frac{1}{4}e^{2t}\widetilde\chi_0+\widetilde\chi_0',$$
where $b_3\in S(e^{2t_k}(\beta_ke^{4t_k})^{-2/3},\Gamma)$ is a Poisson bracket and $r_3\in S(e^{2t_k}(\beta_ke^{4t_k})^{-1},\Gamma)$. We continue to expand $b_3$
\begin{align*}
b_3&=-\frac{1}{2i}\Big\{-\frac{1}{i}(\beta_ke^{4t_k})^{-1/3}\psi'\big((\beta_ke^{4t_k})^{-1/3}\tau\big), (\widetilde\chi_0\gamma)'(t)\Big\}\\
&=-\frac{1}{2}(\beta_ke^{4t_k})^{-2/3}\psi''\big((\beta_ke^{4t_k})^{-1/3}\tau\big)(\widetilde\chi_0\gamma)''(t)\\
&=-\frac{1}{2}(\beta_ke^{4t_k})^{-2/3}\psi''\big((\beta_ke^{4t_k})^{-1/3}\tau\big)\sharp(\widetilde\chi_0\gamma)''(t)+r_4,
\end{align*}
where $r_4\in S(e^{2t_k}(\beta_ke^{4t_k})^{-1},\Gamma)$. %and
%\begin{align*}
%\varphi_2(t)&=\varphi_1(t)^2+\varphi_1'(t)\\
%&=\big(2\widetilde\chi_0-\frac{1}{4}e^{2t}\widetilde\chi_0+\widetilde\chi_0'\big)^2+2\widetilde\chi_0'-\frac{1}{2}e^{2t}\widetilde\chi_0-\frac{1}{4}e^{2t}\widetilde\chi_0'+\widetilde\chi_0''.
%\end{align*} 
Thus we get for $w\in C_0^\infty(\R)$,
\begin{align*}
\big[\widetilde\chi_0\gamma,&\psi\big((\beta_ke^{4t_k})^{-1/3}D_t\big)\big]w
= -\frac{1}{i}(\beta_ke^{4t_k})^{-1/3}\psi'\big((\beta_ke^{4t_k})^{-1/3}D_t\big)(\widetilde\chi_0\gamma)'(t) w\\
& \quad-\frac{1}{2}(\beta_ke^{4t_k})^{-2/3}\psi''\big((\beta_ke^{4t_k})^{-1/3}D_t\big)(\widetilde\chi_0\gamma)''(t) w+(r_2^w+r_3^w+r_4^w)w,
%&=:R_1w+R_2w+R_3w,
\end{align*}
where $r_2,r_3,r_4\in S(e^{2t_k}(\beta_ke^{4t_k})^{-1},\Gamma)$.
Using the boundedness of $\psi'$ and $\psi''$, we obtain for $w\in C_0^\infty(\R)$,
\begin{align*}
\normt{\big[\widetilde\chi_0\gamma,&\psi\big((\beta_ke^{4t_k})^{-1/3}D_t\big)\big]w}
\leq C(\beta_ke^{4t_k})^{-1/3}\normt{(\widetilde\chi_0\gamma)'w}\\
&+C(\beta_ke^{4t_k})^{-2/3}\normt{(\widetilde\chi_0\gamma)''w}+ Ce^{2t_k}(\beta_ke^{4t_k})^{-1}\normt{w}.
\end{align*}
%The operators $R_1,R_2,R_3$ can be estimated in the following way
%\begin{align*}
%\normt{R_1w}&\leq C(\beta_ke^{4t_k})^{-1/3}\normt{\varphi_1(t)\gamma w},\\
%\normt{R_2w}&\leq C(\beta_ke^{4t_k})^{-2/3}\normt{\varphi_2(t)\gamma w},\\
%\normt{R_3w}&\leq Ce^{2t_k}(\beta_ke^{4t_k})^{-1}\normt{w}.
%\end{align*}
%A_{22}
Now the term $A_{22}$ defined in \eqref{def.A2.case3} can be estimated as follows:
\begin{align}
\notag\big|A_{22}\big|&=\big|2{\rm Re}\poscal{i\beta_k \chi_0\Dk\gamma u}{\Big[\widetilde\chi_0\gamma,\psi\big((\beta_ke^{4t_k})^{-1/3}D_t\big)\Big]\chi_0u}\big|\\
%&=\big|2{\rm Re}\poscal{i\beta_k\Dk \gamma u}{\chi_0\big(R_1+R_2+R_3\big)\chi_0u}\big|\\
\notag&\leq 2\beta_k\normt{\chi_0\Dk\gamma u}\normt{\Big[\widetilde\chi_0\gamma,\psi\big((\beta_ke^{4t_k})^{-1/3}D_t\big)\Big]\chi_0u}\\
\notag&\leq 2\beta_k\normt{\chi_0\Dk\gamma u}\times\Big(C(\beta_ke^{4t_k})^{-1/3}\normt{(\widetilde\chi_0\gamma)'\chi_0u}\\
\notag&\qquad+C(\beta_ke^{4t_k})^{-2/3}\normt{(\widetilde\chi_0\gamma)''\chi_0u}+ Ce^{2t_k}(\beta_ke^{4t_k})^{-1}\normt{\chi_0u}\Big)\\
\notag&\leq C\beta_k(\beta_ke^{4t_k})^{-1/3}k^{-2}\normt{\gamma u}\normt{\gamma'\chi_0u}+C\beta_k(\beta_ke^{4t_k})^{-2/3}k^{-2}\normt{\gamma u}\normt{\gamma''\chi_0u}\\
\label{eq.A22.case3}&\qquad+Ck^{-2}\normt{g(e^t) u} \normt{\chi_0u},
\end{align}
where in the last inequality we use the following 
$$\normt{\chi_0\Dk\gamma u}=\normt{\chi_0e^{2t}\underbrace{e^{-2t}\Dk e^{2t}}_{\text{has norm $\leq 3k^{-2}$}\atop\text{since $k\geq3$}}g(e^t)u}\leq Ce^{2t_k}k^{-2}\normt{g(e^t)u}$$
%We have for $j=1,2$,
%%R_1,R_2
%\begin{align*}
%&\big|2{\rm Re}\poscal{i\beta_k\Dk \gamma u}{\chi_0R_j\chi_0u}\big|\\
%&\quad\leq2\beta_k\normt{ \Dk \gamma u}\normt{\chi_0R_j\chi_0u}\\
%&\quad\leq 2\beta_kk^{-2}\normt{\gamma u}\times C(\beta_ke^{4t_k})^{-j/3}\normt{\varphi_j(t)\gamma \chi_0u}\\
%&\quad\leq C\beta_k(\beta_ke^{4t_k})^{-j/3}k^{-2}\normt{\gamma u}\normt{\varphi_j(t)\gamma u}.
%\end{align*}
%%R_3
%For $j=3$, we have
%\begin{align*}
%&\big|2{\rm Re}\poscal{i\beta_k \Dk\gamma u}{\chi_0R_3\chi_0u}\big|\\
%&\quad=\big|2{\rm Re}\poscal{i\beta_k\chi_0 \Dk\gamma u}{R_3\chi_0u}\big|\\
%&\quad=\big|2{\rm Re}\poscal{i\beta_k\big(e^{2t}\chi_0\big) \big(e^{-2t}\Dk e^{2t}\big)\big(g(e^t)u\big)}{R_3\chi_0u}\big|\\
%&\quad\leq2\beta_k\normt{ \big(e^{2t}\chi_0\big)\big(e^{-2t}\Dk e^{2t}\big)g(e^t)u}\normt{R_3\chi_0u}\\
%&\quad\leq 2\beta_k\times Ce^{2t_k}k^{-2}\normt{g(e^t)u}\times Ce^{2t_k}(\beta_ke^{4t_k})^{-1}\normt{\chi_0u}\\
%&\quad\leq Ck^{-2}\normt{u}^2,
%\end{align*}
%where we use the fact that the ${\mathcal L}(L^2(\R;dt))$-norm of the operator $e^{-2t}\Dk e^{2t}$ is bounded by $3k^{-2}$, for $k\geq3$. 
(see Lemma \ref{lem.Dk3}). 
%A_{22}
%We obtain
%\begin{align*}
%\big|A_{22}\big|&\leq C\beta_k(\beta_ke^{4t_k})^{-1/3}k^{-2}\normt{\gamma u}\normt{\varphi_1(t)\gamma u}\\
%&\quad+C\beta_k(\beta_ke^{4t_k})^{-2/3}k^{-2}\normt{\gamma u}\normt{\varphi_2(t)\gamma u}+Ck^{-2}\normt{u}^2.
%\end{align*}
%Given $\epsilon_0$, we can find a constant $C>0$ such that for all $e^{t_k}<\epsilon_0$, $t\in \R$,
%$$\big|\varphi_1(t)\big|\leq C\quad\text{and}\quad \big|\varphi_2(t)\big|\leq C,$$
%then since $\beta_ke^{4t_k}\geq1$, we get
%\begin{align}
%\label{eq.A22.3}\big|A_{22}\big|&\leq C\beta_k(\beta_ke^{4t_k})^{-1/3}k^{-2}\normt{\gamma u}^2+Ck^{-2}\normt{u}^2.
%\end{align}

%A_{2}
It follows from \eqref{def.A2.case3}, \eqref{eq.A21.case3} and \eqref{eq.A22.case3} that
\begin{align}
\notag\big|A_2\big|&\leq C\beta_k(\beta_ke^{4t_k})^{-1/3}k^{-2} \normt{\gamma u}\big(\normt{\chi_0\gamma u}+\normt{\chi_0\gamma' u}\big)\\
&\quad+C\beta_k(\beta_ke^{4t_k})^{-2/3}k^{-2}\normt{\gamma u}\normt{\chi_0\gamma''u}
\label{eq.A2.case3'}+Ck^{-2}\normt{u}^2.
\end{align}
From \eqref{eq.gamma.d} we deduce that for $e^{t_k}<\epsilon_1$,
$$|\chi_0(t-t_k)\gamma'(t)|\leq C\gamma(t),\quad |\chi_0(t-t_k)\gamma''(t)|\leq C\gamma(t),$$
with $C$ depending only on $\epsilon_1$, so that
\begin{align}
\big|A_2\big|&\leq C\beta_k(\beta_ke^{4t_k})^{-1/3}k^{-2} \normt{\gamma u}^2
\label{eq.A2.case3}+Ck^{-2}\normt{u}^2.
\end{align}

%A_3
The estimate for $A_3$ defined in \eqref{def.A.case3} is the same as that in Case 1
%\begin{align}
%\notag A_3&=2{\rm Re}\poscal{D_t^2u}{m_{0,k}^wu}+2{\rm Re}\poscal{k^2u}{m_{0,k}^wu}+\frac{1}{8}{\rm Re}\poscal{e^{4t}u}{m_{0,k}^wu}\\
%\label{def.case3.A3}&=:A_{31}+A_{32}+A_{33},
%\end{align}
%%A_{31} %A_{32}
%\begin{align}
%\label{eq.A31.3}\big|A_{31}\big|&=\big|2{\rm Re}\poscal{D_t^2u}{m_{0,k}^wu}\big|\leq C\normt{D_tu}^2+C\normt{u}^2,\\
%\label{eq.A32.3}\big|A_{32}\big|&=\big|2{\rm Re}\poscal{k^2u}{m_{0,k}^wu}\big|\leq Ck^2\normt{u}^2.
%\end{align}
%%A_{33}
%For $A_{33}$ in \eqref{def.case3.A3},% using $\widetilde\chi_0\chi_0=\chi_0$, 
%we have
%\begin{align}\label{eq.A33.3}
%|8A_{33}|
%&=\big|{\rm Re}\poscal{e^{2t}u}{\underbrace{e^{2t}\chi_0\psi\big((\beta_ke^{4t_k})^{-1/3}D_t\big)\chi_0e^{-2t}}_{\text{bounded on $L^2(\R;dt)$}}e^{2t}u}\big|\leq C\normt{e^{2t}u}^2.
%\end{align}
%so that
%Since $\psi((\beta_ke^{4t_k})^{-1/3}\tau)\in S(1,\Gamma)$ and $\widetilde\chi_0e^{2t}\in S(e^{2t_{k}},\Gamma)$, the double commutator in the last equality has a symbol in $S(e^{4t_k}\lambda_\Gamma^{-2},\Gamma)\subset S(e^{4t_k}(\beta_ke^{4t_k})^{-2/3},\Gamma)$, with $\lambda_\Gamma$ given in \eqref{eq.lambda3}. Hence
%\begin{equation}
%\big|A_{33}\big|%\leq C\normt{e^{2t}\chi_0u}^2+C(\beta_ke^{4t_k})^{-2/3}e^{4t_k}\normt{\chi_0u}^2
%\leq C\normt{e^{2t}u}^2.
%\end{equation}
%A_{3}
%By \eqref{def.case3.A3}, \eqref{eq.A31.3}, \eqref{eq.A32.3} and \eqref{eq.A33.3} we get
\begin{equation}\label{eq.A3.case3}
\big|A_3\big|\leq C\normt{D_tu}^2+Ck^2\normt{u}^2+C\normt{e^{2t}u}^2.
\end{equation}

%A
We deduce from \eqref{def.A.case3}, \eqref{eq.A1.case3}, \eqref{eq.A2.case3} and \eqref{eq.A3.case3} that
\begin{align}
\notag A&\geq \frac{C_3}{2}(\beta_ke^{4t_k})^{2/3}\normt{\chi_0u}^2
-C\beta_k(\beta_ke^{4t_k})^{-1/3}k^{-2} \normt{\gamma u}^2-Ck^{-2}\normt{u}^2\\
\label{eq.A.case3}&\qquad-C\normt{D_tu}^2-Ck^2\normt{u}^2-C\normt{e^{2t}u}^2.
\end{align}

\medskip

%%%%%Estimate for B_+%%%%%
\noindent{\it Estimates for $2{\rm Re}\poscal{\SL_ku}{m_{+,k}^wu}$.}
Recall $m_{+,k}^w=-i(\beta_ke^{4t_k})^{-1/3}\chi_+(t-t_k)^2$.
\begin{align}
\notag B^+&:=2{\rm Re}\poscal{\SL_ku}{m_{+,k}^wu}=2{\rm Re}\poscal{\SL_ku}{-i(\beta_ke^{4t_k})^{-1/3}\chi_+^2u}\\
\notag&=2{\rm Re}\poscal{i\beta_ke^{2t}\big(\sigma(e^{t})-\sigma(e^{t_k})\big)u}{-i(\beta_ke^{4t_k})^{-1/3}\chi_+^2u}\\
\notag&\quad-2{\rm Re}\poscal{i\beta_k\gamma\Dk\gamma u}{-i(\beta_ke^{4t_k})^{-1/3}\chi_+^2u}\\
\notag&\quad+2{\rm Re}\poscal{D_t^2u}{-i(\beta_ke^{4t_k})^{-1/3}\chi_+^2u}\\
\label{def.B+.case3}&=B^{+}_{1}+B^{+}_{2}+B^{+}_{3}.
\end{align}
%B_{+1}
Recall that the support of $\chi_+(t-t_k)$ is included in $\{t-t_k\geq c_0/2\}$ and \eqref{eq.case3sigma.2}. Thus
\begin{align}
\notag B^{+}_{1}&=2\beta_k(\beta_ke^{4t_k})^{-1/3}\poscal{e^{2t}\big(\sigma(e^{t_k})-\sigma(e^t)\big)u}{\chi_+^2u}\\
\label{eq.B+1.case3}&\geq 2c_3\beta_k(\beta_ke^{4t_k})^{-1/3}\poscal{e^{2t}\big(1-\sigma(e^{t})\big)\chi_+^2u}{u}.
\end{align}
%B_{+2}
For $B^{+}_{2}$ in \eqref{def.B+.case3} we get
\begin{align}
\notag\big|B^{+}_{2}\big|&=\big|2\beta_k(\beta_ke^{4t_k})^{-1/3}{\rm Re}\poscal{\gamma\Dk\gamma u}{\chi_+^2u}\big|\\
\label{eq.B+2.case3}&\leq 2\beta_k(\beta_ke^{4t_k})^{-1/3}k^{-2}\normt{\gamma u}^2.
\end{align}
%{\color{blue}We have also
%\begin{align*}
% B^{+}_{2}&=\beta_k(\beta_ke^{4t_k})^{-1/3}\poscal{\big[\chi_+,[\chi_+,\Dk]\big]\gamma u}{\gamma u}\\
%& \qquad+2\beta_k(\beta_ke^{4t_k})^{-1/3}{\rm Re}\underbrace{\poscal{\Dk\chi_+\gamma u}{\chi_+\gamma u}}_{\geq0}.
%\end{align*}
%The double commutator $\big[\chi_+,[\chi_+,\Dk]\big]$ has kernel
%$$\frac{1}{2k}e^{-k|t-s|}\big[\chi_+(t-t_k)-\chi_+(s-t_k)\big]^2,$$
%which vanishes if $\max(t,s)\leq t_k+c_0/2$ and also if $\min(t,s)\geq t_k+c_0$. Then we have
%\begin{align*}
%&\big|\poscal{\big[\chi_+,[\chi_+,\Dk]\big]\gamma u}{\gamma u}\big|\\
%&=\Big|\iint_{t-t_k\leq c_0}\frac{1}{2k}e^{-k|t-s|}\big[\chi_+(t-t_k)-\chi_+(s-t_k)\big]^2(\gamma u)(s)\overline{(\gamma u)(t)}dtds\\
%&\quad+\iint_{t-t_k\geq c_0\atop s-t_k\leq c_0}\frac{1}{2k}e^{-k|t-s|}\big[\chi_+(t-t_k)-\chi_+(s-t_k)\big]^2(\gamma u)(s)\overline{(\gamma u)(t)}dtds\Big|\\
%&\leq e^{2t_k}\iint_{t-t_k\leq c_0}\frac{1}{2k}e^{-k|t-s|}\Big(e^{2s}g(e^s)|u(s)|\Big)\Big(g(e^{t})| u(t)|\Big)dtds\\
%&\quad+e^{2t_k}\iint_{t-t_k\geq c_0\atop s-t_k\leq c_0}\frac{1}{2k}e^{-k|t-s|}\Big(g(e^s) |u(s)|\Big)\Big(e^{2t}g(e^t) |u(t)|\Big)dtds\\
%&\leq e^{2t_k}k^{-2}\normt{e^{2t}g(e^t)u}\normt{g(e^t)u},
%\end{align*}
%so that
%$$B_{+2}\geq -\beta_ke^{2t_k}(\beta_ke^{4t_k})^{-1/3}k^{-2}\normt{e^{2t}g(e^t)u}\normt{g(e^t)u}.$$
%}
%For $B_{+2}$ in \eqref{def.B+.case3} we have
%\begin{align}
%\notag\big|B_{+2}\big|&=\big|2\beta_k(\beta_ke^{4t_k})^{-1/3}{\rm Re}\poscal{\gamma\Dk\gamma u}{\chi_+^2u}\big|\\
%\label{eq.B+2.3}&\leq 2\beta_k(\beta_ke^{4t_k})^{-1/3}k^{-2}\normt{\gamma u}^2.
%\end{align}
%B_{+3}
For $B^{+}_{3}$ in \eqref{def.B+.case3} we have
\begin{align}
\notag\big|B^{+}_{3}\big|&=\big|(\beta_ke^{4t_k})^{-1/3}\poscal{\big[D_t^2,-i\chi_+^2\big]u}{u}\big|\\
\label{eq.B+3.case3}&\leq C(\beta_ke^{4t_k})^{-1/3}\big(\normt{u}^2+\normt{\p_tu}\normt{u}\big)\leq C\normt{D_tu}^2+C\normt{u}^2.
\end{align}

%B_+
We get from \eqref{def.B+.case3}, \eqref{eq.B+1.case3}, \eqref{eq.B+2.case3} and \eqref{eq.B+3.case3} that
\begin{align}
\notag B^+&\geq 2c_3\beta_k(\beta_ke^{4t_k})^{-1/3}\poscal{e^{2t}\big(1-\sigma(e^{t})\big)\chi_+^2u}{u}\\
\label{eq.B+.case3}&\qquad-C\normt{u}^2-C\normt{D_tu}^2-2\beta_k(\beta_ke^{4t_k})^{-1/3}k^{-2}\normt{\gamma u}^2.
\end{align}
%%%%%%

\medskip

%%%%%Estimate for B_-%%%%%
\noindent{\it Estimates for $2{\rm Re}\poscal{\SL_ku}{m_{-,k}^wu}$.}
Recall that $m_{-,k}^w=i(\beta_ke^{4t_k})^{-1/3}\chi_-(t-t_k)^2$.
\begin{align}
\notag B^-&:=2{\rm Re}\poscal{\SL_ku}{m_{-,k}^wu}=2{\rm Re}\poscal{\SL_ku}{i(\beta_ke^{4t_k})^{-1/3}\chi_-^2u}\\
\notag&=2{\rm Re}\poscal{i\beta_ke^{2t}\big(\sigma(e^t)-\sigma(e^{t_k})\big)u}{i(\beta_ke^{4t_k})^{-1/3}\chi_-^2u}\\
\notag&\quad-2{\rm Re}\poscal{i\beta_k\gamma\Dk\gamma u}{i(\beta_ke^{4t_k})^{-1/3}\chi_-^2u}\\
\notag&\quad+2{\rm Re}\poscal{D_t^2u}{i(\beta_ke^{4t_k})^{-1/3}\chi_-^2u}\\
\label{def.B-.case3}&=B^{-}_{1}+B^{-}_{2}+B^{-}_{3}.
\end{align}
%B_{-1}
Recall that the support of $\chi_-(t-t_k)$ is included in $\{t-t_k\leq -c_0/2\}$ and \eqref{eq.case3sigma.2}. Thus
\begin{align}
\notag B^{-}_{1}&=2\beta_k(\beta_ke^{4t_k})^{-1/3}\poscal{e^{2t}\big(\sigma(e^{t})-\sigma(e^{t_k})\big)u}{\chi_-^2u}\\
\label{eq.B-1.case3}&\geq 2c_3\beta_k(\beta_ke^{4t_k})^{-1/3}\poscal{e^{2t}\big(1-\sigma(e^{t_k})\big)\chi_-^2u}{u}.
\end{align}
%B_{-2}
For $B^{-}_{2}$ in \eqref{def.B-.case3} we have
\begin{align}
\notag\big|B^{-}_{2}\big|&=\big|2\beta_k(\beta_ke^{4t_k})^{-1/3}{\rm Re}\poscal{\gamma\Dk\gamma u}{\chi_-^2u}\big|\\
\label{eq.B-2.case3}&\leq 2\beta_k(\beta_ke^{4t_k})^{-1/3}k^{-2}\normt{\gamma u}^2.
\end{align}
%{\color{blue}Or we have as in Case 1:
%$$B^{-}_{2}\geq -\beta_ke^{2t_k}(\beta_ke^{4t_k})^{-1/3}k^{-2}\normt{e^{2t}g(e^t)u}\normt{g(e^t)u}-2\beta_k(\beta_ke^{4t_k})^{-1/3}k^{-2}\normt{\chi_-\gamma u}^2.$$}
%B_{-3}
For $B^{-}_{3}$ in \eqref{def.B-.case3}, we have
\begin{align}
\notag\big|B^{-}_{3}\big|&=\big|(\beta_ke^{4t_k})^{-1/3}\poscal{\big[D_t^2,i\chi_-^2\big]u}{u}\big|\\
\label{eq.B-3.case3}&\leq C(\beta_ke^{4t_k})^{-1/3}\big(\normt{u}^2+\normt{\p_tu}\normt{u}\big)\leq C\normt{u}^2+C\normt{D_tu}^2.
\end{align}

%B_-
We get from \eqref{def.B-.case3}, \eqref{eq.B-1.case3}, \eqref{eq.B-2.case3} and \eqref{eq.B-3.case3} that
\begin{align}
\notag B^-&\geq 2c_3\beta_k(\beta_ke^{4t_k})^{-1/3}\poscal{e^{2t}\big(1-\sigma(e^{t_k})\big)\chi_-^2u}{u}\\
\label{eq.B-.case3}&\qquad-C\normt{u}^2-C\normt{D_tu}^2-2\beta_k(\beta_ke^{4t_k})^{-1/3}k^{-2}\normt{\gamma u}^2.
\end{align}

%%%%%

%%%%%Conclusion%%%%%
\begin{proof}[End of the proof of Proposition \ref{prop.case3}]
By \eqref{eq.A.case3}, \eqref{eq.B+.case3}, \eqref{eq.B-.case3} and the definition \ref{def.case3} of $M_k$, we get
\begin{align}
\notag2{\rm Re}\poscal{\SL_ku}{M_ku}&=A+B^++B^-\\
\notag&\geq \frac{C_3}{2}(\beta_ke^{4t_k})^{2/3}\normt{\chi_0u}^2\\
\notag&\qquad+2c_3\beta_k(\beta_ke^{4t_k})^{-1/3}\poscal{e^{2t}\big(1-\sigma(e^{t})\big)\chi_+^2u}{u}\\
\notag&\qquad+2c_3\beta_k(\beta_ke^{4t_k})^{-1/3}\poscal{e^{2t}\big(1-\sigma(e^{t_k})\big)\chi_-^2u}{u}\\
\notag&\qquad-C\beta_k(\beta_ke^{4t_k})^{-1/3}k^{-2} \normt{\gamma u}^2\\
\notag&\qquad-C\normt{D_tu}^2-Ck^2\normt{u}^2-C\normt{e^{2t}u}^2\\
\notag&\geq \beta_k(\beta_ke^{4t_k})^{-1/3}\poscal{\Big(c\widetilde\rho(t,t_k)-Ck^{-2}e^{4t}g(e^t)^2\Big)u}{u}\\
%\label{eq.case3.1}
\notag&\qquad-C\normt{D_tu}^2-Ck^2\normt{u}^2-C\normt{e^{2t}u}^2,
\end{align}
where
%\begin{equation*}%\label{eq.rho3}
$\widetilde\rho(t,t_k)$ is given in \eqref{eq.rho3}. %=e^{4t_k}\chi_0(t-t_k)^2+e^{2t}\big(1-\sigma(e^{t})\big)\chi_+(t-t_k)^2+e^{2t}\big(1-\sigma(e^{t_k})\big)\chi_-(t-t_k)^2.
%\end{equation*}
This completes the proof of \eqref{eq.prop.case3} in Proposition \ref{prop.case3}.
\end{proof}

We have the following estimates for $\widetilde\rho(t,t_k)$.
\begin{lem}\label{lem.rho3}
There exist $C_{10},C_{11}>0$ such that for all $e^{t_k}<\epsilon_1$, $t\in\R$, $\alpha\geq\alpha_0$, $k\geq1$,
\begin{equation}\label{eq.rho3.1}
\widetilde\rho(t,t_k)\geq C_{10}e^{4t}g(e^{t})^2,
\end{equation}
\begin{equation}\label{eq.rho3.2}
\beta_k(\beta_ke^{4t_k})^{-1/3}\widetilde\rho(t,t_k)+k^2\geq C_{11}\beta_k^{1/2}e^{2t},
\end{equation}
where $\widetilde\rho$ is given in \eqref{eq.rho3}, $g$ is given in \eqref{eq.sigma} and $\beta_k$ is given in \eqref{eq.beta}.
\end{lem}
\begin{proof}[Proof of Lemma \ref{lem.rho3}]
Suppose $e^{t_k}<\epsilon_1$, then $1-\sigma(e^{t_k})\geq \delta e^{2t_k}$. If $t$ is in the support of $\chi_0(\cdot-t_k)$, i.e. $|t-t_k|\leq c_0$, we have
$$e^{4t}g(e^t)^2\leq e^{4t}\leq e^{4c_0}e^{4t_k},$$
$$\beta_k(\beta_ke^{4t_k})^{-1/3}e^{4t_k}+k^2
\geq\big((\beta_ke^{4t_k})^{2/3}\big)^{3/4}(k^2)^{1/4}=\beta_k^{1/2}e^{2t_k}k^{1/2}\geq e^{-2c_0}\beta_k^{1/2}e^{2t}.$$
 If $t$ is in the support of $\chi_-(\cdot-t_k)$, i.e. $t\leq t_k-c_0/2$, we have
$$e^{4t}g(e^t)^2\leq e^{2t}e^{2t_k}\leq \delta^{-1}e^{2t}\big(1-\sigma(e^{t_k})\big),$$
\begin{align*}
&\beta_k(\beta_ke^{4t_k})^{-1/3}e^{2t}\big(1-\sigma(e^{t_k})\big)+k^2\\
&\geq  \big(\delta\beta_k(\beta_ke^{4t_k})^{-1/3}e^{2t_k}+k^2e^{-2t_k}\big)e^{2t}\\
&\geq C\big(\beta_k^{2/3}e^{2t_k/3}\big)^{3/4}\big(k^{2}e^{-2t_k}\big)^{1/4}e^{2t}=C\beta_k^{1/2}k^{1/2}e^{2t}.
\end{align*}
 If $t$ is in the support of $\chi_+(\cdot-t_k)$, i.e. $t\geq t_k+c_0/2$, we have, by \eqref{sigma.3}
 $$e^{4t}g(e^t)^2\leq 8e^{2t}\big(1-\sigma(e^{t})\big).$$
Suppose $t\geq t_k+c_0/2$, if $e^t\leq2$, we have $1-\sigma(e^t)\geq e^{2t}/16$ and then
\begin{align*}
\beta_k(\beta_ke^{4t_k})^{-1/3}e^{2t}&\big(1-\sigma(e^t)\big)+k^2\geq \big(\beta_k^{2/3}e^{-4t_k/3}\frac{e^{2t}}{16}+k^2e^{-2t}\big)e^{2t}\\
&\geq \big(\frac{1}{16}\beta_k^{2/3}e^{2t/3}\big)^{3/4}\big(k^{2}e^{-2t}\big)^{1/4}e^{2t}=\frac18\beta_k^{1/2}k^{1/2}e^{2t};
\end{align*}
if $e^t\geq2$, we have $1-\sigma(e^t)\geq1-\sigma(2)=e^{-1}$ and then
$$\beta_k(\beta_ke^{4t_k})^{-1/3}e^{2t}\big(1-\sigma(e^t)\big)\geq e^{-1}\beta_k^{2/3}e^{-4t_k/3}e^{2t}\geq e^{-1}\epsilon_1^{-4/3}\beta_k^{2/3}e^{2t}\geq e^{-1}\epsilon_1^{-4/3}\beta_k^{1/2}e^{2t}.$$
This completes the proof of \eqref{eq.rho3.1}, \eqref{eq.rho3.2}.
\end{proof}

\begin{proof}[Proof of Theorem \ref{thm.case3}]
The estimates \eqref{eq.prop.case3} and \eqref{eq.rho3.1} imply that there exists $k_0\geq3$, for all $k\geq k_0$,
\begin{align*}
\notag2{\rm Re}\poscal{\SL_ku}{M_ku}&\geq \frac{c}{2}\beta_k(\beta_ke^{4t_k})^{-1/3}\poscal{\widetilde\rho(t,t_k)u}{u}\\
&\quad-C\normt{D_tu}^2-Ck^2\normt{u}^2-C\normt{e^{2t}u}^2.
\end{align*}
Choosing $C_{12}>0$ large enough, we have, using \eqref{eq.re}
\begin{align}
\label{eq.case3.1}
2{\rm Re}\poscal{\SL_ku}{(C_{12}+M_k)u}&\geq \frac{c}{2}\poscal{\Big(\beta_k(\beta_ke^{4t_k})^{-1/3}\widetilde\rho(t,t_k)+D_t^2+k^2+e^{4t}\Big)u}{u}.
\end{align}
We deduce from \eqref{eq.rho3.2} and \eqref{eq.case3.1} that for $k\geq k_0$,
\begin{align*}
\notag2{\rm Re}\poscal{\SL_ku}{(C_{12}+M_k)u}&\geq C\poscal{\beta_k^{1/2}e^{2t}u}{u}.
\end{align*}
Using that $e^{t}(C_{12}+M_k)e^{-t}$ is bounded on $L^2(\R;dt)$ and Cauchy-Schwarz inequality, we get
$$\normt{e^{-t}\SL_ku}\geq C\beta_k^{1/2}\normt{e^tu},\quad\forall k\geq k_0,$$
completing the proof of Theorem \ref{thm.case3}.
\end{proof}

%%%%%
%%%%%Case 4%%%%%
\subsubsection{Case 4: $e^{t_k}\leq \beta_k^{-1/4}$}\label{sec.case4}
If $e^{t_k}\leq\beta_k^{-1/4}$, we can get estimate by using the multipliers ${\rm Id}$ and $i{\rm Id}$. 
\begin{lem}\label{lem.case4}
Suppose $e^{t_k}\leq\beta_k^{-1/4}$. There exists $C>0$ such that for all $k\geq1$, $\alpha\geq8\pi$, $u\in C_0^\infty(\R)$,
\begin{equation}\label{est.case4}
\normt{e^{-t}\SL_ku}\geq C\beta_k^{1/2}\normt{e^tu},
\end{equation}
where $\SL_k$ is given in \eqref{eq.lk2} and $\beta_k$ is given in \eqref{eq.beta}.
\end{lem}
\begin{proof}
At first note that %by \eqref{pro.sigma2} $e^{t_k}<\epsilon_0$ implies 
$$\forall r>0,\quad 1-\sigma(r)\leq \frac18r^{2}.$$
We have for $e^{t_k}\leq\beta_k^{-1/4}$,
\begin{align}
\notag{\rm Re}\poscal{\SL_ku}{-iu}&=\beta_k\poscal{e^{2t}\big(\sigma(e^{t_k})-\sigma(e^{t})\big)u}{u}+\beta_k\poscal{\gamma\Dk\gamma u}{u}\\
\notag&\geq \beta_k\poscal{e^{2t}\Big(\big(1-\sigma(e^{t})\big)-\big(1-\sigma(e^{t_k})\big)\Big)u}{u}\\
\notag&\geq \beta_k\poscal{e^{2t}\Big(\big(1-\sigma(e^{t})\big)-\frac18\beta_k^{-1/2}\Big)u}{u},
\end{align}
so that with \eqref{eq.re} we get,
\begin{align*}%\label{eq.case4}
{\rm Re}\poscal{\SL_ku}{(1-i)u}\geq \poscal{\Big(\underbrace{k^2e^{-2t}+\beta_k\big(1-\sigma(e^t)\big)}_{\geq e^{-1}\beta_k^{1/2},\ \text{by \eqref{sigma.2}}}-\frac18\beta_k^{1/2}\Big)e^{2t}u}{u},
\end{align*}
%We can find $\alpha_1\geq8\pi$ such that for all $k\geq1, \alpha\geq\alpha_1$,
%\begin{equation}\label{eq.mu.case4}
% k^2e^{-2t}+\beta_k\big(1-\sigma(e^t)\big)\geq \frac14\beta_k^{1/2}.
%\end{equation}
%Indeed, $1-\sigma(e^t)\geq e^{2t}/16$ for $e^t<2$, and $1-\sigma(e^t)\geq e^{-1}$ for $e^t\geq 2$. Then we have for $k\geq1$, $\alpha\geq\alpha_1$,
%\begin{align*}
%&\text{if }e^t<2,\  k^2e^{-2t}+\beta_k\big(1-\sigma(e^t)\big)\geq k^2e^{-2t}+\frac{1}{16}\beta_ke^{2t}\geq \frac{1}{2} k\beta_k^{1/2};\\
%&\text{if }e^t\geq 2,\  k^2e^{-2t}+\beta_k\big(1-\sigma(e^t)\big)\geq e^{-1}\beta_k\geq e^{-1}\sqrt{\frac{\alpha_1k}{8\pi}}\beta_k^{1/2},
%\end{align*}
%so that \eqref{eq.mu.case4} holds with $\alpha_1=\pi e^2/2$.
thus
$${\rm Re}\poscal{\SL_ku}{(1-i)u}\geq \big(\frac{1}{e}-\frac{1}{8}\big)\beta_k^{1/2}\poscal{e^{2t}u}{u}.$$
By Cauchy-Schwarz inequality, we complete the proof of \eqref{est.case4}.
\end{proof}

%%%%%%

%%%%End of the proof%%%%%
\subsection{End of the proof of Theorem \ref{thm.result}}
Summarizing the estimates in Lemma \ref{lem.easy1}, \ref{lem.easy2}, Theorem \ref{thm.case1}, \ref{thm.case2}, \ref{thm.case3} and Lemma \ref{lem.case4}, we have proved the estimate for the operator $\SL_k$ given in \eqref{eq.lk}:
 There exist $C>0$, $k_0\geq3$, $\alpha_0\geq1$ such that for all $|k|\geq k_0$, $\alpha\geq\alpha_0$, $u\in C_0^\infty(\R)$,
\begin{equation}
\|e^{-t}\SL_ku\|_{L^2(\R;dt)}\geq C|\beta_k|^{1/3}\|e^tu\|_{L^2(\R;dt)},
\end{equation} 
and an estimate of the same type for $\widetilde\SL_k$ given in \eqref{eq.tilde.lk} (with different constants $C,\alpha_0$). 
This corresponds to the following estimate for the operator $\h_{\alpha,k,\lambda}=\h_k$ given in \eqref{eq.hk1}, \eqref{eq.hk} for $v\in C_0^\infty((0,+\infty))$, by the equivalence of \eqref{eq.u} and \eqref{eq.v}
\begin{equation}
\|\h_{k,\alpha,\lambda}v\|_{L^2(\R_+;rdr)}\geq C|\beta_k|^{1/3}\|v\|_{L^2(\R_+;rdr)}.
\end{equation}
Then noticing \eqref{eq.h.alpha}, we get for $\omega=\sum_{|k|\geq k_0}\omega_k(r)e^{ik\theta}\in C_0^\infty(\R^2)\cap X_{k_0}$,
\begin{align*}
\normd{(\h_\alpha-i\lambda)\omega}^2&=\sum_{|k|\geq k_0}\|\h_{k,\alpha,\lambda}\omega_k\|_{L^2(\R_+;rdr)}^2\\
&\geq \sum_{|k|\geq k_0}C^2|\beta_k|^{2/3}\|\omega_k\|_{L^2(\R_+;rdr)}^2\\
&=C^2\alpha^{2/3}\normd{|D_\theta|^{1/3}\omega}^2,
\end{align*}
$$\text{since}\quad\normd{|D_\theta|^{1/3}\omega}^2=\sum_{|k|\geq k_0}|k|^{2/3}\|\omega_k\|_{L^2(\R_+;rdr)}^2.$$
Thus \eqref{eq.result1} is proved. 
Since $k_0\geq3$, we know that the imaginary axis does not intersect with the spectrum of $\h_\alpha$ viewed as an operator acting on $X_{k_0}$, which gives \eqref{eq.result2}.
The proof of Theorem \ref{thm.result} is complete.

%%%%%Appendix%%%%%

\section{Appendix}

%%%Weyl calculus%%%%
\subsection{Weyl calculus}\label{sec.app.weyl}
We present some facts about the Weyl calculus, which can be found in \cite[Chapter 18]{Hor3} as well as in \cite[Chapter 2]{Lerner}.
The Weyl quantization associates to a symbol $a$ the operator $a^w$ defined by
\begin{equation}\label{def.weyl}
(a^wu)(x)=\frac{1}{(2\pi)^n}\iint_{\R^n\times\R^n} e^{i(x-y)\cdot\xi}a\bigl(\frac{x+y}{2},\xi\bigr)u(y)dyd\xi.
\end{equation}
Consider the symplectic space $\R^{2n}$ equipped with the symplectic form $\sigma=\sum_{i=1}^nd\xi^i\wedge dx^i$. Given a positive definite quadratic form $\Gamma$ on $\R^{2n}$, we define
$$\Gamma^\sigma(T)=\sup_{\Gamma(Y)=1}\sigma(T,Y)^2,$$
which is also a positive quadratic form. We say that $\Gamma$ is an admissible metric if there exist $C_0$, $\tilde C_0$, $\tilde N_0>0$ such that for all $X,Y\in\R^{2n}$,
\begin{equation}\label{def.adm.metric}
\begin{cases}
\text{uncertainty principle:} & \Gamma_X\leq \Gamma_X^\sigma, \\
\text{slowness: }  &\Gamma_X(X-Y)\leq C_0^{-1} \Longrightarrow(\Gamma_Y/\Gamma_X)^{\pm1}\leq C_0,\\
\text{temperance: } & \Gamma_X\leq \tilde C_0\Gamma_Y\big(1+\Gamma_X^\sigma(X-Y)\big)^{\tilde N_0}.  
\end{cases}
\end{equation}
$C_0,\tilde C_0,\tilde N_0$ in \eqref{def.adm.metric} are called {\it structure constants} of the metric $\Gamma$. 
An admissible weight is a positive function $m$ on the phase space $\R^{2n}$, such that there exist $C_0'$, $\tilde C_0'$, $\tilde N_0'>0$ so that for all $X,Y\in\R^{2n}$,
\begin{equation}\label{def.adm.weight}
\begin{cases}
\text{slowness: }  &\Gamma_X(X-Y)\leq C_0'^{-1} \Longrightarrow(m(Y)/m(X))^{\pm1}\leq C_0',\\
\text{temperance: } &m(X)\leq \tilde C_0m(Y)\big(1+\Gamma_X^\sigma(X-Y)\big)^{\tilde N_0}.  
\end{cases}
\end{equation}
$C_0',\tilde C_0',\tilde N_0'$ in \eqref{def.adm.weight} are called {\it structure constants} of the weight $m$. 
In particular, the function defined by
\begin{equation}\label{def.lambda}
\lambda_\Gamma(X)=\inf_{T\in\R^{2n},T\neq0}\bigl(\Gamma_X^\sigma(T)/\Gamma_X(T)\bigr)^{1/2}
\end{equation}
is an admissible weight for $\Gamma$ and its structure constants depend only on the structure constants of $\Gamma$ (see \cite{D2}). The uncertainty principle is equivalent to $\lambda_\Gamma\geq1$.

We prove the uniform admissibility of a special type of metrics, including those we have used in the proof, given in Definition \ref{def.metric}, \ref{def.case3}.
\begin{lem}\label{lem.metric}
For $\gamma\geq1$,
the metric on the phase space $\R_t\times\R_\tau$ given by
\begin{align*}
\Gamma&=|dt|^2+\frac{|d\tau|^2}{\tau^2+\gamma^2},
\end{align*}
is admissible. Moreover, the structure constants of $\Gamma$ defined in \eqref{def.adm.metric} are bounded above independently of $\gamma$.
\end{lem}
\begin{proof}
%% uncertainty principle
First we notice that
$$\lambda_\Gamma=(\tau^2+\gamma^2)^{1/2}\geq \gamma\geq1,$$ so that $\Gamma$ satisfies the uncertainty principle. 

%% slowly varying
\noindent{\it Slowness.}
 It suffices to prove for $X=(x,\xi)$, $Y=(y,\eta)$, $T=(t,\tau)$, $\Gamma_X(X-Y)\leq s^2$ implies $\Gamma_Y\leq C_0\Gamma_X$. 
Indeed, if $\Gamma_X(X-Y)\leq s^2$ then $|\xi-\eta|^2\leq s^2(\xi^2+\gamma^2)$, and we obtain 
$$\xi^2\leq2(\xi-\eta)^2+2\eta^2\leq2s^2(\xi^2+\gamma^2)+2\eta^2,$$ 
$$\text{thus}\quad (1-2s^2)(\xi^2+\gamma^2)\leq2(\eta^2+\gamma^2).$$ By choosing $0<s<1/\sqrt{2}$ and $C_0=2(1-2s^2)^{-1}>1$, we get 
$$\xi^2+\gamma^2\leq C_0(\eta^2+\gamma^2).$$ 
$$\text{Then}\quad \Gamma_Y(T)=t^2+\frac{\tau^2}{\eta^2+\gamma^2}\leq t^2+\frac{C_0\tau^2}{\xi^2+\gamma^2}\leq C_0\Gamma_X(T).$$

%% temperate
\noindent{\it Temperance.} We have
$$\Gamma_X^\sigma=(\xi^2+\gamma^2)|dt|^2+|d\tau|^2,$$
$$\frac{\Gamma_X(T)}{\Gamma_Y(T)}\leq\max\big(1,\frac{\eta^2+\gamma^2}{\xi^2+\gamma^2}\big).$$
If $|\eta|\leq2|\xi|$ or $|\eta|\leq \gamma$, the right-hand side of the last inequality is bounded from above by 4. If $|\eta|>2|\xi|$ and $|\eta|\geq \gamma$, then $|\xi-\eta|\geq\frac{1}{2}|\eta|$, which implies that $\Gamma_X^\sigma(X-Y)\geq (\xi-\eta)^2\geq \frac{1}{4}\eta^2$; on the other hand, we have
$$\frac{\eta^2+\gamma^2}{\xi^2+\gamma^2}\leq\frac{\eta^2+\gamma^2}{\gamma^{2}}=1+\gamma^{-2}\eta^2,$$
since $\gamma\geq1$, we have 
$$\frac{\Gamma_X(T)}{\Gamma_Y(T)}\leq1+4\Gamma_X^\sigma(X-Y).$$
So the inequality $\Gamma_X(T)/\Gamma_Y(T)\leq4(1+\Gamma_X^\sigma(X-Y))$ holds for any $X,Y,T$.
As a result, we have proved that $\Gamma$ is admissible. 
From the proof above, we see that the structure constants are independent of $\gamma$, and this ends the proof of lemma. 
\end{proof}

The space of symbols $S(m,\Gamma)$ is defined as the set of functions $a\in C^\infty(\R^{2n})$ such that the following semi-norms for all $k\in\N$ 
\begin{equation}\label{def.seminorm}
\sup_{\Gamma_X(T_j)\leq1}\big|a^{(k)}(X)(T_1,\cdots,T_k)\big|m(X)^{-1}<+\infty.
\end{equation}

The composition law $\sharp$ is defined by $a^wb^w=(a\sharp b)^w$ and we have
\begin{equation}
(a\sharp b)(X)=\exp\big(\frac{i}{2}\sigma(D_X,D_Y)\big)a(X)b(Y)_{|Y=X}.
\end{equation}
For $a\in S(m_1,\Gamma)$, $b\in S(m_2,\Gamma)$, we have the asymptotic expansion
\begin{equation}
(a\sharp b)(x,\xi)=\sum_{0\leq k<N}w_k(a,b)+r_N(a,b),
\end{equation}
\begin{align}
\text{with}\quad w_k(a,b)=2^{-k}\sum_{|\alpha|+|\beta|=k}\frac{(-1)^{|\beta|}}{\alpha!\beta!}D_\xi^\alpha\partial_x^\beta a\ D_\xi^\beta\partial_x^\alpha b&\quad\in S(m_1m_2\lambda_\Gamma^{-k},\Gamma),\\
r_N(a,b)(X)=R_N\big(a(X)\otimes b(Y)\big)_{|X=Y}&\quad\in S(m_1m_2\lambda_\Gamma^{-N},\Gamma),
\end{align}
\begin{equation}
R_N=\int_0^1\frac{(1-\theta)^{N-1}}{(N-1)!}\exp\frac{\theta}{2i}[\partial_X,\partial_Y]d\theta\Big(\frac{1}{2i}[\partial_X,\partial_Y]\Big)^N.
\end{equation}
We use here the notation $D=i^{-1}\partial$. The $w_k(a,b)$ with $k$ even are symmetric in $a$, $b$ and skew-symmetric for $k$ odd. In particular, we have
\begin{equation}\label{eq.r}
a\sharp b-b\sharp a=\frac{1}{i}\{a,b\}+\tilde r,\qquad \tilde r\in S(m_1m_2\lambda_\Gamma^{-3},\Gamma),
\end{equation}
where $\{\ ,\ \}$ is the Poisson bracket, implying that
$[a^w,b^w]=\frac{1}{i}\{a,b\}^w+\tilde r^w.$

%Let $a,b$ be real-valued symbols in $S(1,\Gamma)$, then
%\begin{equation}\label{eq.aba}
%a\sharp b\sharp a\equiv a^2b\quad\mod S(\lambda_\Gamma^{-2},\Gamma).
%\end{equation}
%Indeed, we have
%$$a\sharp b\equiv ab+\frac{1}{2i}\{a,b\},\quad\mod S(\lambda_\Gamma^{-2},\Gamma),$$
%where $\lambda_\Gamma$ is given in \eqref{def.lambda}. This implies, modulo $S(\lambda_\Gamma^{-2},\Gamma)$,
%\begin{align*}
%a\sharp b\sharp a&\equiv \bigl(ab+\frac{1}{2i}\{a,b\}\bigr)\sharp a\\
%&\equiv aba+\frac{1}{2i}\{a,b\}a+\frac{1}{2i}\{ab,a\}.
%\end{align*}
%Since $a\sharp b\sharp a$ is real-valued, we obtain \eqref{eq.aba}.

%

%

The symbols in $S(1,\Gamma)$ are quantified in bounded operators on $L^2(\R^n)$, with operator norm depending on the structure constants of $\Gamma$ defined in \eqref{def.adm.metric} and a semi-norm \eqref{def.seminorm} of the symbol in $S(1,\Gamma)$, whose order depends only on the dimension $n$ and the structure constants of $\Gamma$. See \cite{D2}.

\begin{prop}[Fefferman-Phong inequality]\label{prop.fp}
If $a\in S(\lambda_\Gamma^2,\Gamma)$ and $a\geq0$, then $a^w$ is bounded from below by a constant depending on the structure constants of $\Gamma$ given in \eqref{def.adm.metric} and a semi-norm \eqref{def.seminorm} of the symbol $a$ in $S(\lambda_\Gamma^2,\Gamma)$, whose order depends only on the dimension $n$ and the structure constants of $\Gamma$.
\end{prop}

%%%%Appendix B For the operator $(D_t^2+k^2)^{-1}$%%%%%
\subsection{For the operator $(k^2+D_t^2)^{-1}$}
\begin{lem}\label{lem.Dk1}
For $k\geq1$, we have
$$(\tau^2+k^2)^{-1}\in S\big((\tau^2+k^2)^{-1},\frac{|d\tau|^2}{\tau^2+k^2}\big)\subset S(k^{-2},\frac{|d\tau|^2}{\tau^2+k^2},),$$
with semi-norms bounded above independently of $k$.
Moreover, the Fourier multiplier $\Dk=(k^2+D_t^2)^{-1}$ is bounded on $L^2(\R;dt)$ with ${\mathcal L}(L^2(\R;dt))$-norm bounded by $k^{-2}$.
\end{lem}
\begin{proof}
We see that 
$$\frac{1}{\tau^2+k^2}=\frac{k^{-2}}{(k^{-1}\tau)^2+1}.$$
Then for any $m\geq0$, 
$$\Big|\frac{d^m}{d\tau^m}\big(\frac{1}{\tau^2+k^2}\big)\Big|\leq C_mk^{-2}\big((k^{-1}\tau)^2+1\big)^{-1-m/2}k^{-m}= C_m(k^2+\tau^2)^{-1-m/2},$$
where $C_m$ is a positive constant depending only on $m$. This completes the proof of the lemma.
\end{proof}
We can also compute the kernel of the operator $(k^2+D_t^2)^{-1}$.
\begin{lem}\label{lem.Dk2}
For $k\geq1$, we have
$$ \frac{1}{2k}\int_\R e^{-k|t|}e^{i t\tau}dt=\frac{1}{k^2+\tau^2}.$$
As a consequence, $\Dk$ is just the convolution operator with the function $(2k)^{-1}e^{-k|\cdot|}$.
\end{lem}
\begin{proof}
We have
$$
 \frac{1}{2k}\int_\R e^{-k|t|}e^{i t\tau}dt=\frac{1}{k}\int_0^{+\infty}e^{-kt}\cos(t\tau)dt=\frac{1}{k}{\rm Re}\int_0^{+\infty}e^{-t(k+i\tau)}=\frac{1}{k^2+\tau^2}.\qedhere
$$
\end{proof}
As a corollary of the lemma, the following result is used in the proof of Case 3.
\begin{lem}\label{lem.Dk3}
For $k\geq3$, the operator $e^{-2t}\Dk e^{2t}$ is bounded on $L^2(\R;dt)$ with ${\mathcal L}(L^2(\R;dt))$-norm bounded above by $3k^{-2}$.
\end{lem}
\begin{proof}
We deduce from the previous lemma that the operator $e^{-2t}\Dk e^{2t}$ has kernel
$$T_k(t,s)=\frac{1}{2k}e^{-2t}e^{-k|t-s|}e^{2s}=\frac{1}{2k}e^{-k|t-s|-2(t-s)}.$$
We have
$$|T_k(t,s)|\leq \frac{1}{2k}e^{-(k-2)|t-s|}.$$
If $k\geq3$, then the convolution with $\frac{1}{2k}e^{-(k-2)|\cdot|}$ is bounded on $L^2(\R;dt)$ with norm
$$\frac{1}{2k}\|e^{-(k-2)|\cdot|}\|_{L^1(\R;dt)}=\frac{1}{k(k-2)},$$
which is smaller than $3k^{-2}$ since $k\geq3$. This completes the proof of the lemma.
\end{proof}

%%%%%%%%%Appendix C Inequalites %%%%%%%%%
\subsection{Some inequalities}\label{sec.ineq}

We present some inequalities that we have used in the proof. Recall the functions $\sigma,g$ given in \eqref{eq.sigma}
$$\sigma(r)=\frac{1-e^{-r^2/4}}{r^2/4},\quad g(r)=e^{-r^2/8},\quad r>0.$$ 
Firstly, a calculation shows that $$\inf_{\theta>0}\theta^{-2}(e^\theta-1)\simeq1.54414\dots$$ so that
\begin{equation}\label{sigma.1}
\forall r>0,\quad\delta r^2g(r)^2\leq \sigma(r),\quad\text{with } \delta\simeq \frac{1}{4}\times1.54414.
\end{equation}
We verify easily 
\begin{equation}\label{sigma.3}
\forall r>0,\quad r^2g(r)\leq 16\sigma(r),\quad r^2g(r)^2\leq 8\big(1-\sigma(r)\big).
\end{equation}
%{\color{red}16,8 can be improved.}
%
We can get by induction on $n\in\N$ that 
\begin{equation}\label{sigma.d}
\sigma^{(n)}(r)=(-1)^n4r^{-n-2}\big((n+1)!-p_n(r)e^{-r^2/4}\big),
\end{equation}
where $p_n$ is a polynomial of degree $2n$. 
In particular, we have
\begin{equation}\label{sigma'}
\forall r>0,\quad\sigma'(r)=-\frac{8}{r^3}\big(1-e^{-r^2/4}-\frac{r^2}{4}e^{-r^2/4}\big)<0,
\end{equation}
so that $\sigma$ is decreasing.
We have the Taylor expansion of $\sigma$ near 0
\begin{equation}\label{sigma.Taylor}
\sigma(r)=1-\frac{1}{2}\cdot\frac{r^2}{4}+\frac{1}{3!}\Big(\frac{r^2}{4}\Big)^2-\dots+\frac{(-1)^l}{(l+1)!}\Big(\frac{r^2}{4}\Big)^{l}+O(r^{2l+2}),\quad \text{as }r\to0.
\end{equation}

\begin{lem}
For all $k\geq1$, $\alpha\geq8\pi$, $r>0$, we have
\begin{equation}\label{sigma.2}
\begin{cases}
\displaystyle r^2+\beta_k\sigma(r)\geq (2\log2)^{-1}\beta_k^{1/2},\\ 
\displaystyle\frac{k^2}{r^2}+\beta_k\big(1-\sigma(r)\big)\geq e^{-1}\beta_k^{1/2},
\end{cases}
\end{equation}
where $\beta_k=\alpha k/(8\pi)$ is given in \eqref{eq.beta} and $\sigma(r)$ is given in \eqref{eq.sigma}.
\end{lem}
\begin{proof}
By the definition of $\sigma$, we have
\begin{align*}
 \text{if}\ r\leq 2(\log2)^{1/2},&\quad \text{then}\ \sigma(r)\geq \sigma\big(2(\log2)^{1/2}\big)=(2\log2)^{-1},\\ 
\text{if}\ r> 2(\log2)^{1/2},&\quad \text{then}\ e^{-r^2/4}<\frac12,\text{ implying } \sigma(r)> 2r^{-2}.
\end{align*}
Therefore, we get
\begin{align*}
 \text{if}\ r\leq 2(\log2)^{1/2},&\quad \text{then}\ r^2+\beta_k\sigma(r)\geq (2\log2)^{-1}\beta_k,\\ 
\text{if}\ r> 2(\log2)^{1/2},&\quad \text{then}\ r^2+\beta_k\sigma(r)\geq r^2+2\beta_k r^{-2}\geq 2\sqrt2\beta_k^{1/2},
\end{align*}
which implies for any $r\geq0$,
\begin{align*}
r^2+\beta_k\sigma(r)&\geq \min\big((2\log2)^{-1}\beta_k,2\sqrt2\beta_k^{1/2}\big)\geq(2\log2)^{-1}\beta_k^{1/2},\quad \text{if }\alpha\geq8\pi,
\end{align*}
proving the first inequality in \eqref{sigma.2}.
Noting that 
$$\forall\  0\leq\theta\leq1,\quad e^{-\theta}-1+\theta\geq \frac{\theta^2}{4},$$
we have 
$$\forall 0<r\leq2,\quad1-\sigma(r)=\frac{4}{r^2}\big(e^{-r^2/4}-1+\frac{r^2}{4}\big)\geq \frac{r^2}{16}.$$
Then we obtain
\begin{align*}
\text{if}\ r\leq 2,&\quad \frac{k^2}{r^2}+\beta_k\big(1-\sigma(r)\big)\geq\frac{k^2}{r^2}+\beta_k\frac{r^2}{16}\geq \frac{1}{2}k\beta_k^{1/2},\\
\text{if}\ r> 2,&\quad \frac{k^2}{r^2}+\beta_k\big(1-\sigma(r)\big)\geq\beta_k\big(1-\sigma(2)\big)=e^{-1}\beta_k.
\end{align*}
Therefore for any $r\geq0$, $k\geq1$, 
\begin{align*}
\frac{k^2}{r^2}+\beta_k\big(1-\sigma(r)\big)&\geq\min\big(\frac{1}{2}k\beta_k^{1/2},e^{-1}\beta_k\big)\geq e^{-1}\beta_k^{1/2},\quad\text{if }\alpha\geq8\pi,
\end{align*}
which proves the second inequality in \eqref{sigma.2}.
\end{proof}

%%%%%%%%%%%%%%%%%%%%%%%%%%%%%

%
%More generally, for any $0<\delta_0<1$, there exist $R_1>0$, $0<c_1<1$ such that
%\begin{equation}\label{pro.sigma1}
%\begin{cases}\displaystyle
%4\delta_0r^{-2}\leq \sigma(r)\leq 4r^{-2},&
%\text{for }
%r>R_1,\\
%\displaystyle\sigma(r)\geq c_1,& \text{for } 0\leq r\leq R_1,
%\end{cases}
%\end{equation}
%%since $\sigma(r)$ behaves like $4r^{-2}$ as $r\to+\infty$. 
%%
%and there exist $R_2>0$, $0<c_2<1$ such that
%\begin{equation}\label{pro.sigma2}
%\begin{cases}
%-\frac{1}{8}r^2\leq\sigma(r)-1\leq-\frac{1}{8}\delta_0r^2,&\text{for }0\leq r\leq R_2,\\
%\sigma(r)\leq c_2, &\text{for }r> R_2.
%\end{cases}
%\end{equation}

%\bigskip

%%%%%%%%%%%%%%%%%%%%%%%%%%%%%%%%%%%
Now we prove inequalities about the function $\sigma$ that are used in the proof of Case 1, 2 and 3.
\begin{prop}\label{prop.sigma}
Given $\epsilon_0,\epsilon_1\in(0,1)$, there exist $c_0\in(0,1)$, $C_1,C_2,C_3,\tilde C_n>0$, $c_1,c_2,c_3\in(0,1)$ satisfying the following. Recall that $\sigma$ is given in \eqref{eq.sigma}.
\begin{enumerate}
\item For $e^{t_k}>\epsilon_0^{-1}$, we have
\begin{equation}\label{eq.case1sigma}
\forall |t-t_k|\leq 2c_0,\quad\frac{d}{dt}\Big[e^{2t}\big(\sigma(e^t)-\sigma(e^{t_k})\big)\Big]\leq -C_1,
\end{equation}
\begin{equation}\label{eq.case1sigma.1}
\displaystyle\forall |t-t_k|\leq3c_0,\ \forall n\in\N,\quad\Big|\frac{d^n}{dt^n}\Big[e^{2t}\big(\sigma(e^t)-\sigma(e^{t_k})\big)\Big]\Big|\leq \tilde C_n,
\end{equation}
\begin{equation}\label{eq.case1sigma.2}
\text{and}\qquad
\begin{cases}
\displaystyle\sigma(e^t)-\sigma(e^{t_k})\leq-c_1\sigma(e^{t_k}),\quad &\text{for }t-t_k\geq c_0/2,\\
\displaystyle\sigma(e^t)-\sigma(e^{t_k})\geq c_1\sigma(e^{t}),\quad &\text{for }t-t_k\leq -c_0/2.
\end{cases}
\end{equation} 
\item For $\epsilon_1\leq e^{t_k}\leq\epsilon_0^{-1}$, we have
\begin{equation}\label{eq.case2sigma}
 \forall |t-t_k|\leq 2c_0,\quad \frac{d}{dt}\Big[e^{2t}\big(\sigma(e^t)-\sigma(e^{t_k})\big)\Big]\leq -C_2 ,
\end{equation}
\begin{equation}\label{eq.case2sigma.1}
\forall |t-t_k|\leq3c_0,\ \forall n\in\N,\quad\Big|\frac{d^n}{dt^n}\Big[e^{2t}\big(\sigma(e^t)-\sigma(e^{t_k})\big)\Big]\Big|\leq\tilde C_n,
\end{equation}
\begin{equation}\label{eq.case2sigma.2}
\text{and}\qquad
\begin{cases}
\displaystyle\sigma(e^t)-\sigma(e^{t_k})\leq -c_2\sigma(e^{t_k}),&\text{for }t-t_k\geq c_0/2,\\
\displaystyle\sigma(e^t)-\sigma(e^{t_k})\geq c_2\sigma(e^t),&\text{for }t-t_k\leq-c_0/2.
\end{cases}
\end{equation}
\item For $ e^{t_k}<\epsilon_1$, we have
\begin{equation}\label{eq.case3sigma}
\forall |t-t_k|\leq 2c_0,\quad\frac{d}{dt}\Big[e^{2t}\big(\sigma(e^t)-\sigma(e^{t_k})\big)\Big]\leq -C_3 e^{4t_k}, 
\end{equation}
\begin{equation}\label{eq.case3sigma.1}
\forall |t-t_k|\leq3c_0,\ \forall n\in\N,\quad\Big|\frac{d^n}{dt^n}\Big[e^{2t}\big(\sigma(e^t)-\sigma(e^{t_k})\big)\Big]\Big|\leq\tilde C_ne^{4t_k},
\end{equation}
\begin{equation}\label{eq.case3sigma.2}
\text{and}\qquad
\begin{cases}
\displaystyle\sigma(e^t)-\sigma(e^{t_k})\leq -c_3\big(1-\sigma(e^{t})\big),&\text{for }t-t_k\geq c_0/2,\\
\displaystyle\sigma(e^t)-\sigma(e^{t_k})\geq c_3\big(1-\sigma(e^{t_k})\big),&\text{for }t-t_k\leq-c_0/2.
\end{cases}
\end{equation}
\end{enumerate}
\end{prop}

\begin{proof}
{\it Step 1.}
The essential step is to choose $c_0\in(0,1)$ such that \eqref{eq.case1sigma}, \eqref{eq.case2sigma} and \eqref{eq.case3sigma} hold.

By \eqref{sigma'}, given $\epsilon_0,\epsilon_1\in(0,1)$, there exist $\mu_1>\mu_2>0$ such that 
\begin{align}
\notag&\forall r>e^{-2}\epsilon_0^{-1}, \quad\qquad-\mu_1r^{-3}\leq\sigma'(r)\leq -\mu_2r^{-3},\\
\notag&\forall r\in[e^{-2}\epsilon_1,e^2\epsilon_0^{-1}],\quad-\mu_1\leq\sigma'(r)\leq -\mu_2,\\
\label{eq.C2}&\forall r<e^{2}\epsilon_1,\ \qquad\qquad-\mu_1r\leq\sigma'(r)\leq-\mu_2r.
\end{align}
Let us denote
$$f(t,t_k):=\frac{d}{dt}\Big[e^{2t}\big(\sigma(e^t)-\sigma(e^{t_k})\big)\Big]=e^{3t}\sigma'(e^t)+2e^{2t}\big(\sigma(e^t)-\sigma(e^{t_k})\big).$$
The Taylor's formula gives
$$\sigma(e^t)-\sigma(e^{t_k})=\int_0^1 \sigma'(e^{ t_k+\theta(t-t_k)})e^{ t_k+\theta(t-t_k)}(t-t_k)d\theta.$$
Suppose $|t-t_k|\leq 2c_0$ with $c_0<1$. Using \eqref{eq.C2}, we get the following
\begin{itemize}
\item if $e^{t_k}>\epsilon_0^{-1}$, then $e^{t}>e^{-2}\epsilon_0^{-1}$ and
$$f(t,t_k)\leq-\mu_2+4\mu_1e^{4c_0}c_0;$$
\item if $e^{t_k}\in[\epsilon_1,\epsilon_0^{-1}]$, then $e^t\in[e^{-2}\epsilon_1,e^2\epsilon_0^{-1}]$ and 
$$ f(t,t_k)\leq-\mu_2e^{3t}+4\mu_1e^{3t}e^{2c_0}c_0=-(\mu_2-4\mu_1e^{2c_0}c_0)e^{3t}; $$
\item if $e^{t_k}<\epsilon_1$, then $e^t<e^{2}\epsilon_1$ and
$$f(t,t_k)\leq-\mu_2e^{4t}+4\mu_1e^{4t}e^{4c_0} c_0=-(\mu_2-4\mu_1e^{4c_0}c_0)e^{4t}.$$
\end{itemize}
Let $c_0\in(0,1)$ satisfying
$$4c_0e^{4c_0}\leq\frac{\mu_2}{2\mu_1},$$
then we get \eqref{eq.case1sigma}, \eqref{eq.case2sigma}, \eqref{eq.case3sigma} with $$C_1=\mu_2/2,\quad C_2=\mu_2e^{-6}\epsilon_1^3/2\quad\text{and}\quad C_3=\mu_2e^{-8c_0}/2.$$

{\it Step 2.} The inequalities \eqref{eq.case1sigma.1}, \eqref{eq.case2sigma.1} and \eqref{eq.case3sigma.1} are consequences of \eqref{sigma.d}, noticing that for $|t-t_k|\leq3c_0$,
$$\big|\sigma(e^t)-\sigma(e^{t_k})\big|\leq
\begin{cases}
Ce^{-2t_k},\quad &\text{if }e^{t_k}>\epsilon_0^{-1},\\
C,\quad &\text{if }e^{t_k}\in[\epsilon_1,\epsilon_0^{-1}],\\
Ce^{2t_k},\quad &\text{if }e^{t_k}<\epsilon_1.
\end{cases}$$

\medskip

{\it Step 3.} It remains to prove \eqref{eq.case1sigma.2}, \eqref{eq.case2sigma.2} and \eqref{eq.case3sigma.2}. 
Denoting $r_k=e^{t_k}$ and $r=e^t$, then \eqref{eq.case1sigma.2}, \eqref{eq.case2sigma.2} are equivalent to the following 
\begin{equation}\label{eq.case1sigma.2'}
\forall r_k>\epsilon_0^{-1},\qquad
\begin{cases}
\displaystyle\sigma(r)\leq(1-c_1)\sigma(r_k),\quad &\text{for }r\geq r_ke^{c_0/2},\\
\displaystyle\sigma(r_k)\leq(1-c_1)\sigma(r),\quad &\text{for }r\leq r_ke^{-c_0/2}.
\end{cases}
\end{equation}
\begin{equation}\label{eq.case2sigma.2'}
\forall r_k\in[\epsilon_1,\epsilon_0^{-1}],\qquad
\begin{cases}
\displaystyle\sigma(r)\leq(1-c_2)\sigma(r_k),\quad &\text{for }r\geq r_ke^{c_0/2},\\
\displaystyle\sigma(r_k)\leq(1-c_2)\sigma(r),\quad &\text{for }r\leq r_ke^{-c_0/2}.
\end{cases}
\end{equation}
The function $\sigma$ is decreasing, so that in order to prove \eqref{eq.case1sigma.2'} and \eqref{eq.case2sigma.2'}, it suffices to prove the following
\begin{equation}\label{eq.C3}
\exists c\in(0,1),\ \forall r\geq\epsilon_1e^{c_0/2},\quad \sigma(re^{c_0/2})\leq(1-c)\sigma(r).
\end{equation}
We know that for any $\lambda>1$, the function
$$[0,1]\ni\theta\mapsto f_1(\theta;\lambda)=\frac{1-\theta^\lambda}{\lambda(1-\theta)}$$
is strictly increasing in $(0,1)$ and $f_1(1;\lambda)=1$. Hence for all $\lambda>1$, there exists $0<\delta_1(\lambda)<1$ such that 
$$\forall 0<\theta\leq \exp(-\epsilon_1^2e^{c_0}/4),\quad f_1(\theta;\lambda)\leq \delta_1(\lambda).$$
Then we have for $r\geq\epsilon_1e^{c_0/2}$, 
$$\frac{\sigma(re^{c_0/2})}{\sigma(r)}=\frac{1-(e^{-r^2/4})^{e^{c_0}}}{e^{c_0}(1-e^{-r^2/4})}=f_1(e^{-r^2/4};e^{c_0})\leq \delta_1(e^{c_0}),\quad \text{\small since }e^{c_0}>1,$$
which proves \eqref{eq.C3} with $c=1-\delta_1(e^{c_0})$. Thus \eqref{eq.case1sigma.2} and \eqref{eq.case2sigma.2} are proved.

Now we turn to prove \eqref{eq.case3sigma.2}, which is equivalent to the following
\begin{equation}
\forall r_k<\epsilon_1,\qquad
\begin{cases}
\displaystyle\sigma(r)-\sigma(r_k)\leq -c_3\big(1-\sigma(r)\big),&\text{for }r\geq r_ke^{c_0/2},\\
\displaystyle\sigma(r)-\sigma(r_k)\geq c_3\big(1-\sigma(r_k)\big),&\text{for }r\leq r_ke^{-c_0/2}.
\end{cases}
\end{equation}
Since $1-\sigma(r)$ is increasing, we need only to prove
\begin{equation}\label{eq.C4}
\exists c_3\in(0,1),\ \forall r<\epsilon_1,\quad 1-\sigma(r)\leq(1-c_3)\big(1-\sigma(re^{c_0/2})\big).
\end{equation}
By direct computation, we find that for any $\lambda>1$, the function
$$(0,+\infty)\ni x\mapsto f_2(x;\lambda)=\frac{\lambda(e^{-x}-1+x)}{e^{-\lambda x}-1+\lambda x}$$
is continuous on $[0,+\infty)$, $f_2(0;\lambda)=\lambda^{-1}$ and $f_2(x;\lambda)<1$ for all $x>0$.
Hence for any $\lambda>1$, there exists $0<\delta_2(\lambda)<1$ such that 
$$f_2(x;\lambda)<\delta_2(\lambda),\quad \forall 0<x<1/4.$$
We get for all $r<\epsilon_1<1$,
$$\frac{1-\sigma(r)}{1-\sigma(re^{c_0/2})}=e^{c_0}\frac{e^{-r^2/4}-1+r^2/4}{e^{-e^{c_0} r^2/4}-1+e^{c_0}r^2/4}=f_2(r^2/4;e^{c_0})<\delta_2(e^{c_0}),$$
which proves \eqref{eq.C4} with $c_3=1-\delta_2(e^{c_0})$ thus \eqref{eq.case3sigma.2} is proved. 
The proof of Proposition \ref{prop.sigma} is now complete.
\end{proof}

\bibliography{oseenvortices}%%%%%%%%%%%%%%%%%%%%%%%%%%%

\providecommand{\bysame}{\leavevmode\hbox to3em{\hrulefill}\thinspace}
\providecommand{\MR}{\relax\ifhmode\unskip\space\fi MR }
\providecommand{\MRhref}[2]{%
  \href{http://www.ams.org/mathscinet-getitem?mr=#1}{#2}
}
\providecommand{\href}[2]{#2}
\begin{thebibliography}{10}

\bibitem{BenArtzi}
Matania Ben-Artzi, \emph{Global solutions of two-dimensional {N}avier-{S}tokes
  and {E}uler equations}, Arch. Rational Mech. Anal. \textbf{128} (1994),
  no.~4, 329--358. \MR{1308857 (96h:35148)}

\bibitem{Bony}
Jean-Michel Bony, \emph{Sur l'in\'egalit\'e de {F}efferman-{P}hong}, Seminaire:
  \'{E}quations aux {D}\'eriv\'ees {P}artielles, 1998--1999, S\'emin. \'Equ.
  D\'eriv. Partielles, \'Ecole Polytech., Palaiseau, 1999, pp.~Exp. No. III,
  16. \MR{1721321 (2000i:35232)}

\bibitem{DSZ}
N.~Dencker, J.~Sj{\"o}strand, and M.~Zworski, \emph{Pseudospectra of
  semiclassical (pseudo-) differential operators}, Comm. Pure Appl. Math.
  \textbf{57} (2004), no.~3, 384--415. \MR{2020109 (2004k:35432)}

\bibitem{D3}
Wen Deng, \emph{Numerical computations for the value of $k_0$},
  \url{http://www.math.jussieu.fr/~wendeng/}.

\bibitem{D}
\bysame, \emph{Resolvent estimates for a two-dimensional non-self-adjoint
  operator}, preprint, 2010.

\bibitem{D2}
\bysame, \emph{Structure constants of the {W}eyl calculus}, preprint, 2010.

\bibitem{GGN}
Isabelle Gallagher, Thierry Gallay, and Francis Nier, \emph{Spectral
  asymptotics for large skew-symmetric perturbations of the harmonic
  oscillator}, Int. Math. Res. Not. IMRN (2009), no.~12, 2147--2199.
  \MR{2511908 (2010e:34198)}

\bibitem{Gallay}
Thierry Gallay, \emph{Nonselfadjoint operators in fluid mechanics: a case
  study}, lecture notes for the summer school ``Spectral analysis of
  non-selfadjoint operators and applications", Rennes, 2011.

\bibitem{GW2}
Thierry Gallay and C.~Eugene Wayne, \emph{Invariant manifolds and the long-time
  asymptotics of the {N}avier-{S}tokes and vorticity equations on {$\bold
  R^2$}}, Arch. Ration. Mech. Anal. \textbf{163} (2002), no.~3, 209--258.
  \MR{1912106 (2003c:37123)}

\bibitem{GW1}
\bysame, \emph{Global stability of vortex solutions of the two-dimensional
  {N}avier-{S}tokes equation}, Comm. Math. Phys. \textbf{255} (2005), no.~1,
  97--129. \MR{2123378 (2005m:35224)}

\bibitem{Hor1}
Lars H{\"o}rmander, \emph{The analysis of linear partial differential
  operators. {I}}, Grundlehren der Mathematischen Wissenschaften [Fundamental
  Principles of Mathematical Sciences], vol. 256, Springer-Verlag, Berlin,
  1983, Distribution theory and Fourier analysis. \MR{717035 (85g:35002a)}

\bibitem{Hor3}
\bysame, \emph{The analysis of linear partial differential operators. {III}},
  Grundlehren der Mathematischen Wissenschaften [Fundamental Principles of
  Mathematical Sciences], vol. 274, Springer-Verlag, Berlin, 1985,
  Pseudodifferential operators. \MR{781536 (87d:35002a)}

\bibitem{Kato2}
Tosio Kato, \emph{The {N}avier-{S}tokes equation for an incompressible fluid in
  {${\bf R}^2$} with a measure as the initial vorticity}, Differential Integral
  Equations \textbf{7} (1994), no.~3-4, 949--966. \MR{1270113 (95b:35173)}

\bibitem{Kato}
\bysame, \emph{Perturbation theory for linear operators}, Classics in
  Mathematics, Springer-Verlag, Berlin, 1995, Reprint of the 1980 edition.
  \MR{1335452 (96a:47025)}

\bibitem{Lerner}
Nicolas Lerner, \emph{Metrics on the phase space and non-selfadjoint
  pseudo-differential operators}, Pseudo-Differential Operators. Theory and
  Applications, vol.~3, Birkh\"auser Verlag, Basel, 2010. \MR{2599384}

\bibitem{Maekawa}
Yasunori Maekawa, \emph{Spectral properties of the linearization at the burgers
  vortex in the high rotation limit}, to appear in J. Math. Fluid Mech.

\bibitem{Pravda}
Karel Pravda-Starov, \emph{A general result about the pseudo-spectrum of
  {S}chr\"odinger operators}, Proc. R. Soc. Lond. Ser. A Math. Phys. Eng. Sci.
  \textbf{460} (2004), no.~2042, 471--477. \MR{2034648 (2004j:47094)}

\bibitem{ProPullin}
A.~Prochazka and D.~I. Pullin, \emph{On the two-dimensional stability of the
  axisymmetric {B}urgers vortex}, Phys. Fluids \textbf{7} (1995), no.~7,
  1788--1790. \MR{1336103 (96b:76053)}

\bibitem{Tref}
Lloyd~N. Trefethen, \emph{Pseudospectra of linear operators}, SIAM Rev.
  \textbf{39} (1997), no.~3, 383--406. \MR{1469941 (98i:47004)}

\bibitem{MR2155029}
Lloyd~N. Trefethen and Mark Embree, \emph{Spectra and pseudospectra}, Princeton
  University Press, Princeton, NJ, 2005, The behavior of nonnormal matrices and
  operators. \MR{2155029 (2006d:15001)}

\end{thebibliography}
\nocite{*}
\bibliographystyle{amsplain}

%%%%%
%%%%%

\end{document}